\documentclass[reqno]{amsart}
\usepackage{enumerate}
\usepackage{amsthm,amsmath,amssymb,url,color, booktabs,nccmath}
\usepackage[left=3.2cm,right=3.2cm,top=4cm,bottom=4cm]{geometry}
\usepackage{mathrsfs}
\usepackage{enumitem}
\usepackage[graphicx]{realboxes}
\usepackage{tikz}

\usepackage{pinlabel}
\usepackage{graphicx}	
\usepackage{subcaption}
\usepackage{tikz}
\usepackage{tikz-cd}
\usepackage{color}
\usepackage[colorlinks=true,hypertexnames=false]{hyperref}
\hypersetup{linkcolor=blue,citecolor=red,filecolor=black,urlcolor=black}

\definecolor{darkergreen}{rgb}{0.0, 0.5, 0.0}

\numberwithin{equation}{section}

\newcommand{\be}{\begin{eqnarray}}
\newcommand{\ee}{\end{eqnarray}}
\newcommand{\ce}{\begin{eqnarray*}}
\newcommand{\de}{\end{eqnarray*}}
\newtheorem{theorem}{Theorem}[section]
\newtheorem{lemma}[theorem]{Lemma}
\newtheorem{proposition}[theorem]{Proposition}
\newtheorem{Examples}[theorem]{Example}
\newtheorem{corollary}[theorem]{Corollary}

\newtheorem{definition}[theorem]{Definition}
\theoremstyle{definition}
\newtheorem{remark}[theorem]{Remark}

\DeclareMathOperator{\supp}{supp}

\def\Im{{\mathrm{Im}}}

\def\<{{\langle}}
\def\>{{\rangle}}

\def\bx{{\mathbf{x}}}

\def\sgn{\mbox{\rm sgn}}

\def\dif{{\mathord{{\rm d}}}}

\def\min{{\mathord{{\rm min}}}}

\def\={&\!\!=\!\!&}

\def\1{{\mathbf{1}}}

\def\geq{\geqslant}
\def\leq{\leqslant}
\def\ge{\geqslant}
\def\le{\leqslant}

\def\Im{{\mathrm{Im}}}

\def\<{{\langle}}
\def\>{{\rangle}}

\def\bx{{\mathbf{x}}}

\def\sgn{\mbox{\rm sgn}}

\def\dif{{\mathord{{\rm d}}}}

\def\min{{\mathord{{\rm min}}}}

\def\={&\!\!=\!\!&}
\def\bt{\begin{theorem}}
\def\et{\end{theorem}}
\def\bl{\begin{lemma}}
\def\el{\end{lemma}}
\def\br{\begin{remark}}
\def\er{\end{remark}}
\def\bx{\begin{Examples}}
\def\ex{\end{Examples}}
\def\bd{\begin{definition}}
\def\ed{\end{definition}}
\def\bp{\begin{proposition}}
\def\ep{\end{proposition}}
\def\bc{\begin{corollary}}
\def\ec{\end{corollary}}

\def\geq{\geqslant}
\def\leq{\leqslant}
\def\ge{\geqslant}
\def\le{\leqslant}

 \def\R{\mathbb R}
 \def\R{\mathbb R}

\def\<{\langle} \def\>{\rangle}

\allowdisplaybreaks

\begin{document}

\title[The Wave  Kinetic Theory for Quasi MMT]{The Wave  Kinetic Theory for Quasilinear MMT Equation}

\author{Huaxiang L\"u}
\address[H. L\"u]{Academy of Mathematics and Systems Science, Chinese Academy of Sciences, Beijing 100190, China, and Fakult\"at f\"ur Mathematik, Bielefeld Universit\"at, D 33615 Bielefeld, Germany}\email{lvhuaxiang22@mails.ucas.ac.cn }

\thanks{
Funded by the Deutsche Forschungsgemeinschaft (DFG,German Research Foundation)-Project-ID 317210226-SFB 1283"
}

\begin{abstract}

We study the one-dimensional quasilinear Majda--McLaughlin--Tabak (MMT) equation on a large torus $[0,L]$:
\begin{align*}
    i \partial_t u +2\pi|\nabla|^\sigma u +\lambda^{2}|\nabla|^\beta\left[ \left||\nabla|^\beta u\right|^{2} |\nabla|^\beta u\right]=0,
\end{align*}
with $\sigma\in(0,1)\cup (1,2]$ and $0<\beta<\sigma/4$. We consider Gaussian random initial data and work in the joint
limit $L\to\infty$ with
$\alpha=\lambda^2L^{-1}\to0$.   The quasilinear structure leads to unavoidable derivative loss, which prevents the direct construction of solutions via iteration of the Duhamel formula.    For $\sigma\in(1,2]$, we prove that, with high probability, solutions exist up to time scales ${T} \sim\alpha^{-{\frac43}+} \wedge \alpha^{-\frac1{\sigma/2-\beta}+}$, and that only trivial four-wave resonances occur, leading to a degenerate wave kinetic equation. For $\sigma\in(0,1)$, we prove the existence up to time scales $T\sim \alpha^{-\frac43+}$, and show that the second-order statistics are well approximated by the wave kinetic equation. In both cases, the solutions remain smooth while exhibiting smallness in suitable 
$L^\infty$-based norms.

The proof combines propagation of randomness for a truncated equation with deterministic energy estimates and a bootstrap argument. The main technical difficulties arise in the regime $\sigma\in(0,1)$. At low frequencies, the counting argument requires simultaneous
control of the frequency gap and the associated resonant phase gap.  In the energy estimate, nontrivial resonances generate quartic terms that prevent the  energy estimates from closing the bootstrap on $[0,T]$. To overcome this obstruction, we develop a new  normal-form $L^2$ energy estimates, using distinct frequency truncations for the initial data and the nonlinearity.  
\end{abstract}


\date{\today}

\maketitle

\tableofcontents

 \section{Introduction}

The kinetic theory of nonlinear wave systems serves as the formal foundation for the nonequilibrium statistical physics of such systems, and can be viewed as a natural extension of Boltzmann’s kinetic framework, originally introduced in the setting of interacting particle systems. The central objective of this theory is to extract macroscopic statistical dynamics from microscopic wave-wave interactions. To describe the effective energy spectrum, physicists formulated the Wave Kinetic Equation (WKE), whose role closely parallels that of the Boltzmann equation in classical particle kinetic theory, acting as a canonical bridge between microscopic interaction mechanisms and macroscopic kinetic-scale statistical descriptions. Then, Zakharov and his collaborators discovered a close connection to  hydrodynamic
 turbulence, whence the names “wave turbulence theory” and "Kolmogorov-Zakharov spectra". 

From a mathematical perspective, a rigorous derivation of the wave kinetic equation starting from dispersive equations constitutes a fundamental problem, and can be viewed as a  generalized form of Hilbert’s Sixth Problem for  statistical theory of wave systems.
Such a rigorous mathematical derivation has been developed in recent years. Through the precise mathematical analysis of Feynman diagram expansions, major breakthroughs have been achieved in the rigorous derivation of the WKE for NLS, or other semilinear models. We refer to Section \ref{sec:liter} for further relevant literature. However, for a large class of quasilinear equations, the situation is totally different, as solutions cannot be constructed through iterative applications of the Duhamel formula due to the unavoidable derivative loss.  In particular, the solution cannot be represented through iterative applications of  Feynman trees, which constitutes the core analytical strategy for semilinear equations.  This difficulty was recently overcome in the quasilinear setting by
Deng, Ionescu, and Pusateri \cite{DIP25a,DIP25} for the two-dimensional
gravity water-wave system, whose dispersion relation
corresponds to the case $\sigma=\frac12$.

Among the quasilinear models, a notable example is the toy model introduced by Majda, McLaughlin, and Tabak to conduct  numerical investigations of wave turbulence \cite{MMT97}. This equation is analytically tractable and straightforward to simulate, yet it still displays an exceptionally rich and diverse range of phenomena.
The universality of the weak turbulence scenario was called into question after studies of the MMT model revealed spectral behaviors that, in some cases, deviate from the classical Kolmogorov-Zakharov spectra predicted by wave turbulence theory. The introduction of the MMT model thus marked an important step toward a deeper understanding of wave turbulence. Nevertheless, a rigorous derivation of the quasilinear MMT equation remains an open problem.  

In this paper, we study the wave kinetic theory of the quasilinear MMT equation, which reads as:
  \begin{align*}
 \tag{MMT} \label{MMT}
\begin{cases}
i \partial_t u +2\pi|\nabla|^\sigma u \pm \lambda^{2}|\nabla|^\beta\left[ \left||\nabla|^\beta u\right|^{2} |\nabla|^\beta u\right]=0,  \quad x\in \mathbb{T}_L := [0,L],   \\
u(0,x) = u_{\textrm{in}}(x).
\end{cases}
\end{align*}
Here the sign of the nonlinearity has no effect on the kinetic description and we consider the defocusing case in this paper.
  The parameter $\sigma\in(0,1)\cup (1,2]$ controls the dispersion.  The parameter $\beta\in(0,1)$ quantifies the "degree of quasilinearity" of the equation.  The case $\sigma=2,\beta=0$ corresponds to the cubic  Schr\"odinger equation (NLS).
 We rescale   the operators     $|\nabla|^\sigma$ and $|\nabla|^\beta$   into
\begin{align*} 2\pi|\nabla|^\sigma e^{ 2\pi ikx} =2\pi |k|^\sigma  e^{2\pi ik x},\ \ \
 |\nabla|^\beta e^{2\pi ikx} = |k|^\beta  e^{2\pi ik x}.
\end{align*} 

     The corresponding  wave kinetic equation reads as:
\begin{align*}
\tag{WKE}\label{WKE}
\begin{split}
\partial_t n(t, \xi) =&\mathcal K\left(n(t,\cdot)\right)(\xi), \\
\mathcal K(\phi)(\xi):=& \int_{\substack{(\xi_1, \xi_2, \xi_3)\in \R^{3}\\\xi_1-\xi_2+\xi_3=\xi,\\ |\xi_1|^\sigma-|\xi_2|^\sigma+|\xi_3|^\sigma=|\xi|^\sigma}}|\xi|^{2\beta}|\xi_1|^{2\beta}|\xi_2|^{2\beta}|\xi_3|^{2\beta} \phi \phi_1 \phi_2 \phi_3 \left(\frac{1}{\phi}-\frac{1}{\phi_1}+\frac{1}{\phi_2}-\frac{1}{\phi_3}\right) \, \dif \xi_1 \dif \xi_2 \dif \xi_3,
\end{split}
\end{align*}
where we used the shorthand notations $\phi:=\phi(\xi),$ and $\phi_j:=\phi(\xi_j)$ for $i=1,2,3$. We note that \eqref{WKE} does not depend upon the sign of nonlinearity in \eqref{MMT}. In particular, as discussed in \cite{MMT97}, in the case $\sigma\in(1,2]$, the collision operator on the right hand of \eqref{WKE} is trivial, since the   dispersion relation $$\xi_1-\xi_2+\xi_3=\xi,\ \ |\xi_1|^\sigma-|\xi_2|^\sigma+|\xi_3|^\sigma=|\xi|^\sigma$$ only admits trivial solution  $\{\xi,\xi_2\}=\{\xi_1,\xi_3\}$.  In any case, the  main part of the integrand equals to 0.  The question now becomes what
 is the appropriate kinetic theory. However, in the case $\sigma \in (0,1)$, the resonant four-wave interactions admit additional nontrivial solutions. For instance, for the water-wave dispersion law $\sigma = \tfrac{1}{2}$ as in \cite{DIP25a,DIP25}, there exist the so-called Benjamin–Feir resonances (taking symmetry into account):
\begin{align*}
( \xi,\xi_1,\xi_2,\xi_3)=  \left( \frac{(a+b+\sqrt{ab})^2}{(\sqrt{a}+\sqrt{b})^2},\ a,\  -\frac{ab}{(\sqrt{a}+\sqrt{b})^2},\ b\right).
\end{align*}
However, for a general dispersive relation, the resonance structure usually does not admit an explicit characterization.

The kinetic  description  is  akin to a law of large numbers,  therefore
we start with a random distribution of the initial data. The conjecture   states that  in the limit of large torus size $L$ and vanishing strength of the nonlinearity 
\begin{align*}
  \alpha:=\lambda^2L^{-1},  
\end{align*}
the effective dynamics of the Fourier-space mass density $\mathbf E |\widehat u(t, k)|^2$ ($k \in \mathbb{Z}_L:=L^{-1}\mathbb{Z}$) is well approximated  by
\begin{equation*}
\mathbf E |\widehat u(t, k)|^2 \approx n\left(\frac{t}{T_{\mathrm{kin}}}, k\right),
\end{equation*}
 where $T_{\mathrm{kin}}\sim \alpha^{-2}=\frac{L^{2}}{\lambda^4}$ is the  kinetic timescale. While the discussion and derivation under the quasi-Gaussian random phase hypothesis are established in \cite{MMT97}, a rigorous derivation remains an open problem.

\subsection{Statement of the results}  

The spatial Fourier series of a function $u: \mathbb{T}_L \to \mathbb C$ is defined on $\mathbb{Z}_L:=L^{-1}\mathbb{Z}$ by
\begin{equation*}
\widehat{u}(k)=\int_{\mathbb{T}_L} u(x) e^{-2\pi i kx}\dif x,\quad \mathrm{\; such \ that \;}\quad u(x)=\frac{1}{L}\sum_{k \in \mathbb{Z}_L} \widehat{u}(k) \,e^{2\pi i kx}. 
\end{equation*}
 To match the kinetic approximation at time $t=0$, we consider 
well-prepared random homogeneous initial data of the form
 \begin{equation}\label{wellprepared}
 u_{\rm in}(x)=\frac{1}{L}\sum_{k \in \mathbb{Z}_L}\sqrt{n_{\textrm{in}}(k)} \, \eta_{k}(\omega)\,e^{2\pi i kx},
\end{equation}
where $n_{\mathrm{in}}$ is a {nonnegative} Schwartz function,
and $\{\eta_k\}_{k\in\mathbb{Z}_L}$ are independent, identically distributed 
complex Gaussian random variables.

Our goal in this paper is to rigorously investigate wave turbulence
and to derive the wave kinetic equation  for the MMT model. Our main results are twofold. 
First, we establish long-time regularity for solutions to the MMT equation
on a large torus $\mathbb{T}_L$,  where the total energy is large,
but the local energy remains small.  Second, we show that as $L \to \infty$ and $\alpha := L^{-\gamma} \to 0$
for some small $\gamma \in (0,1)$, the second moment of the solution
$\mathbf{E}|\widehat{u}(t,k)|^2$ can be approximated by the power-series
expansion associated with the wave kinetic equation.   
\begin{theorem} \label{thm1}
Let $\sigma \in (0,1)\cup(1,2]$ and $\beta \in \left(0,\frac{\sigma}{4}\right)$. 
Assume that $\alpha$ satisfies the scaling law $\alpha = L^{-\gamma}$ for some small $\gamma \in (0,1)$. 
Let $n_{\mathrm{in}} \in \mathcal{S}(\mathbb{R})$ be a {nonnegative} Schwartz function, and let 
$\{\eta_k(\omega)\}_{k\in\mathbb{Z}_L}$ be i.i.d. complex normalized Gaussian random variables.  Let $0<\delta \ll 1$  be sufficiently small in terms of the fixed parameters, and define
 \begin{align}\label{def:T}
 T:= \begin{cases}
       L^{-20\delta}\min\{L,L^{\frac{(1+\gamma)(1+2\beta)}{2}},L^{ \frac{4\gamma}{3}},L^{\frac{\gamma}{\sigma/2-\beta}}\},\ &{\rm if}\ \sigma\in(1,2]; \\
      L^{-20\delta}\min\{L,L^{\frac{(1+\gamma)(1+2\beta)}{2}},L^{ \frac{4\gamma}{3}}\},\ &{\rm if}\ \sigma\in(0,1).
 \end{cases}
 \end{align} 
 Then, with probability $\geq 1-L^{-40}$,
equation \eqref{MMT} admits a smooth solution on $[0,T]$ satisfying
\begin{align*}
 \sup_{t\in[0,T]}\|u(t)\|_{H^{10}}\leq L^\delta,\qquad
 \sup_{t\in[0,T]}\sum_{i=0}^2\|\partial_x^iu(t)\|_\infty
 &\leq L^{-1/2+\delta}.
\end{align*}

 Moreover, we have 
\begin{equation}\label{app}
\mathbf E |\widehat u(t, k)|^2 =n_{\mathrm{in}}(k)+\frac{t}{T_{\mathrm{kin}}}\mathcal K(n_{\mathrm{in}})(k)
+o_{l^\infty_k}\left(\frac{t}{T_{\mathrm {kin}}}\right)_{L \to \infty}
\end{equation}
for all $L^{0+} \leq t \leq {T}$, where $T_{\mathrm{kin}}=\alpha^{-2}/2$, and $o_{l^\infty_k}\left(\frac{t}{T_{\mathrm {kin}}}\right)_{L \to \infty}$ is a quantity that is bounded in $l^\infty_k$ by $L^{-\delta'}\cdot \frac{t}{T_{\mathrm {kin}}}$ for some $\delta'>0$.

\end{theorem}

\begin{remark}
\begin{enumerate}
   
    \item In Theorem~\ref{thm1}, we also establish long-time regularity for solutions to \eqref{MMT} with high probability. 
The expectation $\mathbf{E}$ in \eqref{app} is taken over the event where a smooth solution exists on $[0,T]$, 
which has probability $\geq 1 - L^{-40}$. 
Outside this event, the corresponding quantity is defined to be zero.
\item 
The two time scales
$L^{4\gamma/3}$ and $L^{\gamma/(\sigma/2-\beta)}$
arise from different parts of the argument.
In \cite[Section~8]{DIP25}, 
at a heuristic level, the authors provided a family of diagrams whose
contributions diverge beyond the time scale.
This suggests an obstruction to extending the current
diagrammatic argument beyond $\alpha^{-4/3}$.

For $\sigma>1$, the second time scale comes from the
condition $\alpha T^{\sigma/2-\beta}\ll1$, which ensures
that the renormalized phase correction remains perturbative
in the counting estimates. This additional restriction is redundant when
$\sigma-2\beta\leq3/2$, and is strictly stronger when
$\sigma-2\beta>3/2$.
In particular, it imposes no further restriction for
$1<\sigma\leq3/2$. 
    \item As in the quasilinear water-wave setting \cite{DIP25a,DIP25}, 
our solutions are assumed to be small in certain $L^\infty$-based norms, 
while allowing for large total energy as the size of the torus increases. 
The smallness of the $L^\infty$ norm also reflects the smallness of local energy per unit volume. 
This feature plays an essential role in the energy estimates and in controlling the quasilinear interactions over long time intervals. 
  
    \item The well-/ill-posedness theory for \eqref{MMT} itself remains a difficult and interesting problem. 
In the semilinear case $\beta=0$, this question has been extensively studied in the literature; 
see, for instance, \cite{CHHO13,CHHO14,CHKL14,CHKL15,HS15,BLLZ23} and the references therein. 
In contrast, the quasilinear regime $\beta>0$ is substantially more challenging. 
Indeed, the derivative-type nonlinear interactions amplify high-frequency components and lead to a derivative-loss mechanism, making the equation significantly harder to analyze.  
We refer to the recent work \cite{PPW26} for progress on the well-/ill-posedness theory in the quasilinear setting with $\sigma\in(1,2]$. 
    \item  As mentioned earlier, in the case $\sigma \in (1,2]$, the first-order term $\mathcal{K}(n_{\mathrm{in}})$ vanishes. In particular, we show that there is no nontrivial evolution of the second moment up to time ${T}$. This reflects the strong non-resonant structure of the equation in this regime. We also note that even the local well-posedness of \eqref{WKE} remains a highly nontrivial problem, especially in the regime $\beta > 0$, where the collision kernel becomes singular near low frequencies. Recent progress on the well-/ill-posedness theory of the WKE can be found in \cite{GIK20,GLZ25,AL25}. Our analysis does not rely on the well-posedness theory of the WKE itself; instead, we establish convergence only at the level of the first-order kinetic expansion.
 
\end{enumerate}
\end{remark}

The derivation of the wave kinetic equation has be achieved for the cubic NLS \cite{DH23b,DH23c} and the MMT  with no derivatives in the nonlinearity  \cite{Vas24}.  
 However, under the circumstances we are considering,  the MMT equations \eqref{MMT} are quasilinear and solutions cannot
 be constructed by iteration of the Duhamel formula due to unavoidable derivative loss. Following the method  developed for 2D gravity water waves \cite{DIP25a,DIP25}, we establish our main result, Theorem \ref{thm1}, via a two-step strategy: propagation of randomness, and energy estimates.
 
We begin by introducing a Fourier truncation of \eqref{MMT}. 
More precisely, we define the truncation operator
\[
\widehat{P_{\leq K}f}
=
\varphi_{\leq K}\widehat{f},
\]
where the cutoff function $\varphi_{\leq K}(x)=\varphi(x/2^K)$ is defined in Section~\ref{sec:pre}. 
The truncation parameter $K$ is chosen as
\begin{align}
    K:=\Bigl[\log_2\bigl(L^{{\delta/10}}\bigr)\Bigr],
    \label{def:K}
\end{align}
where $\delta>0$ is the small constant appearing in Theorem~\ref{thm1}, 
and $[\cdot]$ denotes the floor function.

For later use, we introduce an additional fixed exponent
$0<\eta\ll\delta$.
Let $\psi\in C_c^\infty(\mathbb R)$ be an even function satisfying
$0\leq\psi\leq1$, equal to one on $[-1,1]$, and supported in $[-2,2]$.
Set $D_0=L^\eta$ and define
\[
 \widehat{P_{\leq\eta}f}(k)
 :=\psi(k/D_0)\widehat f(k),
 \qquad
 n_{{\rm in},\eta}(k)
 :=\psi(k/D_0)^2n_{\rm in}(k).
\]
The exponent $\eta$ will be chosen sufficiently small in the proof of
Theorem~\ref{thm1}. All Schwartz seminorms of $n_{{\rm in},\eta}$
are bounded uniformly in $L$.

We now consider the truncated MMT equation on $[0,T]$:
\begin{equation}
\tag{Tr-MMT}\label{TrMMT}
\begin{cases}
i\partial_t u_{\rm tr}
+2\pi|\nabla|^\sigma u_{\rm tr}
+\lambda^2P_{\leq K}|\nabla|^\beta
\Bigl[
\bigl||\nabla|^\beta u_{\rm tr}\bigr|^2
|\nabla|^\beta u_{\rm tr}
\Bigr]
=0,\\
u_{\rm tr}(0,x)=P_{\leq\eta}u_{\rm in}(x).
\end{cases}
\end{equation}
Since $\eta\ll\delta$, for  large $L$ we have
\[
P_{\leq\eta}u_{\rm in}= P_{\leq K}\bigl(P_{\leq\eta}u_{\rm in}\bigr).
\]
In the analysis of the truncated equation, we therefore absorb
the initial-data cutoff into $n_{{\rm in},\eta}$ and regard this
function as the prescribed initial spectral density.
The estimates below depend on this profile only through
Schwartz seminorms that are uniformly bounded in $L$.

As a first step, we study the truncated equation using an iterated
Duhamel expansion and the associated Feynman tree representation.
Our goal is to establish well-posedness on $[0,T]$ and to show that
the Fourier second moments of the solution are approximated by
the kinetic expansion with initial spectral density
$n_{{\rm in},\eta}$.

\begin{theorem}\label{thm2}
Let $\sigma\in (0,1)\cup (1,2],\beta\in(0,\frac\sigma4)$. We assume  $\alpha$ has the scaling law $\alpha=L^{-\gamma}$ for some small $\gamma\in(0,1)$. Let $n_{\mathrm{in}} \in \mathcal S$ be a {nonnegative} Schwartz function, and $\eta_{k}(\omega)$ are i.i.d. complex normalized   Gaussian random variables. We assume $0<\delta\ll1$, and   define $T$ as in \eqref{def:T}. 
Let the truncation parameter $K$ be defined as in \eqref{def:K}.  Then, with probability $\geq 1 - L^{-40}$, the truncated MMT equation \eqref{TrMMT}  has a smooth {solution} on $[0,T]$, and
\begin{equation}\label{approx2}
\mathbf E |\widehat u_{\rm tr}(t,k)|^2
={n_{{\rm in},\eta}(k)}
+\frac{t}{T_{\mathrm{kin}}}
 {\mathcal K(n_{{\rm in},\eta})(k)}
+o_{l^\infty_k}\left(\frac{t}{T_{\mathrm{kin}}}\right)_{L\to\infty}.
\end{equation}
for all $L^{0+} \leq t \leq T$, where $T_{\mathrm{kin}}=\alpha^{-2}/2$. 
\end{theorem}
 
  In the second step, we estimate the energy increment of higher-order derivatives of the solution. 
The goal is to establish a deterministic bound showing that the growth of the energy can be controlled by an $L^\infty$-based norm.
\begin{theorem}\label{thm:energy inequality}
    Let $\sigma\in(0,1)\cup (1,2]$ and $\beta\in(0,\frac{\sigma}{4})$. Let ${q\geq10}$ {be an integer} and $T_1<T_2\in\mathbb{R}$. Assume that $u$ is a solution to \eqref{MMT} satisfying for $t'\in[T_1,T_2]$ \begin{align}
   \|u(t')\|_{H^{{{q}}}}\leq A,\ \  \|u(t')\|_{ {W}^{2,0}}\leq \epsilon(t'),\label{bd:thm9.1}
\end{align}
where $A>1$, and $\epsilon:[T_1,T_2]\to[0,\infty)$ is a continuous function with $\epsilon(t)\ll \lambda^{-1}$.
Then   for any $t_1\leq t_2\in[T_1,T_2]$, it holds that
\begin{align*}
 \|u(t_2)\|_{H^{{{q}}}}^2\lesssim  \|u(t_1)\|_{H^{{{q}}}}^2 +\lambda^4A^2\int_{t_1}^{t_2}\epsilon^4(s)\dif s+1_{\{\sigma\in(0,1)\}}{\lambda^2\int_{t_1}^{t_2}\epsilon^2(s)\|u(s)\|_{H^{{{q}}-\frac12+\beta}}^2\dif s}.
\end{align*}
\end{theorem}
We refer to Section~\ref{sec:pre} for the definition of the $L^\infty$-based space $W^{2,0}$. Compared with \cite{DIP25a}, which focused on the gravity
water-wave dispersion relation $|\xi|^{1/2}$, we establish high-frequency
energy estimates for general power-law dispersion relations
$|\xi|^\sigma$. 
An additional difficulty arises when $\sigma\in(0,1)$,
due to the quartic term on the right-hand side of the
energy estimate. Although this term gains $1-2\beta$
derivatives, its  contribution need not be
small on the time scale under consideration.
The energy estimate therefore provides a polynomial
bound for higher Sobolev norms, but does not directly
yield the improvement needed to close the bootstrap.

To address this difficulty, we compare $u$ with a
frequency-localized finite expansion and combine
low-frequency comparison estimates with a normal-form
$L^2$ estimate for the difference.
The polynomial Sobolev bound obtained from Theorem \ref{thm:energy inequality} is then used  to control the nonlinear remainder and
recover the required $W^{2,0}$ bound.
This closes the bootstrap and establishes the existence
of solutions to \eqref{MMT} up to time ${T}$.

\subsection{Further relevant literature}\label{sec:liter}
The study of wave kinetic theory can be traced back to Peierls’ pioneering work \cite{Pei29} on anharmonic crystals. In that study, he first introduced  the phonon Boltzmann equation, which established a statistical physics framework and later inspired extensive generalizations and developments (see subsequent works in \cite{Spo06,Spo08}). Shortly afterwards, Nordheim \cite{Nor28} proposed the quantum Boltzmann equation to describe the dynamics of quantum interacting gases (see subsequent works in \cite{UU33,Hug83,ESY03}). Then, this framework had been extended to a wide class of wave systems \cite{Dav72,ZS67,Has62,BN69}. The MMT equation is introduced by Majda, McLaughlin and Tabak \cite{MMT97}  as a model for assessing the validity of weak turbulence theory for random waves 
in an unambiguous and transparent fashion.  Under certain experimental settings, their observations appear to contradict the classical theory, revealing a new spectral behavior known as the so-called “MMT spectrum” \cite{CMMT99,CMMT01}.  For further development, we refer to \cite{ZDP04,ZGPD01, DB23}.

The main difficulty lies in the rigorous justification of the convergence of the Feynman diagram expansions. In the linear case, it is firstly derived by   Spohn \cite{Spo77} for short times, and then by  Erd\"os, Salmhofer and Yau \cite{ESY03b,EY00} for a global result.  
Subsequently, the rigorous derivation of this problem has seen substantial progress, driven primarily by the work of  \cite{BGHS18,BGHS21,CG20,CG25,DH21,DH23a,DH23b}. In particular, Deng and Hani \cite{DH23c}  extended the  justification of NLS to
 arbitrarily long times that cover the full life span of the WKE. It is worth mentioning that this approach has already been widely applied to other dispersive equations, such as the Wick  NLS \cite{HSZ24} and the NLS with additive stochastic forcing \cite{GH24}. We also refer to \cite{ST21,HRST22} for the wave turbulence theory for   KdV
 type equations.
However, in one spatial dimension, the derivation of the WKE encounters new obstacles due to the difficulty of  divergent contributions in the Feynman expansion. Notably, even at the subcritical scale $T \lesssim   \alpha^{-2}$, \cite{DIP25} identified explicit pairs of coupled Feynman diagrams that yield divergent contributions which appear to lack mutual cancellation. This divergence is an pure one-dimensional phenomenon. 
 Despite these challenges, significant progress has been made in the rigorous derivation of kinetic and long-time regularity results. In particular, Jin and  Ma~\cite{JM21} first established long-time dynamics for NLS via normal form analysis in the regime $ \gamma > 1 $. For the MMT equation with $ \beta = 0 $, Vassilev~\cite{Vas24} derived the WKE up to timescales $ T \sim \alpha^{-5/4} $.
 
However, for quasilinear equations, there are only few results on rigorous justification, due to the unavoidable
derivative loss. Recently, Deng, Ionescu and Pusateri~\cite{DIP25a,DIP25} initiated the study of long-time regularity for quasilinear water-wave systems. Wu \cite{Wu25} and Vassilev, Wu \cite{VW26} established the WKE for the $\beta$-FPUT system, first for a reduced evolution equation and subsequently for the full system, using refined diagrammatic expansions and counting arguments.

\subsection{Ideas of proof}
In what follows, we present an overview of the main steps of the proof, identify the new challenges that arise in this context, and describe the techniques developed to overcome them.
\subsubsection{Propagation of randomness} By using the truncation function to control the  derivative, we could write the solution as a sum of Feynman trees using the iterated Duhamel formula:
\begin{align*} 
     c_k(t)=c^{\leq N}_k (t)+\mathcal{R}^{N+1}_k,
\end{align*}
where $  c^{\leq N}_k (t) =\sum_{i={0}}^N c^{i}_k (t)$ are the sum of the ternary trees of order no more than $N$, and $\mathcal{R}^{N+1}_k$ is the higher order remainder term. The constant $N$ is chosen large enough but independent of $L$. The main components and  ideas of the  analysis are the following:
\begin{enumerate}
\item \textbf{Two truncation scales.}
We introduce two frequency cutoffs for the MMT equation: the
nonlinearity is truncated at scale $L^\delta$, while the initial
data are truncated at scale $L^\eta$, with $0<\eta\ll\delta$.
The separation of these scales is used in the final bootstrap
argument. At this stage, we only require $L^\eta\leq L^\delta$;
the role of choosing $\eta$ sufficiently small relative to
$\delta$ will be explained later. 
   \item  \textbf{Renormalization.} 
To formulate the expansion in terms of ternary trees, it is necessary to remove the nonlinear resonant contributions, which would otherwise lead to unfavorable counting estimates.  In the present quasilinear setting,   we introduce an implicit deterministic renormalization factor depending on both time and frequency in order to cancel part of the resonant interactions.  After renormalization, the dispersion relation takes the form
\[
f(x)=|x|^\sigma+\alpha C\varphi_{\leq K}(x)|x|^{2\beta},
\]
where the second term arises from the renormalization of trivial
couples. The low-frequency analysis is delicate: the unperturbed
dispersion $|x|^\sigma$ already has singular derivatives at the
origin, while the correction $|x|^{2\beta}$
introduces a stronger singularity since $2\beta<\sigma$. 
\item \textbf{Counting estimates with renormalization.}
We first establish two- and three-vector counting estimates for
the renormalized phase $f$. Configurations with at least one
frequency of size at most $T^{-1/2}$ are handled by direct counting.
When all frequencies are away from the origin, we observe that
\[
\alpha |x|^{2\beta}
\leq \alpha T^{\sigma/2-\beta}|x|^\sigma,
\qquad |x|\geq T^{-1/2}.
\]
Under the additional restriction
$\alpha T^{\sigma/2-\beta}\ll 1$, the phase correction is
perturbative, also at the level of the derivatives used in
the counting argument. We can therefore establish the desired
estimates by analyzing the resonance relations as in the
unperturbed case.

For $\sigma\in(0,1)$, we instead exploit the concavity of the
corrected phase on each half-axis, thereby avoiding this
additional time restriction.

\item \textbf{Algorithm.}
We reduce the moment estimates for the coefficients $c_k^i$
to a combinatorial counting problem on molecules.
To carry out this analysis, we apply
a refined  cutting and removal algorithm of \cite[Section 5]{DIP25} and
\cite[Section~6]{VW26}.
A new feature arises in the $\alpha$-cutting step, where we
need to control the phase difference $f(x)-f(x-h)$.
In particular, when $\sigma\in(0,1)$, a small frequency gap
does not necessarily yield a phase gap of comparable size
near the origin.  We therefore refine the $\alpha$-cutting operation by imposing
simultaneous restrictions on the frequency gap and the associated
phase difference:
\[
|h|\leq g,\qquad |f(x)-f(x-h)|\leq g.
\]
This refinement requires us to count the possible phase labels
introduced by each $\alpha$-cut. The remaining $\epsilon$-atoms
with large phase gaps are handled in the removal step, where
these gaps yield lower bounds on the relevant phase derivatives
and hence improved counting estimates. 

 In \cite{DIP25}, the loss in the phase counting is
compensated by the cubic symbol. For our model,
the strength of cubic symbol depends on $\beta$, so the same
argument would impose an additional lower bound on $\beta$. 
We therefore introduce an additional phase cutoff to control this
loss while retaining the range $0<\beta<\sigma/4$.

 \item  \textbf{Convergence of low-order ternary trees.} 
In the proof of Theorem~\ref{thm2}, the higher-order ternary trees are treated as error terms, 
so it suffices to analyze the contributions of couples of orders $0$, $1$, and $2$. 
For these low-order terms, we compute explicit formulas and identify the cancellations arising from the renormalized resonant interactions.  
\end{enumerate}

\subsubsection{Energy estimate}
To control the high-frequency component, we establish deterministic estimates for the energy increment of the solution. 
In the quasilinear setting, this appears to be the only available mechanism for propagating high-regularity bounds over long time intervals. 
Due to the cubic nonlinearity, one may initially expect an energy estimate of quartic type:
\begin{align*}
 \|u(t_2)\|_{H^{{{q}}}}^2
 \lesssim
 \|u(t_1)\|_{H^{{{q}}}}^2
 +\lambda^2A^2\int_{t_1}^{t_2}\epsilon^2(s)\dif s.
\end{align*}
However, with high probability, the $L^\infty$ norm satisfies $\epsilon(t)\sim L^{-1/2},$
which implies $ \lambda^2\epsilon^2\sim {\alpha}.$ 
As a consequence, such a quartic estimate would only provide a lifespan up to
\[
T\sim \alpha^{-1}.
\]
At this timescale,  the main difficulty is  to improve the energy estimate beyond the quartic level.  The work of Deng, Ionescu, and Pusateri \cite{DIP25a,DIP25}
established a framework for quasilinear wave turbulence  corresponding to
the dispersion relation $|\xi|^{1/2}$. A crucial feature of this case is
the special algebraic structure of the phase function, which admits a
favorable factorization on the resonant set.  For a general power-law dispersion
$|\xi|^\sigma$, such a factorization is no longer available in general.
The key ingredients and ideas leading to this improvement are summarized as follows:  
\begin{enumerate}
\item \textbf{Symmetrization and factorization of symbols.} 
Working in Fourier space, we study the quartic energy increment
\begin{align*}
Q(u,\overline u,u,\overline u)
:=
\frac{1}{L^{3}}
\sum_{\xi_1-\xi_2+\xi_3-\xi_4=0}
\widehat{u}_{\xi_1}\overline{\widehat{u}}_{\xi_2}
\widehat{u}_{\xi_3}\overline{\widehat{u}}_{\xi_4}
M(\xi),
\end{align*}
where $M(\xi)
:=
|\xi_1|^\beta|\xi_2|^\beta|\xi_3|^\beta|\xi_4|^\beta
\bigl(|\xi_3|^{2{{q}}}-|\xi_4|^{2{{q}}}\bigr).$ By symmetrizing the quartic sum, we may replace $M$ by the
symmetrized multiplier 
\[
M(\xi)
=
\frac12
|\xi_1|^\beta|\xi_2|^\beta|\xi_3|^\beta|\xi_4|^\beta
\bigl(
|\xi_1|^{2{{q}}}-|\xi_2|^{2{{q}}}
+
|\xi_3|^{2{{q}}}-|\xi_4|^{2{{q}}}
\bigr).
\]
Next, after an integration by parts in time, the analysis reduces to estimating
\[
M(\xi)\,[\Omega^{(\sigma)}(\xi)]^{-1},
\]
where $\Omega^{(\sigma)}(\xi)
:=
|\xi_1|^\sigma-|\xi_2|^\sigma
+
|\xi_3|^\sigma-|\xi_4|^\sigma.$ 
This leads naturally to a small-divisor problem associated with the resonant phase function.
To overcome this difficulty, we exploit the resonance relation
\[
\xi_1-\xi_2+\xi_3-\xi_4=0
\]
and factorize the symbol:
 \begin{align*}
    |\xi_1|^{2{{q}}}&-|\xi_2|^{2{{q}}}+|\xi_3|^{2{{q}}}-|\xi_4|^{2{{q}}}\\
    &=2{{q}}(2{{q}}-1)
     (\xi_1-\xi_2)(\xi_2-\xi_3)\\
     &\quad\quad \times \int_0^1\int_0^1((1-\theta)\eta\xi_2+(1-\theta)(1-\eta)\xi_3+\eta\theta\xi_1+\theta(1-\eta)\xi_4)^{2{{q}}-2}\dif \eta\dif\theta.
\end{align*}
The key point is that the integral term retains sufficient regularity for {every integer $q\geq10$}, while the additional factor $(\xi_1-\xi_2)(\xi_2-\xi_3)$
provides cancellation against the small divisor arising from $\Omega^{(\sigma)}$.  Here, we have to analyze the phase function
$\Omega^{(\sigma)}$ for the full range of power-law dispersions
$|\xi|^\sigma$. Unlike the special cases $\sigma=\frac12$ and
$\sigma=2$, where the phase function admits explicit algebraic
factorizations, no universal factorization is available for a general
value of $\sigma$. Therefore, we need to perform a case-by-case analysis
of the resonance geometry and establish suitable factorizations.
Such structures naturally appear in normal form transformations, where division by resonant phases may otherwise lead to a loss of derivatives.

Moreover, expressing the symbol in this form allows us to obtain robust derivative bounds needed for the application of Lemma~\ref{lem:sinfty}. 
A crucial observation is that for any multi-index $\alpha$, any $a,b>0$, {$\xi_1,\xi_2>0$,} and any $\gamma\in\mathbb R$,
\begin{align*}
    |\xi^\alpha\partial_\xi^\alpha[(a\xi_1+b\xi_2)^\gamma]|
    \lesssim
    (a\xi_1+b\xi_2)^\gamma.
\end{align*}
This property guarantees that the differentiated symbols preserve the same scaling behavior as the original phase factors. Subsequently, we rewrite all relevant expressions in terms of combinations of homogeneous power-law factors, which allows us to recover the appropriate scaling properties and establish uniform symbol estimates.
 
 \item \textbf{Cubic resonances.} 
To improve the quartic energy estimate, it is necessary to analyze carefully the resonant four-wave interactions, characterized by the system
\begin{align*}
    \xi_1 - \xi_2 + \xi_3 - \xi_4 = 0,
    \qquad
    |\xi_1|^{\sigma} - |\xi_2|^{\sigma}
    + |\xi_3|^{\sigma} - |\xi_4|^{\sigma} = 0.
\end{align*}
For every $\sigma\in(0,1)\cup(1,2]$, this system always admits the trivial resonant solutions $\{\xi_1,\xi_3\}=\{\xi_2,\xi_4\}.$
When $\sigma\in(1,2]$, these are in fact the only resonances. 
In contrast, for $\sigma\in(0,1)$, the system also admits a family of genuinely nontrivial resonant solutions, which we denote by $(\xi_i^*)_{1\le i\le4}$.

In the regime $\sigma\in(1,2]$, only trivial resonances occur. 
{The quartic energy symbol vanishes on these trivial resonances, allowing a bounded normal-form correction.} 
To control the associated small divisors, we perform a detailed case-by-case analysis according to the signs and relative sizes of the frequencies $\xi_i$. 
Exploiting the symmetry of the interaction, we reduce the argument to three representative configurations. 
In each case, the resonance relation admits a suitable factorization structure, allowing us to control the small divisors and derive robust derivative estimates for the corresponding symbols. 
Combined with a normal form transformation, this yields the desired sextic energy increment.

The situation is substantially more delicate when $\sigma\in(0,1)$. 
In this regime, the nontrivial resonances generate genuine singularities, since
\[
M(\xi)[\Omega^{(\sigma)}(\xi)]^{-1}
\]
blows up near the resonant set $\xi=\xi^*$. 
For the gravity water-wave equation \cite{DIP25a} (corresponding to $\sigma=\frac12$), the cubic symbol vanishes exactly on the nontrivial resonant set, so the singularity cancels:
\[
\frac{M(\xi)}{\Omega^{(\sigma)}(\xi)}
=
\frac{M(\xi)-M(\xi^*)}
{\Omega^{(\sigma)}(\xi)-\Omega^{(\sigma)}(\xi^*)}.
\]
However, the MMT model does not enjoy such a favorable algebraic structure. 
Instead, we perform a similar renormalized decomposition, while treating separately the contribution of the resonant component $M(\xi^*)$.

A further difficulty is to control the resonant contribution
$M(\xi^*)$. A refined analysis of its symbol yields a quartic
energy term with a gain of $1-2\beta$ derivatives. However,
this gain alone is insufficient to close the bootstrap on
the time scale $T<\alpha^{-4/3}$ by following the argument
of \cite{DIP25}.
\end{enumerate}
\subsubsection{Bootstrap argument}

To overcome this difficulty and achieve $T<\alpha^{-4/3}$, we compare the solution $u$
with the finite tree expansion $V$ associated with
$c_k^{\leq N}$,  rather than with $u_{\rm tr}$.
Since the initial data are truncated at scale $L^\eta$, it holds that $\operatorname{supp}\widehat V 
\subset\{|k|\leq D\},$ with $D:=2(2N+1)L^\eta$. 
Here the expansion order $N$ is independent of $L$.
By choosing $\eta$ sufficiently small relative to $\delta$,
we ensure that $D\ll  L^{\delta/10}$, so the Fourier support of $V$
lies well inside the nonlinear cutoff. This separation keeps the finite tree expansion well inside the
nonlinear cutoff and prevents the cutoff error from reaching
the comparison region.

The difference $w:=u-V$ satisfies
\[
i\partial_tw+2\pi|\nabla|^\sigma w
=\mathcal{L}\mathscr B(V)w+f,
\]
where $f=\mathscr Q_V(w)-r$, $r$ is the equation residual
of $V$, and $\mathscr Q_V(w)$ consists of the terms that
are at least quadratic in $w$. The linearized operator is
\[
\mathcal{L}\mathscr B(V)w:=A_Vw+C_V\overline w,
\]
with $A_Vw:=-2\lambda^2|\nabla|^\beta
\bigl(|\,|\nabla|^\beta V|^2|\nabla|^\beta w\bigr),$ and $
C_V\overline w:=-\lambda^2|\nabla|^\beta
\bigl((|\nabla|^\beta V)^2|\nabla|^\beta\overline w\bigr).$

We first control the low-frequency component of $w$ within $|k|\leq 200D$.
Each application of the linearized Duhamel operator changes
the absolute frequency by at most $2D$. Consequently, after
a fixed number of iterations, the high-frequency cutoff
error with $|k|\geq L^{\delta/10}$ cannot reach this region.
The previously established bounds for the operator iterates
then control the low-frequency component in terms of the
initial discrepancy and the forcing, up to a small multiple
of the full $L^2$ norm of $w$.

The high-frequency linear interactions are treated by a
normal-form correction to the $L^2$ energy of $w$.
Since $A_V$ is self-adjoint, it makes no contribution.
For the contribution of $C_V$, the two $V$ frequencies
are bounded by $D$, while the momentum constraint forces
the two $w$ frequencies to have comparable magnitudes
in the high-frequency region. We obtain
\[
\Omega^{(\sigma)}(\xi)
:=
|\xi_{w,1}|^\sigma-|\xi_{V,2}|^\sigma
+|\xi_{w,3}|^\sigma-|\xi_{V,4}|^\sigma
\gtrsim  (100D)^\sigma.
\]
This lower bound allows us to apply a normal-form correction
as in the preceding energy argument.
Apart from the low-frequency contribution and the terms
involving the residual and forcing, the remaining linear
growth has a time integral of order $\alpha^2TL^{\delta}$,
which is small on the time scale under consideration.

Meanwhile, under the bootstrap assumption on
$\|u\|_{W^{2,0}}$, the higher-order energy estimate gives
a polynomial bound for $\|u\|_{H^{12}}$.
Interpolating this bound with the $L^2$ norm of $w$
makes the estimate for $\mathscr Q_V(w)$ superlinear in
$\sup_{s\leq t}\|w(s)\|_2$.
Together with the small residual and initial discrepancy,
this closes the $L^2$ bootstrap for the difference.
A final interpolation yields the required smallness of
$\|w\|_{W^{2,0}}$ and, combined with the bounds on $V$,
improves the original bootstrap bound for $u$.

\subsection{Organization of the paper} In Section \ref{sec:pre}, we gather the preliminary results  that will be used throughout the text.
Then, the proof is carried out in two main steps. In Sections \ref{trees and LTE}-\ref{sec:proof of theorem 2}, we establish the propagation of randomness for the truncated equation \eqref{TrMMT}. Section \ref{trees and LTE} introduces the basic notions of trees, couples, and molecules, and expresses the solution in terms of ternary trees. In Section \ref{sec:mainest}, we gather some estimates for couples, including both counting estimates and integral bounds. Section \ref{sec:irregular} is devoted to the analysis of irregular chains and the demonstration of their cancellation. Section \ref{sec:Algorithm} presents our main algorithm for molecules. The treatment of the remainder term is given in Section \ref{bd:remainder}. Finally, in Section \ref{sec:proof of theorem 2}, we complete the proof of Theorem \ref{thm2}, which involves delicate calculations for   the low-order ternary trees. Then, in Sections~\ref{sec:acceptineq}–\ref{sec:proofthm1}, we establish the energy estimate stated in Theorem~\ref{thm:energy inequality} and present the proof of the main result, Theorem~\ref{thm1}, respectively.

 \section{Preliminaries}\label{sec:pre}
 \subsection{The Littlewood-Paley projections and associated operators}
 For any $k\in\mathbb{Z}$, we define  an even smooth function $\varphi:\mathbb{R}\to[0,1]$ with  support in $[-\frac85, \frac85]$ and satisfying $\varphi=1$ in $[-\frac54, \frac54]$. We define for $k\in\mathbb{Z}$
\begin{align*}
    \varphi_k(x):=\varphi(x/2^{k})-\varphi(x/2^{k-1}),
\end{align*}
and
\begin{align*}
     \varphi_{\leq k}(x):=\varphi(x/2^{k}),\ \  \varphi_{\geq k}(x):=1-\varphi(x/2^{k-1}).
\end{align*}
Let $P_k,P_{\geq k}, P_{\leq k }$ be the  operators on $\mathbb{T}_L$ defined by the Fourier multipliers $\varphi_k, {\varphi_{\geq k}}, {\varphi_{\leq k}}$ respectively.  Moreover, we write
\begin{align*}
    P_k'=\sum_{a=-2,-1,0,1,2}P_{k+a},\ \  \varphi_k'=\sum_{a=-2,-1,0,1,2}\varphi_{k+a}.
\end{align*}
We also denote $\overline{\mathbb{Z}}:=\mathbb{Z}\cup\{-\infty\}$, and  define the operator $P_{-\infty}$ as
\begin{align*}
    P_{-\infty}f:=\frac{1}{L}\int_{\mathbb{T}_L}f\dif x.
\end{align*}
The operators $P_k$ are given by convolution in the physical space. More precisely, for $f\in L^1(\mathbb{T}_L)$ and $k\in\mathbb{Z}$, we have
\begin{align*}
    P_kf(x):&=\int_{\mathbb{T}_L}f(y)K(x-y)\dif y,\\
    K(x):&=\frac{1}{L}\sum_{\xi\in\mathbb{Z}_L}\varphi_k(\xi)e^{2\pi i\xi x}.
\end{align*}

\subsection{Multipliers }
We will often work with multipliers $m:\mathbb{R}^n\to\mathbb{C}$ operators defined by such multipliers. 
For $n\geq2$, we define
 
\begin{align*}   
   \mathbb{H}_{L}^n& := \{ (\xi_1, \ldots, \xi_{n+1}) \in (L^{-1}\mathbb{Z})^{n+1} : \xi_1 + \ldots + \xi_{n+1} = 0 \},\\
    \mathbb{H}_{\infty}^n &:= \{ (\xi_1, \ldots, \xi_{n+1}) \in \mathbb{R}^{n+1} : \xi_1 + \ldots + \xi_{n+1} = 0 \},
\end{align*} 
and the associated class of symbols
\begin{align*}
    \tilde S^\infty:=\tilde S^\infty(\mathbb{H}_L^n),
\end{align*}
where for $m:\mathbb{H}_L^n\to\mathbb{C}$,
\begin{equation*}
   \| m \|_{\tilde S^\infty} := \bigg{\|}\frac{1}{L^n} \sum_{\xi_1, \ldots, \xi_{n} \in L^{-1} \mathbb{Z}} m(\xi_1, \ldots, \xi_{n},-\xi_1-\xi_2...-\xi_{n}) e^{2\pi i(x_1\xi_1+x_2\xi_2+...+x_n\xi_n)}\bigg{\|}_{L^1_x(\mathbb{T}_L^n)} < \infty.
\end{equation*} 
Then we have the following property on  symbols:  \begin{lemma}$($\cite[Lemma 2.1]{DIP25a}$)$\label{lem:holder}
     Let $p_1,p_2,...,p_n\in[1,\infty],n\geq3$ satisfying $\frac1{p_1}+\frac1{p_2}+...+\frac1{p_n}=1$ and $m\in\tilde{S}^\infty(\mathbb{H}_L^{n-1})$. Then it holds that
     \begin{align*}
        \left| \frac{1}{L^{n-1}}\sum_{(\xi_1,\xi_2,...,\xi_n)\in\mathbb{H}_L^{n-1}}\widehat{f}_1(\xi_1)...\widehat{f}_n(\xi_n)m(\xi_1,\xi_2,...,\xi_n)\right|\lesssim \|f_1\|_{L^{p_1}} \|f_2\|_{L^{p_2}}... \|f_n\|_{L^{p_n}}\|m\|_{\tilde{S}^\infty}.
     \end{align*}
 \end{lemma}
We also need to the following estimates on  the  multipliers.  
 \begin{lemma}$($\cite[Lemma 2.2]{DIP25a}$)$\label{lem:sinfty}
 \begin{enumerate}
     \item 
    If $m,m'\in \tilde{S}^\infty$, then $m\cdot m'\in \tilde{S}^\infty$ and it holds that 
     \begin{align*}
         \|m\cdot m'\|_{\tilde{S}^\infty}\lesssim   \|m \|_{\tilde{S}^\infty}  \|  m'\|_{\tilde{S}^\infty}.
     \end{align*}
   \item 
Assume that $n\geq1$ and $m:\mathbb{R}^n\to\mathbb{C}$ is a continuous compactly supported function and $\check{m}\in L^1(\mathbb{R}^n)$ is the  inverse Fourier transform of $m$. Then
\begin{align*}
    \bigg \|\frac{1}{L^{n}}\sum_{\xi_1,\xi_2,...,\xi_n\in L^{-1} \mathbb{Z}}m(\xi_1,\xi_2,...,\xi_n)e^{2\pi i(x_1\xi_1+x_2\xi_2+...+x_n\xi_n)} \bigg \|_{L^1_x({\mathbb{T}_L^n})}\lesssim\|\check{m}\|_{L^1(\mathbb{R}^n)}.
\end{align*}
\item  Assume that $\underline{k}=(k_1,...,k_{d+1})\in\mathbb{Z}^{d+1},\Lambda>0$,  and the multiplier $m:\mathbb{H}_L^d\to\mathbb{C}$  admits an extension $\tilde{m}:\mathbb{R}^{d+1}\to\mathbb{C}$, such that  $m(\xi)\varphi_{\underline{k}}(\xi)=\tilde{m}(\xi)\varphi_{\underline{k}}(\xi)$ for all $\mathbb{H}_L^d$. Here we introduce the notation $$\varphi_{\underline{k}}(\xi):=\varphi_{k_1}(\xi_1)\varphi_{k_2}(\xi_2)\cdots\varphi_{k_{d+1}}(\xi_{d+1}).$$  If it  satisfies the differential inequalities
\begin{align*}
    |2^{\alpha\cdot \underline{k}}\partial_{\xi}^\alpha[\tilde{m}(\xi_1,...,\xi_{d+1})\varphi_{\underline{k}}(\xi)]|\lesssim \Lambda,
\end{align*}
for any multi-index $\alpha=(\alpha_1,...,\alpha_{d+1})$ with $|\alpha|\leq \frac{d+2}{2}$ and any $\xi\in\mathbb{H}_{\infty}^d$, then
\begin{align*}
    \|m(\xi)\varphi_{\underline{k}}(\xi)\|_{\tilde{S}^\infty}\lesssim \Lambda.
\end{align*}
 \end{enumerate}

 \end{lemma}

\subsection{Function spaces}
In this paper, we denote $H^N/\dot{H}^N$ by the  standard inhomogeneous/ homogeneous  Sobolev spaces on $\mathbb{T}_L$ respectively.  We define the $L^\infty$ based Sobolev space $W^{N,0}$ on $\mathbb{T}_L$ by
\begin{align*}
    \|f\|_{W^{N,0}}:=\|P_{-\infty}f\|_{L^\infty}+\sum_{k\in\mathbb{Z}}\|P_kf\|_{L^\infty} 2^{N\max\{k,0\}}.
\end{align*}
We define an additional $L^\infty$ based Sobolev space $\widetilde{W}^N$ by
\begin{align*}\|f\|_{\widetilde{W}^N}:=\|f\|_{L^\infty}+\sum_{k\in\mathbb{Z}^+}\|P_kf\|_{L^\infty}2^{Nk}.
\end{align*}
In particular, for $N\geq1$, we have
\begin{align*}
    \|f\|_{\widetilde{W}^N}\lesssim  \|f\|_{W^{N,0}}.
\end{align*}

  \section{Tree Expansions}\label{trees and LTE}
  From this section up to Section~\ref{sec:proof of theorem 2}, we focus on the low-frequency component, namely the truncated equation~\eqref{TrMMT}, and aim to prove Theorem~\ref{thm2}.
Our analysis is based on a Feynman diagram expansion. In this section, we show how the solution can be represented in terms of ternary trees (Feynman diagrams) and formulate the main estimates associated with these trees.   
Throughout Sections~\ref{trees and LTE}--\ref{sec:proof of theorem 2},
we write the truncated initial covariance explicitly as
$n_{{\rm in},\eta}$ and we can ignore $\eta$. Such a component will be used in Section \ref{sec:proofthm1}.   From Sections~\ref{trees and LTE}--\ref{sec:proof of theorem 2}, all the estimate is independent of $\eta$.

\subsection{First reductions}\label{reductions}
We consider the truncated MMT equation as introduced in \eqref{TrMMT}.
 Let $(\widehat{u}_{\rm tr})_k(t)$ denotes the Fourier coefficients of $u_{\rm tr}(t,x)$. First, we define $$a_k(t):= e^{-2\pi i|k|^\sigma t} (\widehat u_{\rm tr})_k(t),\ \ t\in[0,T],$$ and obtain the following equation for the Fourier modes:
 \begin{equation*}
\begin{cases}
 i \partial_t{a_k} =- \lambda^2 L^{-2} \sum\limits_{\substack{(k_1,k_2,k_{3}) \in (\mathbb{Z}_L)^3 \\ k - k_1 + k_2 -k_3 = 0}} a_{k_1}\overline{a_{k_2}}  a_{k_3} e^{2\pi i \Omega^{(\sigma)}(k_1,k_2,k_3,k)t}B_{k_1,k_2,k_3,k}\varphi_{\leq K}(k), \\ 
a_k(0) =(a_k)_{\rm in},
\end{cases}
\end{equation*}
where we define 
\begin{align*}
\Omega^{(\sigma)}(k_1,k_2,k_3,k) &:=|k_1|^\sigma-|k_2|^\sigma+|k_3|^\sigma-|k|^\sigma,\\ B_{k_1,k_2,k_3,k}&:=|k_1|^\beta|k_2|^\beta|k_3|^\beta|k|^\beta.
\end{align*}
Here and in the following, we write $\Omega^{(\sigma)}(k_1,k_2,k_3,k)=\Omega(k_1,k_2,k_3,k)$ when there is no ambiguity.
Then, we rescale the time by setting
$$
b_k(t)=a_k(Tt),\ \ t\in[0,1],
$$ 
which solves 
\begin{equation*}
\begin{cases}
 i \partial_t{b_k} =- \lambda^2 L^{-2} T\sum\limits_{k - k_1 + k_2 -k_3 = 0} b_{k_1}\overline{b_{k_2}}  b_{k_3} e^{2\pi i \Omega(k_1,k_2,k_3,k)Tt}B_{k_1,k_2,k_3,k}\varphi_{\leq K}(k),\\
b_k(0) =(a_k)_{\rm in}.
\end{cases}
\end{equation*}
Note that the positions of $k_1,k_3$ are symmetric above, so the sum above can be written as 
$$
\sum\limits_{ k - k_1 + k_2 -k_3 = 0}= \sum_{k_1, k_3\neq k}+2\sum_{k_1=k,k_3\neq k}+\sum_{k_1=k_2=k_3=k},
$$
which allows us to write, by introducing the renormalized variable
$$c_k=b_ke^{-2\pi i \Gamma(k,t)},\ \ t\in[0,1],$$
as
\begin{align}\label{ipartck}
 i \partial_t{c_k}& = - \lambda^2 L^{-2} T\sum\limits_{k - k_1 + k_2 -k_3 = 0} c_{k_1}\overline{c_{k_2}}  c_{k_3} e^{2\pi i \Theta(k_1,k_2,k_3,k,t)}B_{k_1,k_2,k_3,k}\varphi_{\leq K}(k) +2\pi\partial_t\Gamma(k,t)c_k\notag\\ 
& = -\alpha L^{-1} T
\bigg(   \sum\limits_{(k_1,k_2, k_{3})}^{\times} c_{k_1}\overline{c_{k_2}}  c_{k_3} e^{2\pi i \Theta(k_1,k_2,k_3,k,t)}B_{k_1,k_2,k_3,k}\varphi_{\leq K}(k) {-}|c_k|^2c_kB_{k,k,k,k}\varphi_{\leq K}(k) \bigg)\notag\\
 &\qquad- 2\alpha L^{-1} T\sum_{k_1}|c_{k_1}|^2B_{k_1,k_1,k,k}\varphi_{\leq K}(k)c_k+2\pi\partial_t\Gamma(k,t)c_k,
\end{align}
where we denote
$$\Theta(k_1,k_2,k_3,k,t):=\Omega(k_1,k_2,k_3,k)Tt+ \Gamma(k_1,t)- \Gamma(k_2,t)+ \Gamma(k_3,t)- \Gamma(k,t).$$ Here, we still do not emphasize $\sigma$ if it does not cause ambiguity.
We recall the relation $\alpha =\lambda^2L^{-1}$.  Here and below, the notation $\sum^{\times}_{(k_1,k_2, k_{3})}$ denotes summation under the constraints 
$k_j\in\mathbb{Z}_L$, $k_1-k_2+k_3=k$, and $k\not\in\{k_1,k_3\}$. 
The function $\Gamma(k,t)$ is introduced to handle resonant contributions via renormalization; 
its precise choice will be specified in the sequel. We additionally ask that $\Gamma(k,0)=0$.

Finally, we recall the well-prepared initial data \eqref{wellprepared}, which yield
\begin{equation}\label{initial0}
c_k(0)=(a_k)_{\mathrm{in}}
=\varphi_{\leq K}(k) \sqrt{n_{{\rm in},\eta}(k)}\cdot\eta_k(\omega).
\end{equation}  
 In the rest of this paper, we focus on the system \eqref{ipartck} on the time interval $[0,1]$ together with the  well-prepared  initial condition \eqref{initial0}.

\subsection{The tree expansion}
Now, we try to express the solution to \eqref{ipartck} in terms of certain ternary tree expansions, following the ideas in \cite{DH21,DH23a,Vas24,Wu25}. 
This approach provides a systematic way to represent the iterated Duhamel expansions of the solution in a combinatorial structure. 
More precisely, each node of the ternary tree corresponds to one trilinear interaction generated by the nonlinearity, while the leaves represent the initial data.  
In particular, this ternary tree formulation can be viewed as an alternative of the classical Feynman diagram expansion, 
which is often used in physics and wave turbulence theory.

\begin{definition} 
\begin{enumerate}
\item Let $\mathcal{T}$ be a ternary tree. We use $\mathcal{L}$ to denote the set of leaves, use $\mathcal{N}=\mathcal{T}\backslash\mathcal{L}$  to denote the set of branching nodes, denote $n$ by the  number of branching nodes, and $\mathfrak{r} \in \mathcal N$ by the root node. Then the number of leaves is $l=2n+1$. The scale of a ternary tree $\mathcal T$ is defined as $|\mathcal T|=n$,   the number of branching nodes. In particular, the scale of a single node $\mathcal{T}=\bullet$ is 0.   The three children subtrees of $\mathcal{T}$ are defined by $\mathcal{T}_1,\mathcal{T}_2$
 and $\mathcal{T}_3$ from left to right.

\item   For each node $\mathfrak{n}\in\mathcal{N}$, let its children from left to right be $\mathfrak{n}_1$, $\mathfrak{n}_2$, $\mathfrak{n}_3$. We refer to them as  siblings of each other. We fix the sign $\zeta_{\mathfrak{n}}\in\{\pm\}$ as follows: first we fix $\zeta_{\mathfrak{r}}=\{\pm\}$, then for any node $\mathfrak{n}\in \mathcal{N}$, define $\zeta_{\mathfrak{n}_1}=-\zeta_{\mathfrak{n}_2}=\zeta_{\mathfrak{n}_3}=\zeta_{\mathfrak{n}}$. The sign of the root $\mathfrak{r}$ of $\mathcal{T}$ is denoted by the sign of $\mathcal{T}$.  The sign of subtrees $\mathcal{T}_1,\mathcal{T}_3$ are the same as $\mathcal{T}$, while the sign of $\mathcal{T}_2$ is opposite.  In the following, we typically denote by $\mathcal{T}^\zeta$ a tree with sign  $\zeta=+/-$. We define the conjugate
 $\overline{\mathcal{T}}$ of a tree $\mathcal{T}$ to be the same tree but with opposite sign.
\end{enumerate}
\end{definition}
 
 Then, we present a tree expression as a  approximation of  $c_k(t)$, which is defined  as follows: let $N$ be a integer to be determined later (independent of $L$), and let
\begin{align}
     c_k(t):=c^{\leq N}_k (t)+\mathcal{R}^{N+1}_k:=\sum_{i=0}^N c^{i}_k (t)+\mathcal{R}^{N+1}_k,\ \  c^{i}_k(t):=\sum_{|\mathcal{T}^+|=i}
 c^{\mathcal{T}^+}_k (t),\label{def:ck}
\end{align}
where the sum in the last term is taken over all the trees with sign $+$ and scale $i$.
 Here, $\mathcal{R}^{N+1}_k$ is denoted by the remainder term. 

For any tree $\mathcal{T}$ (which may have sign $+$ or $-$), the coefficient $c^{\mathcal{T}}_k(t)$ is defined inductively as:
\begin{align}
    c_k^{\bullet}(t)&=\varphi_{\leq K}(k){\sqrt{n_{{\rm in},\eta}(k)}} \cdot\eta_{k}^{\zeta_\mathfrak{r}}(\omega),\notag\\
     c_k^{\mathcal{T}}(t)&=\int_0^t\mathcal{W}(c^{\mathcal{T}_1}, c^{\mathcal{T}_2} , c^{\mathcal{T}_3})_k(t')\,\mathrm{d}t',\label{def:ct:int}
\end{align}
where we write $z^+=z$ and $z^-=\overline{z}$. For a tree with sign $+$, we recall that $\mathcal{T}_1,\mathcal{T}_2,\mathcal{T}_3$ denote the three subtrees of the root $\mathfrak{r}$,  from left to right, and $\mathcal{W}$ is the trilinear operator defined  by
\begin{align}
    \label{Duhamel0}
\mathcal{W}(c^{\mathcal{T}_1}, c^{\mathcal{T}_2} , c^{\mathcal{T}_3})_k(t):
&=\frac{i\alpha T}{L}  \sum\limits_{(k_1,k_2, k_{3})}^{\times} c^{\mathcal{T}_1}_{k_1}{c^{\mathcal{T}_2}_{k_2}}  c^{\mathcal{T}_3}_{k_3} e^{2\pi i \Theta(k_1,k_2,k_3,k,t)}B_{k_1,k_2,k_3,k}\varphi_{\leq K}(k)\notag\\
&\quad+\frac{i\alpha T}{L}\left(c^{\mathcal{T}_1}_{k}{c^{\mathcal{T}_2}_{k}}c^{\mathcal{T}_3}_{k}-c^{\mathcal{T}_1}_{k}\mathbf{E}[{c^{\mathcal{T}_2}_{k}}c^{\mathcal{T}_3}_{k}]-\mathbf{E}[{c^{\mathcal{T}_1}_{k}}c^{\mathcal{T}_2}_{k}]c^{\mathcal{T}_3}_{k}\right)B_{k,k,k,k}\varphi_{\leq K}(k)\notag\\
&\quad+\frac{i\alpha T}{L}\sum_{k_1\neq k}\left(c^{\mathcal{T}_1}_{k_1}{c^{\mathcal{T}_2}_{k_1}} -\mathbf{E}[c^{\mathcal{T}_1}_{k_1}{c^{\mathcal{T}_2}_{k_1}} ] \right) B_{k_1,k_1,k,k}\varphi_{\leq K}(k)c^{\mathcal{T}_3}_{k}\notag\\
&\quad+\frac{i\alpha T}{L}\sum_{k_3\neq k}\left({c^{\mathcal{T}_2}_{k_3}}c^{\mathcal{T}_3}_{k_3} -\mathbf{E}[{c^{\mathcal{T}_2}_{k_3}}c^{\mathcal{T}_3}_{k_3} ]\right) B_{k,k_3,k_3,k}\varphi_{\leq K}(k)c^{\mathcal{T}_1}_{k}.
\end{align}
Here, we apply a renormalization at the level of the product, rather than to the individual factors. This approach partially eliminates the resonant contributions.
 
Comparing \eqref{Duhamel0} with \eqref{ipartck}, we choose the renormalization function so that \eqref{Duhamel0} provides a good approximation to \eqref{ipartck}:
\begin{align}
    \partial_t\Gamma(k,t):&=\frac{\alpha T}{\pi L}\sum_{ |\mathcal{T}^+|\leq N,\ |\mathcal{T}^-|\leq N }\sum_{k_1}\mathbf{E}[c^{\mathcal{T}^+}_{k_1} {c^{\mathcal{T}^-}_{k_1}} ] B_{k_1,k_1,k,k}\varphi_{\leq K}(k),\label{partgamma}\\
    \Gamma(k,0)&=0,\notag
\end{align}
where the first sum is taken over all the positive and negative trees with scale no more than $N$. 
Here we use the fact that $B_{k_1,k_1,k,k}=B_{k,k_1,k_1,k}$. We also notice that the expression $c_k^{\mathcal{T}} = c_k^{\mathcal{T}}(\Gamma)$ depends on $\Gamma$, so that \eqref{partgamma} can be regarded as an ODE system. The choice of $\Gamma$ is then obtained by solving this ODE system via a fixed point argument.  

Compared with the semi-linear case studied in \cite{DH21,DH23a,Vas24}, resonant contributions still persist in the present setting. We therefore introduce the notion of enhanced trees to  capture these residual resonances.

 \begin{definition} 
\begin{enumerate}
\item  (Enhanced Trees). An enhanced tree is a ternary tree $\mathcal{T}$ equipped with an additional set $\mathcal{N}_D\subset \mathcal{N}$ of degenerate branching nodes. Here, for each degenerate node $\mathfrak{n}\in \mathcal{N}_D$ and its three children $\{\mathfrak{n}_1,\mathfrak{n}_2,\mathfrak{n}_3\}$, we choose a marked child $\mathfrak{n}_c$ from $\{\mathfrak{n}_1,\mathfrak{n}_3\}$, and label
 the remaining child   as $\mathfrak{n}_{nc}$. The same degenerate node, with different choices of marked child, is regarded as distinct types of degenerate nodes. In particular, a node $\mathfrak{n}\in\mathcal{N}_D$ may be fully degenerate, meaning that both $\mathfrak{n}_1$ and $\mathfrak{n}_3$ can be chosen as marked children.  In that case the decoration is counted once, and both marked choices are included in the centering term in \eqref{def:Bomega}. 
\item (Admissible Decoration) We assign to each $\mathfrak{n}\in\mathcal{T}$ an element $k_{\mathfrak{n}}\in\mathbb{Z}_L$. We say such a decoration $(\mathcal{N}_D,k[\mathcal{T}])$ is admissible if for any $\mathfrak{n}\in\mathcal{N}$ we have $k_{\mathfrak{n}}=k_{\mathfrak{n}_1}-k_{\mathfrak{n}_2}+k_{\mathfrak{n}_3}.$
Moreover, for any $\mathfrak{n}\in\mathcal{N}/\mathcal{N}_D$, it holds that  $k_{\mathfrak{n}}\not\in\{k_{\mathfrak{n}_1},k_{\mathfrak{n}_3}\}$. For any $\mathfrak{n}\in \mathcal{N}_D$, it holds that    $k_{\mathfrak{n}}=k_{\mathfrak{n}_c},\ k_{\mathfrak{n}_2}=k_{\mathfrak{n}_{nc}}.$ In particular, we say that a node is fully degenerate if $k_{\mathfrak{n}}=k_{\mathfrak{n}_1}=k_{\mathfrak{n}_2}=k_{\mathfrak{n}_3}$, which is consistent with the choice of marked child described above. 
 Moreover, an admissible decoration is uniquely determined by the values of $k_{\mathfrak{l}}$ for $\mathfrak{l}\in\mathcal{L}$.  
\item  For any $n\in\mathcal{N}$ we denote $$\Omega_{\mathfrak{n}}:=\Omega(k_{\mathfrak{n}_1},k_{\mathfrak{n}_2},k_{\mathfrak{n}_3},k_{\mathfrak{n}}),\  \ \sum\Gamma(\mathfrak{n},t)= \Gamma(k_{\mathfrak{n}_1},t)- \Gamma(k_{\mathfrak{n}_2},t)+ \Gamma(k_{\mathfrak{n}_3},t)- \Gamma(k_{\mathfrak{n}},t),$$ and $$\Theta_{\mathfrak{n}}(t):=\Theta(k_{\mathfrak{n}_1},k_{\mathfrak{n}_2},k_{\mathfrak{n}_3},k_{\mathfrak{n}},t):=\Omega_{\mathfrak{n}}Tt+\sum\Gamma(\mathfrak{n},t).$$ 

\end{enumerate}
\end{definition}

From now on, we will refer to an enhanced tree simply as a “tree” for convenience.
Note that the same underlying tree equipped with different decorations is regarded as distinct.

In the expression given by \eqref{Duhamel0}, the four terms on the right-hand side correspond to four distinct configurations: the root $\mathfrak{r}$ is non-degenerate and the decorations $k_i$ are distinct; the root $\mathfrak{r}$ is fully degenerate; the child $\mathfrak{n}_3$ of $\mathfrak{r}$ is degenerate; and the child $\mathfrak{n}_1$ of $\mathfrak{r}$ is degenerate.  

Then, we give the explicit expression for the tree representation $c^{\mathcal{T}}_{k}$.
\begin{proposition} 
For any enhanced tree $\mathcal{T}$ with scale $n$,  we have
\begin{align}
   \label{def:ct}
c^{\mathcal{T}}_{k}(t,\omega)&=\bigg(\frac{\alpha T}{L}\bigg)^n\sum_{\mathcal{D}}\prod_{\mathfrak{n}\in\mathcal{N}}(i\zeta_{\mathfrak{n}}
)\prod_{\mathfrak{n}\in\mathcal{N}}B_{k_{\mathfrak{n}_1},k_{\mathfrak{n}_2},k_{\mathfrak{n}_3},k_{\mathfrak{n}}}\prod_{\mathfrak{n}\in\mathcal{T}}\varphi_{\leq K}(k_\mathfrak{n})\notag\\ 
&\quad\quad \quad\quad \times \mathcal{A}_{\mathcal{T}}(t,\Theta [t,\mathcal{N}])\prod_{\mathfrak{l}\in\mathcal{L}}\sqrt{{n_{{\rm in},\eta}}(k_{\mathfrak{l}})}\mathcal{B}_{\mathcal{T}}(\omega),
\end{align}where the sum on $\mathcal{D}$ is taken over all admissible decorations $(\mathcal{N}_D,k[\mathcal{T}])$ such that $k_{\mathfrak{r}}=k$, and the functions $\mathcal{A}_{\mathcal{T}}(t,\Theta [t,\mathcal{N}])$ are defined iteratively as
\begin{align}\label{def:a}
    \mathcal{A}_{\bullet}(t,\Theta [\mathcal{N}])&:=1,\notag\\
     \mathcal{A}_{\mathcal{T}}(t,\Theta [\mathcal{N}])&:=\int_0^t e^{2\pi i \zeta_{\mathfrak{r}}\Theta_{\mathfrak{r}}(t')}\prod_{i=1}^3 \mathcal{A}_{\mathcal{T}_i}(t',\Theta [\mathcal{N}_i])\dif t',
\end{align}
where $\mathcal{T}_i$ are three subtrees on $\mathcal{T}$ from left to right, and $\mathcal{N}_i$ is the set of branching nodes in  $\mathcal{T}_i$, with $\mathcal{N}=\mathcal{N}_1\cup \mathcal{N}_2\cup \mathcal{N}_3\cup \{\mathfrak{r}\}$. 

The noise terms are defined iteratively as:
\begin{align}
    \mathcal{B}_{\bullet}(\omega)&:=\eta_{k}(\omega)^{\zeta_{\mathfrak{r}}},\notag\\
    \mathcal{B}_{\mathcal{T}}(\omega)&:=\prod_{i=1}^3  \mathcal{B}_{\mathcal{T}_i}(\omega)-1_{\mathfrak{r}\in\mathcal{N}_D} \sum_{\mathfrak r_c} \mathcal{B}_{\mathcal{T}_{c}}(\omega)\mathbf{E}[  \mathcal{B}_{\mathcal{T}_{2}}(\omega)  \mathcal{B}_{\mathcal{T}_{{nc}}}(\omega)],\label{def:Bomega}
\end{align}
where, if $\mathfrak{r}\in\mathcal{N}_D$, we denote by $\mathcal{T}_{c}, \mathcal{T}_{2}, \mathcal{T}_{nc}$ the subtrees of $\mathfrak{r}$ rooted at the nodes $\mathfrak{n}_{c}, \mathfrak{n}_{2}, \mathfrak{n}_{nc}$, respectively.  Here, the centering sum runs over the admissible marked child, for the fully degenerate case, it has two choices.

\end{proposition}
\begin{proof} We proceed by induction. First, \eqref{def:ct} holds for $|\mathcal{T}|=0$ and $1$ by definition and by \eqref{def:ct:int}.  Now, assume that \eqref{def:ct} holds for all trees $\mathcal{T}$ with scale at most $n$. Let $\mathcal{T}$ be a tree of scale $n+1$ (with sign $+$, without loss of generality). Its three children $\mathcal{T}_i$, $i=1,2,3$, have scales at most $n$, so that \eqref{def:ct} holds for each $\mathcal{T}_i$ by the induction hypothesis. Therefore, by \eqref{def:ct:int}, we obtain  
\begin{align*}
c^{\mathcal{T}}_{k}(t,\omega)&=\bigg(\frac{\alpha T}{L}\bigg)^{n+1}\int_0^t \Bigg{\{}
\sum_{k=k_1-k_2+k_3}\prod_{i=1}^3\Big{(}\sum_{\mathcal{D}_i}\prod_{\mathfrak{n}\in\mathcal{N}_i}(i\zeta_{\mathfrak{n}}
)\prod_{\mathfrak{n}\in\mathcal{N}_i}B_{k_{\mathfrak{n}_1},k_{\mathfrak{n}_2},k_{\mathfrak{n}_3},k_{\mathfrak{n}}}\notag\\
&\times\prod_{\mathfrak{n}\in\mathcal{T}_i}\varphi_{\leq K}(k_\mathfrak{n})\mathcal{A}_{\mathcal{T}_i}(s,\Theta [s,\mathcal{N}])\prod_{\mathfrak{l}\in\mathcal{L}_i}\sqrt{{n_{{\rm in},\eta}}(k_{\mathfrak{l}})}\mathcal{B}_{\mathcal{T}_i}(\omega)\Big{)}e^{2\pi i \Theta_{\mathfrak{r}}(s)}i
B_{k_1 ,k_2 ,k_{3},k}\varphi_{\leq K}(k )\notag\\
&-1_{\mathfrak{r}\in\mathcal{N}_D}\sum_{k_2}
\prod_{i=1}^3\Big{(}\sum_{\mathcal{D}_i}\prod_{\mathfrak{n}\in\mathcal{N}_i}(i\zeta_{\mathfrak{n}}
)\prod_{\mathfrak{n}\in\mathcal{N}_i}B_{k_{\mathfrak{n}_1},k_{\mathfrak{n}_2},k_{\mathfrak{n}_3},k_{\mathfrak{n}}}\prod_{\mathfrak{n}\in\mathcal{T}_i}\varphi_{\leq K}(k_\mathfrak{n})\notag\\
&\times\mathcal{A}_{\mathcal{T}_i}(s,\Theta [s,\mathcal{N}])\prod_{\mathfrak{l}\in\mathcal{L}_i}\sqrt{{n_{{\rm in},\eta}}(k_{\mathfrak{l}})}\mathcal{B}_{\mathcal{T}_{c}}(\omega)\mathbf{E}[\mathcal{B}_{\mathcal{T}_{2}}\mathcal{B}_{\mathcal{T}_{nc}}]\Big{)} iB_{k ,k ,k_{2},k_{2}}\varphi_{\leq K}(k )\Bigg{\}}\dif s.\notag
\end{align*} 
Here, in the first line on the right-hand side, the sum over $\mathcal{D}_i$ is taken over all admissible decorations of $\mathcal{T}i$ such that $k_{\mathfrak{n}_i}=k_i$.
In the second line, the sum over $\mathcal{D}_i$ is taken over all admissible decorations of $\mathcal{T}_i$ satisfying $k_{\mathfrak{n}_c}=k$ and $k_{\mathfrak{n}_2}=k_{\mathfrak{n}_{nc}}=k_2$.
We denote by $\mathcal{L}_i$ the set of leaves of $\mathcal{T}_i$.
We also remark that, when $\mathfrak{r}\in\mathcal{N}_D$, there are two possible choices of the marked child, corresponding to the two terms in \eqref{Duhamel0}.
Then, \eqref{def:ct} follows directly from the definitions \eqref{def:a} and \eqref{def:Bomega}.

\end{proof}

From now on, we write $B_\mathfrak{n}=B_{k_{\mathfrak{n}_1},k_{\mathfrak{n}_2},k_{\mathfrak{n}_3},k_{\mathfrak{n}}},\ \varphi_{\mathfrak{n}}=\varphi_{\leq K}(k_\mathfrak{n})$ for simplicity.

In order to analyze the second-order statistics of the solution, we need to consider products of tree expansions and compute their expectations. This naturally leads to pairing structures between leaves, arising from the independence of the Gaussian random variables.
 In addition, due to the presence of renormalization and degenerate nodes, it is necessary to impose further constraints on admissible pairings. 
\begin{definition}\label{def:couple}
\begin{enumerate} \item  (Couple)
    A couple $\mathcal{Q}$ consists of two trees $\mathcal{T}^+$ and $\mathcal{T}^-$, each labeled with
 opposite signs, together with the choice of degenerate nodes $\mathcal{N}_D=(\mathcal{N}_D)_+\cup( \mathcal{N}_D)_-$, and a partition $\mathcal{P}$ of the set of leaves $\mathcal{L}=\mathcal{L}_+ \cup \mathcal{L}_-$ into $(n+1)$ disjoint
 two-element subsets, where $n = n(\mathcal{T}^+) + n(\mathcal{T}^-)$ is called the order of the couple. The
 partition $\mathcal{P}$ must satisfy the condition $\zeta_l=-\zeta_{l'}$ for every pair $\{l,l'\} \in\mathcal{ P}$. For a couple
 $\mathcal{Q}=\{\mathcal{T}^+,\mathcal{T}^-,\mathcal{P},\mathcal{N}_D\}$, we denote by $\mathcal{N} = \mathcal{N}_+\cup \mathcal{N}_-$ the branching nodes. We define $$
 \zeta(\mathcal{Q}) =\prod_{\mathfrak{n}\in\mathcal{N}} (i\zeta_{\mathfrak{n}}).$$
 The trivial couple is given by two trivial trees whose roots are paired.
 A decoration $\mathcal{E}$  is obtained by decorating each tree $\mathcal{T}^+$ and $\mathcal{T}^-$, with  further requirement that $k_l = k_{l'}$ for every pair $\{l,l'\} \in\mathcal{ P}$. 
 A decoration  is called a $k$-decoration if $k_{\mathfrak{r}_+} = k_{\mathfrak{r}_-} = k$.
\item  (Enhanced Couples) We call couple $\mathcal{Q}=\{\mathcal{T}^+,\mathcal{T}^-,\mathcal{P},\mathcal{N}_D\}$  an enhanced couple if  the
 partition $\mathcal{P}$ further satisfies that for any $\mathfrak{n}\in\mathcal{N}_D$, the leaves   from the subtree with root $\{\mathfrak{n}_2,\mathfrak{n}_{nc}\}$  are not completely paired together. In the fully degenerate case, this condition is imposed for each choice of marked child. We refer to such a partition $\mathcal{P}$ as an enhanced pairing.
 \end{enumerate}
\end{definition}

\begin{remark}
  In this paper, the definition of enhanced couples given in Definition~\ref{def:couple} is in fact equivalent to that in \cite[Definition 3.3]{DIP25}. Note that if $k_{\mathfrak{n}_1}\neq k_{\mathfrak{n}_2}$,  then it is impossible for the subtrees rooted at $\mathfrak{n}_1$ and $\mathfrak{n}_2$ to be paired.   In the case $k_{\mathfrak{n}_1}=k_{\mathfrak{n}_2}={k_{\mathfrak{n}_3}}=k_{\mathfrak{n}}$, from \eqref{Duhamel0}, we have the expression
    \begin{align*}
       c^{\mathcal{T}_1}_{k}{c^{\mathcal{T}_2}_{k}}c^{\mathcal{T}_3}_{k} -\mathbf{E}[c^{\mathcal{T}_1}_{k}{c^{\mathcal{T}_2}_{k}} ]c^{\mathcal{T}_3}_{k} -c^{\mathcal{T}_1}_{k}\mathbf{E}[c^{\mathcal{T}_2}_{k}{c^{\mathcal{T}_3}_{k}} ] ,
    \end{align*}
   which implies that neither $\mathcal{T}_1$ nor $\mathcal{T}_3$ can be paired with $\mathcal{T}_2$.
This is consistent with the definition of enhanced pairing in \cite[Definition 3.3]{DIP25}.
\end{remark}

We introduce the following Complex Isserlis' Theorem.
\begin{lemma}
\label{thm:Isserlis}\cite[Lemma A.2]{DH23a}
Let  $k_j \in \mathbb{Z}_L$ and signs $\{\zeta_j\}\subset \{+,-\}$ for $1 \leq j \leq n$ be fixed. It holds that
\begin{align*}
\mathbf{E}\left[\prod_{j=1}^n \eta^{\zeta_j}_{k_j}(\omega)\right] = 
\sum_{\mathcal{P}} 
\prod_{\{j,j'\}\in\mathcal{P} } {1}_{k_j = k_{j'}}
\end{align*}
where the summation ranges over all partitions $\mathcal{P}$ of $\{1,2,...,n\}$ into pairs $\{j,j'\}$ with  $\zeta_j = -\zeta_{j'}$.

\end{lemma}

\begin{lemma}[Enhanced Tree Pairing]

For enhanced trees $\mathcal{T}^+$ and $\mathcal{T}^-$, the expectation satisfies
\begin{align*}
\mathbf{E}\left[\mathcal{B}_{\mathcal{T}^+}\mathcal{B}_{\mathcal{T}^-}\right] = 
\sum_{ \mathcal{P}}1.
\end{align*}
where the summation ranges over all enhanced pairings between $\mathcal{T}^+$ and $\mathcal{T}^-$.
\end{lemma}
The proof of this lemma is a direct consequence of the Complex Isserlis' Theorem introduced in Lemma~\ref{thm:Isserlis}. We refer to \cite[Lemma 5.14]{Wu25} for a detailed proof.

\begin{definition}[Enhanced Couple Operator]
For an enhanced couple  $\mathcal{Q}=\{\mathcal{T}^+,\mathcal{T}^-,\mathcal{P},\mathcal{N}_D\}$ of order $n$, we define  
\begin{align*}
(\mathcal{K}_\mathcal{Q})(t,s,k) := \bigg(\frac{\alpha T}{L}\bigg)^n \zeta(\mathcal{Q}) \sum_{\mathcal{E}}  
\prod_{\mathfrak{n}\in\mathcal{N}}B_\mathfrak{n}\prod_{\mathfrak{n}\in\mathcal{Q}}\varphi_\mathfrak{n}
\int_{\mathcal{E}_\mathfrak{n}}\prod_{\mathfrak{n}\in\mathcal{N}}  e^{\zeta_\mathfrak{n}2\pi i\Theta_\mathfrak{n}(t_\mathfrak{n}) } \dif t_\mathfrak{n}
\prod_{\mathfrak{l} \in \mathcal{L}}^+ {n_{{\rm in},\eta}}(k_l),
\end{align*}
where  the sum on $\mathcal{E}$ is taken over all $k$-decorations on $\mathcal{Q}$.
 The product $\prod_{\mathfrak{l} \in \mathcal{L}}^+$ is taken over all leaves with $+$ signs. 
The time integration domain is
\begin{align*}
\mathcal{E}_\mathfrak{n} = \{t{[\mathcal{N}]} : 0 < t_{\mathfrak{n}'} < t_\mathfrak{n}  \text{ whenever } \mathfrak{n}' \text{ is a child of } \mathfrak{n} ; 
t_\mathfrak{n}  < t \text{ for } \mathfrak{n}  \in \mathcal{N}^+, \text{ and } t_\mathfrak{n}  < s \text{ for } \mathfrak{n}  \in \mathcal{N}^-\}.
\end{align*}
\end{definition}

Then, as a direct result of \eqref{def:ct}, it holds that 
 \begin{align*}
     \mathbf{E}[ c^{\mathcal{T}^+}_k (t)c^{\mathcal{T}^-}_k(s)] =\sum_{\mathcal{P}} (\mathcal{K}_\mathcal{Q})(t,s,k),
 \end{align*}
where the  sum is taken over  all enhanced pairings $\mathcal{P}$ on the couple $(\mathcal{T}^+,\mathcal{T}^-)$. In the following of this paper, we aim to establish the bounds on $(\mathcal{K}_\mathcal{Q})(t,s,k)$, for any fixed  enhanced pairing $\mathcal{P}$.

Going back to \eqref{partgamma}, we rewrite it in the form \begin{align*}
    \partial_t\Gamma(k,t)&=\frac{\alpha T}{\pi L}\varphi_{\leq K}(k)|k|^{2\beta}\sum_{k_1}\varphi_{\leq K}^2(k_1){n_{{\rm in},\eta}}(k_1)|k_1|^{2\beta}\notag\\
    &\quad+\frac{\alpha T}{\pi L}\varphi_{\leq K}(k)|k|^{2\beta}\sum_{k_1}\sum_{\mathcal{Q}_{\leq N}}(\mathcal{K}_\mathcal{Q})(t,t,k_1)|k_1|^{2\beta},\notag
\end{align*}
where the summation $\mathcal{Q}_{\leq N}$ is taken over all nontrivial admissible couples with the rank of each tree at most $N$.
We then decompose $$\Gamma(k,t)=\Gamma_0(k)Tt+\Gamma_1(k,t),$$ where 
 \begin{align}
     \Gamma_0(k)&:= \alpha  \varphi_{\leq K}(k)|k|^{2\beta}C_K,\notag\\
     \partial_t  \Gamma_1(k,t)&:=\frac{\alpha T}{\pi L}\varphi_{\leq K}(k)|k|^{2\beta}\sum_{k_1}\sum_{\mathcal{Q}_{\leq N}}(\mathcal{K}_\mathcal{Q})(t,t,k_1)|k_1|^{2\beta},\ \ \Gamma_1(k,0)=0.\label{def:gamma1}
 \end{align}
 Here  we define $$C_K:=\frac{1}{\pi L}\sum_{k_1\in\mathbb{Z}_L}\varphi_{\leq K}^2(k_1){n_{{\rm in},\eta}}(k_1)|k_1|^{2\beta}  {\lesssim 1}.$$
By the choice of truncation function $\varphi_{\leq K}$, we obtain the bound
\begin{align*}
   \sup_{\xi\in\mathbb{R}} |\Gamma_0(\xi)|\lesssim  \alpha  L^\delta\lesssim L^{-\delta}.
\end{align*}
We remark that the contribution of the trivial couples, encoded in $\Gamma_0(k)$, is the dominant one among all renormalization terms. Indeed, it satisfies  
\[
\Gamma_0(k) T \sim \alpha TL^{\delta},
\]
which would impose the restrictive condition $\alpha T \lesssim 1$ if treated directly.
For this reason, we separate this term and treat it independently in the subsequent analysis.  

We now turn to the analysis of $\Gamma_1$. We observe that \eqref{def:gamma1} defines an ODE in $\Gamma_1$, since   the quantities $c_k^{\mathcal{T}} = c_k^{\mathcal{T}}(\Gamma)$ depend on $\Gamma$, and hence the kernels $\mathcal{K}_{\mathcal{Q}}$ depend implicitly on $\Gamma_1$. Therefore, \eqref{def:gamma1} should be regarded as a nonlinear integral equation for $\Gamma_1$.
 
\subsection{The  remainder term}\label{sec:rem}
In this section, we write down the equation for the remainder term $\mathcal{R}^{N+1}$. Our goal is to derive a precise equation for $\mathcal{R}^{N+1}$ and to exploit its structure in order to control it in the subsequent analysis.
 We recall that by definition,
\begin{align*}
     c_k(t)=c^{\leq N}_k (t)+\mathcal{R}^{N+1}_k.
\end{align*}
Then, combining \eqref{ipartck} and \eqref{def:ct:int} with the definition of the truncated expansion, we deduce that $\mathcal{R}^{N+1}$ satisfies the following equation:
\begin{equation}\mathcal{R}_k^{N+1}=(c_{\sim N})_k+\mathscr{L}_k(\mathcal{R}^{N+1}) +\mathscr{Q}_k(\mathcal{R}^{N+1})+\mathscr{C}_k(\mathcal{R}^{N+1}).\label{eqnr}
\end{equation} Here, the operators $\mathscr{L}$, $\mathscr{Q}$, and $\mathscr{C}$ represent, respectively, the linear, quadratic, and cubic contributions involving the remainder $\mathcal{R}^{N+1}$, while $(c_{\sim N})_k$ collects the source terms coming from trees of scale comparable to $N$.
More precisely, the relevant terms are defined as follows: 
\begin{align*}
(c_{\sim N})_k&:=\sum_{\substack{n_1, n_2, n_3\leq N\\  n_1+n_2+n_3\geq N}}\int_0^t \mathcal{W}(c^{n_1},\overline {c^{n_2}},c^{n_3})_k\dif s,\notag
\end{align*}
where we recall that $c^{i}_k(t):=\sum_{|\mathcal{T}^+|=i}
 c^{\mathcal{T}^+}_k (t),$ and $\mathcal{W}$ is defined in \eqref{Duhamel0}. Note that, in the second term of the expression, we take the complex conjugate, since $c^{i}$ is summed only over trees with sign $+$.
 
To define the remaining three operators, we introduce the trilinear form
\begin{align*}
  \mathcal{W}_1(u,v,w)_k(t):&=\frac{i\alpha T}{L}\sum\limits_{(k_1,k_2, k_{3})}^{\times} u_{k_1}{v_{k_2}}  w_{k_3} e^{2\pi i \Theta(k_1,k_2,k_3,k,t)}B_{k_1,k_2,k_3,k}\varphi_{\leq K}(k)\notag\\
&\quad -\frac{i\alpha T}{L}u_{k}{v_{k}}w_{k}B_{k,k,k,k}\varphi_{\leq K}(k)\notag.    
\end{align*}
Then we have
\begin{align}
\mathscr{L}(v)&:=\int_0^t \big(2\mathcal{W}_1(c^{\leq N}, \overline{c^{\leq N}}, v)+\mathcal{W}_1(c^{\leq N},\overline{ v}, c^{\leq N})+\mathscr{L}_1[c^{\leq N}](v)\big)\dif s, \notag \\
\mathscr{Q}(v)&:=\int_0^t \big(2\mathcal{W}_1(v,\overline{v}, c^{\leq N})+\mathcal{W}_1(v, \overline{c^{\leq N}}, v)+\mathscr{Q}_1[c^{\leq N}](v)\big)\dif s,\notag\\
\mathscr{C}(v)&:=\int_0^t  \big(\mathcal{W}_1(v, \overline{v}, v)+\mathscr{C}_1(v) \big)\dif s,\label{def:l}
\end{align} 
    where
\begin{align*}
(\mathscr{L}_1)[c^{\leq N}]_k(v)&:=2\frac{i\alpha T}{L}\sum_{k_1}\left(|c^{\leq N}_{k_1}|^2 -\mathbf{E}[|c^{\leq N}_{k_1}|^2] \right) B_{k_1,k_1,k,k}\varphi_{\leq K}(k)v_k\notag\\
&\quad +2\frac{i\alpha T}{L}\sum_{k_1}\left(c_{k_1}^{\leq N}\overline{v_{k_1}}+v_{k_1}\overline{c_{k_1}^{\leq N}}\right) B_{k,k_1,k_1,k}\varphi_{\leq K}(k)c_k^{\leq N},\notag\\
(\mathscr{Q}_1)[c^{\leq N}]_k(v)&:=2\frac{i\alpha T}{L}\sum_{k_1}|v_{k_1}|^2 B_{k_1,k_1,k,k}\varphi_{\leq K}(k)c_k^{\leq N}\notag\\
    &\quad +2\frac{i\alpha T}{L}\sum_{k_1}\left(c_{k_1}^{\leq N}\overline{v_{k_1}}+v_{k_1}\overline{c_{k_1}^{\leq N}}\right) B_{k,k_1,k_1,k}\varphi_{\leq K}(k)v_k,\notag\\
(\mathscr{C}_1)_k(v)&:=2\frac{i\alpha T}{L}\sum_{k_1}|v_{k_1}|^2 B_{k_1,k_1,k,k}\varphi_{\leq K}(k)v_k.\notag
\end{align*}
Here we used the symmetry of $\mathcal{W}_1$ and the symmetry of the coefficient $B_{k_1,k_2,k_3,k}$. We also re-decompose the summation as
 $$
\sum\limits_{ k - k_1 + k_2 -k_3 = 0}= \sum_{k_1, k_3\neq k}+2\sum_{k_1=k}-\sum_{k_1=k_2=k_3=k}.
$$

Moreover, by solving \eqref{eqnr} and using the definition of $\varphi_{\leq K}$, we deduce that $\mathcal{R}_k^{N+1}=0$ for all $|k|\geq \frac85\cdot 2^K$.  Let $\tilde\varphi_{\leq K}(x):=\varphi_{\leq K}(x/2)$ and  define the Fourier multiplier $\tilde\varphi_{\leq K}u=(\tilde\varphi_{\leq K}(k)u_k)_{k\in\mathbb{Z}_L}$. Then we
have  $$\tilde\varphi_{\leq K} (k)\mathcal{R}_k^{N+1}=\mathcal{R}_k^{N+1}.$$  
In particular, the remainder $\mathcal{R}^{N+1}$ is localized in frequencies where $\tilde\varphi_{\leq K}\equiv 1$, which allows us to insert this cutoff into the nonlinear terms.
From \eqref{eqnr}, it follows that
\begin{equation*}\mathcal{R}_k^{N+1}=(c_{\sim N})_k+\mathscr{L}(\mathcal{R}_k^{N+1}) +\mathscr{Q}(\tilde\varphi_{\leq K}(k)\mathcal{R}_k^{N+1})+\mathscr{C}(\tilde\varphi_{\leq K}(k)\mathcal{R}_k^{N+1}).\notag
\end{equation*}  
\begin{remark}
We remark that the function $\tilde\varphi_{\leq K}$ is not inserted into the linear operator $\mathscr{L}$, since it will be needed again later in the proof of Theorem~\ref{thm1}.
In particular, when treating the high-frequency component, it is not appropriate to include $\tilde\varphi_{\leq K}$ at this stage.

On the other hand, the inclusion of $\tilde\varphi_{\leq K}$ in the bilinear and trilinear operators is essential, as it ensures control of the growth of the factor $|k|^\beta$ in the nonlinear interactions.
We refer the reader to Section~\ref{sec:analysisL} for further details.
\end{remark}

\subsection{Statement of main estimates} 
We define the $Z$ norm  
  for $\boldsymbol a(t)=(a_k(t))_{k \in \mathbb{Z}_L}$ by
\begin{equation*}
\|\boldsymbol a\|_{Z}^2 = \sup_{t\in[0,1]} L^{-1}\sum_{k \in \mathbb{Z}_L}\langle k\rangle^{10}|{a}_k(t)|^2 .\notag
\end{equation*}

  We recall that $N>0$ is a sufficiently large parameter to be fixed later.
To prove the derivation of the truncated equation, we will require estimates on the contributions of couples, as well as bounds for the linear operator $\mathscr{L}$.

\begin{proposition}[Bounds of tree couples]\label{prop:couple}
Let $1 \leq n \leq N^3$ and $t \in [0,1]$. For any function $\Gamma_1(k,t)$ satisfying
\[
\sup_k \sup_{t \in [0,1]} |\partial_t \Gamma_1(k,t)| \leq 1,
\]
we can define $c_{k}^{\mathcal{T}}=c_{k}^{\mathcal{T}}( \Gamma_1)$ as above, and it holds that
\begin{align}
\left|\sum_{|\mathcal{T}_{+}| + |\mathcal{T}_{-}| = n} \mathbf{E}\left[c_{k}^{\mathcal{T}_{+}}(t) c_{k}^{\mathcal{T}_{-}}(t)\right]\right| 
&\leq C_N\langle k\rangle^{-20}\alpha^2T L^{7\delta-12\delta (n-2)} . \label{bd:couple}
\end{align}The same estimates
hold after fixing the two root-tree orders separately.

Moreover, the right-hand side of \eqref{def:gamma1} defines a functional of $\Gamma_1$,  which we denote by $\mathscr{P}(\Gamma_1)$. For any 
\[
\sup_k \sup_{t \in [0,1]} |\partial_t \Gamma_1(k,t)| + |\partial_t {\Gamma}'_1(k,t)| \leq 1,
\]
it holds that
\begin{align*}
\sup_k \sup_{t \in [0,1]} |\mathscr{P}(\Gamma_1)| &\leq L^{-3\delta},\\
\sup_k \sup_{t \in [0,1]} |\mathscr{P}(\Gamma_1) - \mathscr{P}({\Gamma}'_1)| &\leq L^{-3\delta} \, \sup_k \sup_{t \in [0,1]} |\partial_t \Gamma_1 - \partial_t {\Gamma}'_1|.
\end{align*}
\end{proposition}

Then, by applying a fixed point argument, we obtain the existence of $\partial_t \Gamma_1$. Moreover, it holds that
\begin{align}
   \sup_k\sup_{t\in[0,1]}|   \partial_t\Gamma_1(k,t)|\leq L^{-3\delta}.\label{bd:gamma'}
\end{align}

\begin{proposition}[Bounds of linear operator $\mathscr{L}$]\label{prop:l}
  With probability $\geq 1-L^{-40}$, it holds that for $0\leq n\leq N$
\begin{equation*} \|\mathscr{L}^n\|_{Z\to Z} \lesssim L^{-\delta n/2} L^{50}. 
\end{equation*}
\end{proposition}

These two propositions constitute the main results of the first part. The proof of Proposition~\ref{prop:couple} is presented in Sections~\ref{sec:irregular}--\ref{sec:Algorithm}, while the proof of Proposition~\ref{prop:l} is given in Section~\ref{bd:remainder}.

\subsection{ The molecules}
In order to establish the main estimates, we introduce the notion of molecules. Most of the notation follows \cite{DH23a,Vas24}; however, particular attention is required to handle the degenerate nodes, as in \cite{Wu25}.

Compared with the semilinear setting, the presence of degenerate nodes introduces additional constraints on the frequency assignments and pairing structures, which must be carefully incorporated into the definition of molecules. In particular, these nodes affect both the counting arguments and the available cancellations, and therefore require a refined treatment.

\begin{definition}
\begin{enumerate}
\item (Molecules)
A  molecule $\mathbb{M}$ is a directed graph, where we call vertices as  atoms and edges as bonds. Multiple bonds between the same pair of atoms and loops are also allowed. Each atom has out-degree at most 2 and in-degree at most 2. The degree is the sum of out-degree and in-degree.
We write $v \in \mathbb{M}$ for an atom $v$ in $\mathbb{M}$, and $l \in \mathbb{M}$ for a bond $l$ in $\mathbb{M}$. We write $l \sim v$ If $v$ is an endpoint of $l$. 
We further require that $\mathbb{M}$ does not have any connected components consisting only of atoms each having degree 4 (where the degree is considered in the undirected sense).\\
For a molecule $\mathbb{M}$, we define the quantity
\begin{align*}
\chi := E - V + F,
\end{align*}
where $E$ is the number of bonds, $V$ is the number of atoms, and $F$ is the number of connected components.

\item (Atomic group) An atomic group in $\mathbb{M}$ is a subset of atoms together with some bonds connecting those atoms. 

\item (Bridge)
A single bond $l$ is called a bridge if removing it increases the number of connected components by one.
\item (Loop) A self-connecting bond is called a loop.
\end{enumerate}
\end{definition}

\begin{definition}[Molecules of couples]\label{couple-molecule}
Let $\mathcal{Q}$ be a nontrivial couple. We define the corresponding molecule $\mathbb{M} $ as follows. The atoms of $\mathbb{M}$ correspond to the branching nodes $\mathfrak{n} \in \mathcal{N}$ of $\mathcal{Q}$. For any two branching nodes $\mathfrak{n}_1, \mathfrak{n}_2$, we draw a bond between the corresponding atoms $v_1, v_2$ if one of the two case happens:
\begin{enumerate}
\item ($PC$) One of $\mathfrak{n}_1, \mathfrak{n}_2$ is a parent of the other, and if the child atom corresponds to a node with sign $-$, then the bond is directed  from the parent atom to the child atom. Otherwise, the direction is reversed;
\item ($LP$) A leaf child of $\mathfrak{n}_1$ is paired with a leaf child of $\mathfrak{n}_2$, the bond is directed from the atom whose corresponding child   has sign $-$  to  the other one with $+$.
\end{enumerate}

\end{definition}


\begin{definition}[Enhanced Molecule]
Let $\mathcal{Q}$ be an enhanced couple of order $n$. We construct its  enhanced molecule $\mathbb{M}(\mathcal{Q})$ in the same manner as Definition \ref{couple-molecule} with more restrictions:
\begin{enumerate}
    \item For any degenerate branching node $\mathfrak{n}\in\mathcal{N}_D$, there exists a corresponding atom $v_{\mathfrak{n}}$ in  $\mathbb{M}(\mathcal{Q})$.
    \item There are three atoms $v_c,v_2$, and $v_{nc}$ corresponding to the three children $\{\mathfrak{n}_c,\mathfrak{n}_2, \mathfrak{n}_{nc}\}$ of  $\mathfrak{n}$ (or their corresponding paired  nodes).
\end{enumerate}
\end{definition}
Then for any nontrivial couple $\mathcal{Q}$ of order $n$, the corresponding molecule $\mathbb{M}(\mathcal{Q})$ is connected.  It either has two degree 3 atoms or one degree 2
 atom, while all other atoms have degree 4. 

The following lemma shows the correspondence between molecules and couples. The proof can be found in \cite[Proposition 9.6]{DH23a}, or \cite[Proposition 2.12]{Wu25}.
\begin{lemma}
 For any molecule $\mathbb{M}$ with order $n$, there are at most $C^n$ different couples $\mathcal{Q}$  such that the corresponding  molecule are exactly $\mathbb{M}$.
\end{lemma}

\begin{definition}[Decorations of Molecules]
\label{def:decorations-molecules}
Let $T>0$. Let $\mathbb{M}$ be a molecule. For each atom $v \in \mathbb{M}$, we fix $c_v \in \mathbb{Z}_L$,  and $C_v \in \mathbb{R}$.
A  $(c_v,C_v)$-decoration of $\mathbb{M}$ is a set of $k_l\in\mathbb{Z}_L $ to each  $l \in \mathbb{M}$, such  that for each atom $v$,
\begin{align*}
\sum_{l \sim v} \zeta_{v,l} \,k_l \;=\; c_v,
\quad\text{and}\quad
\bigl|\sum_{l: l\sim v} \zeta_{v,l} (|k_l|^\sigma+\Gamma_0(k_l)) - C_v\bigr|\;<\;T^{-1}.
\end{align*}
\end{definition}
 Here we define $\zeta_{v,l}=1$ if $l$ goes out from $v$, and $\zeta_{v,l}=-1$ if $l$ goes into $v$. We also recall the definition of $\Gamma_0$ given in \eqref{def:gamma1}.

In particular, for a nontrivial couple $\mathcal{Q}$, the related molecule $\mathbb{M}(\mathcal{Q})$, and a $k$-decorations of $\mathcal{Q}$, we define a $k$-decoration of $\mathbb M(\mathcal Q)$ such that $c_v=0$, and for any $l\in\mathbb{M}$,
\begin{enumerate}
\item 
 If $l$ is LP, between leaves $\mathfrak{l}$ and $\mathfrak{l}'$, then define $k_l := k_{\mathfrak{l}} = k_{\mathfrak{l}'}$.
 \item  If $l$ is PC, with child $\mathfrak{n}$, we define $k_l := k_{\mathfrak{n}}$.
 \end{enumerate}
 The parameter $c_v$ in the decoration is usual taken as 
 \begin{equation*}
c_v=\begin{cases}
   0,& {\rm 
\ if}\ v {\rm\ has\ degree\ 2\ or\ 4},\\
 +k,&{\rm 
\ if}\ v {\rm\ has\ out-degree\ 2\ and\ in-degree\ 1},\\
 -k,& {\rm 
\ if}\ v {\rm\ has\ out-degree\ 1\ and\ in-degree\ 2}. 
\end{cases}
\end{equation*}
The choice of  $C_v$ will be decided in the following to derive different counting estimates.

At the end of this section, we introduce the notion of degenerate atoms.

\begin{definition}[Degenerate Atoms]
Consider an enhanced molecule $\mathbb{M}$ equipped with a $(c_v,C_v)$-decoration as above. We say that an atom $v \in \mathbb{M}$ is degenerate if there exist two degenerate bonds $l_1, l_2 \sim v$, with opposite orientations at $v$, such that $k_{l_1}=k_{l_2}$. Furthermore, $v$ is fully degenerate if all bonds $l\sim v$ carry the same frequency $k_l$.
In particular, each degenerate branching node in the couple $\mathcal{Q}$ corresponds to a degenerate atom in the associated molecule $\mathbb{M}(\mathcal{Q})$. 
\end{definition}

\bd[Gaps and Tame Atom] Suppose $\mathbb{M}$ is a molecule with a $(c_v,C_v)$-deco
ration. For atom $v$ and bonds $l_1, l_2 \sim v$ of opposite directions, we define the gap of $v$ with
respect to $(l_1,l_2)$ to be
$$h:=|k_{l_1}-k_{l_2}|.$$ 
We say $v$ is small gap (SG) with respect to $(l_1,l_2)$ if $h \leq T^{-3/4}$ and large gap (LG) otherwise.
We say that v is tame if for all $l_1, l_2 \sim v$ of opposite directions, $v$ is small gap with respect
to  $(l_1,l_2)$. Likewise, if v is large gap with respect to all $l_1, l_2 \sim v$ of opposite directions, we
say that v is a large gap (LG) atom. 
\ed

\section{Collection of necessary estimates}\label{sec:mainest}
In this section, we collect several auxiliary results that will be used in the estimation of $\mathcal{K}_{\mathcal{Q}}$.
\subsection{The counting estimates}
 First, we have the following  two or three-vectors counting estimate.
\begin{lemma}\label{counting} Let $\sigma\in(0,1)\cup (1,2],\beta\in(0,\frac\sigma4)$ and  $D>1.$  Assume $1\le T<L$, and assume $\alpha T^{\sigma/2-\beta} \ll 1$ if  $\sigma\in(1,2]$. Recall $\Gamma_0$ is defined in \eqref{def:gamma1}. Then, uniformly in $(k,a,b,c)\in(\mathbb{Z}_L)^4$ with $ \max\{|k|,|a|,|b|,|c|\}\leq D$ and $m\in\mathbb{R}$, the sets
\begin{multline*}
S_3=\big\{(x,y,z)\in (\mathbb{Z}_L)^3:x-y+z=k,\,\,\big||x|^\sigma+\Gamma_0(x)-|y|^\sigma-\Gamma_0(y)+|z|^\sigma+\Gamma_0(z)-m\big|\leq T^{-1},\\|x-a|\leq 1,\,\,|y-b|\leq 1,\,\,|z-c|\leq 1,\,\,\mathrm{and\ }k\not\in\{x,z\}\big\},\notag
\end{multline*}
\begin{multline*}
S_2^+=\big\{(x,y)\in(\mathbb{Z}_L)^2:x+ y=k,\,\,\big||x|^\sigma+\Gamma_0(x)+|y|^\sigma+\Gamma_0(y)-m\big|\leq T^{-1},\\|x-a|\leq 1,\,\,|y-b|\leq 1\big\},\notag
\end{multline*}
\begin{multline*}
S_2^-=\big\{(x,y)\in(\mathbb{Z}_L)^2:x- y=k,\,\,\big||x|^\sigma+\Gamma_0(x)-|y|^\sigma-\Gamma_0(y)-m\big|\leq T^{-1},\\\quad|x-a|\leq 1,\,\,|y-b|\leq 1\big\},\notag
\end{multline*}satisfy the bounds
\begin{align*}
    \#S_3\lesssim L^{2}T^{-1}D^{2}\log L,\ \  \#S_2^+\lesssim LT^{-1/2}D,\ \  \#S_2^-\lesssim \min\left\{L,1+LT^{-1/2}
                  + \frac{LD^{2-\sigma}}{T|k|}\right\}.
    \end{align*} 
\end{lemma}

The proof follows by a similar  argument to in \cite[Section 3.8]{DIP25}.
 All cell frequencies have absolute value at most $2D$.  We first consider the case $\sigma\in(1,2]$.
\begin{proof}[Proof of $\#S_2^-$]
The bound $\#S_2^- \lesssim L$ follows immediately from a counting argument. Indeed, there are at most $O(L)$ possible choices of $x\in\mathbb{Z}_L$ satisfying $|x-a|\leq 1$, and once $x$ is fixed, $y = x - k$  is uniquely determined.

 For $k\ne0$, if either $|x| \leq T^{-1/2}$ or $|y| \leq T^{-1/2}$, the bound $ C(1+L\cdot T^{-1/2})$ follows directly by counting the number of admissible choices of $x$ or $y$.
 
We therefore assume that $|x|,|y|\geq T^{-1/2}$. Write 
\[
f(x)=|x|^\sigma+\Gamma_0 (x) =|x|^\sigma+\alpha C_K\varphi_{\leq K}(x)|x|^{2\beta}.
\]
with $\Gamma_0 (x)\sim \alpha \varphi_{\leq K}(x)|x|^{2\beta}$. Then,  it holds that
\[
(|x|^\sigma)''\sim |x|^{\sigma-2}\]
and
\[
|\Gamma_0''(x)|\lesssim\alpha |x|^{2\beta-2}\lesssim  \alpha|x|^{\sigma-2}(T^{-1/2})^{2\beta-\sigma}.
\]
By the choice of parameter,  we have 
\begin{align}
 \sgn(x)f'(x)\gtrsim |x|^{\sigma-1},\qquad  f''(x)\sim |x|^{\sigma-2} \gtrsim D^{\sigma-2}.\label{bd:fgeq}
\end{align}

If $x$ and $y$ have the same sign, then by the mean value theorem, 
\begin{align*}
    |f'(x)-f'(x-k)|\gtrsim kD^{\sigma-2}.
\end{align*}
   If $x$ and $y$ have opposite signs (say $x>0$ without loss of generality), we observe that
   \begin{align*}
    |f'(x)-f'(x-k)|\gtrsim x^{\sigma-1}+(-y)^{\sigma-1}\gtrsim (x-y)D^{\sigma-2}.
\end{align*}
 Then, by the definition of $S_2^-$, a standard counting argument yields that the number of admissible values of $x$ is bounded by  $\ C(1+L\cdot T^{-1}\cdot |k|^{-1}D^{2-\sigma})$. 
\end{proof}
\begin{proof}[Proof of $\#S_2^+$]
By the same argument as in the case of $S_2^-$, we may assume that $|x|,|y|\geq T^{-1/2}$. Then by \eqref{bd:fgeq} we have 
\begin{align*}
    |f''(x)+f''(k-x)|\gtrsim D^{\sigma-2}. 
\end{align*}  Therefore, by a standard second-derivative counting argument, the number of admissible choices of $x$ is bounded by $ C(1+L \cdot D^{1-\sigma/2} \cdot T^{-1/2})$.
\end{proof}
\begin{proof}[Proof of $\#S_3$] First, we consider the choice of $y$. If $|y|\leq T^{-1/2}$, then there are at most $L T^{-1/2}$ possible choices for $y$. For each such fixed $y$, by the bound on $\#S_2^+$, the number of pairs $(x,z)$ satisfying $x+z=k+y$ is at most $L T^{-1/2} D$. This yields the desired bound $L T^{-1/2}\cdot L T^{-1/2} D$.

Next, assume that $|y|\geq T^{-1/2}$. If both $|x|,|z|\leq T^{-1/2}$, then there are at most $L T^{-1/2}\cdot L T^{-1/2}$ possible choices. We now assume that $|x|\geq T^{-1/2}$. For each fixed $z$, we consider the two-variable constraint   $x-y=k-z$ with $m$ replaced by $m-(|z|^\sigma+\Gamma_0(z))$.  Applying a similar argument, we obtain that the number of admissible pairs $(x,y)$ is bounded by $C(1+LT^{-1}|k-z|^{-1}D^{2-\sigma})$. Summing over all $z$, we obtain
$$\sum_{z\ne k}LT^{-1}|k-z|^{-1}D^{2-\sigma}\lesssim L^{2}T^{-1}D^{2-\sigma}\log L.$$   The constant terms contribute $O(L)$, which is absorbed by $L^2/T$ since $T<L$. Here, the logarithmic factor $\log L$ arises from summing the weight $|k-z|^{-1}$ over $z\in\mathbb{Z}_L$  with $z\ne k$, as required by $k\notin\{x,z\}$.
\end{proof}
 
\begin{proof}[Proof of the case $\sigma\in(0,1)$.]
 We recall \[
 f(x)=|x|^\sigma+\Gamma_0(x)
     =|x|^\sigma+\alpha C_K\varphi_{\leq K}(x)|x|^{2\beta}.
\]
In the case $|x|\geq 2^{K+1}$, we have  $f(x)=|x|^\sigma$ and it then holds that 
\begin{align}
 \sgn(x)f'(x)\gtrsim |x|^{\sigma-1},\qquad
 -f''(x)\gtrsim |x|^{\sigma-2},\qquad x\ne0.
 \label{eq:concave-counting-phase}
\end{align}
In the case $|x|<2^K$, we have $ f(x)  =|x|^\sigma+\alpha C_K |x|^{2\beta}$, by the positive nature of $|x|^p$, \eqref{eq:concave-counting-phase} also holds. 
In the last case $|x|\sim 2^K$, \eqref{eq:concave-counting-phase} also holds since $O(\alpha (L^\delta)^{2\beta-\sigma})=O(\alpha)$.

Then as the case $\sigma>1$, we also have \[
 \left|\partial_x\bigl[f(x)-f(x-k)\bigr]\right|
 \gtrsim |k|D^{\sigma-2}.
\] \[
 \partial_x^2\bigl[f(x)+f(k-x)\bigr]
 =f''(x)+f''(k-x)\lesssim-D^{\sigma-2}.
\]
With these two bounds in hand, the estimates follow by a similar calculation as in the case $\sigma\in(1,2]$.
\end{proof}

\subsection{The integral estimates on $\mathcal{A}_{\mathcal{T}}$}
In this section, we estimate the time integral $\mathcal{A}_{\mathcal{T}}$ defined in \eqref{def:a}. We rewrite it in the form 
\begin{align*}
 \mathcal{A}_{\mathcal{T}}=   \int_{\mathcal{D}}\prod_{\mathfrak{n}\in\mathcal{N}}  e^{\zeta_\mathfrak{n}2\pi i[\Omega_\mathfrak{n}Tt_\mathfrak{n}+\sum\Gamma(\mathfrak{n},t_\mathfrak{n})] } \dif t_\mathfrak{n},
\end{align*}
where the time integration domain is
\begin{align*}
\mathcal{D} := \{t{[\mathcal{N}]} : 0 < t_{\mathfrak{n}'} < t_\mathfrak{n}<t  \text{ whenever } \mathfrak{n}' \text{ is a child of } \mathfrak{n} \}.
\end{align*}
We write $$\overline\Omega_\mathfrak{n}:=\Omega_\mathfrak{n}+\sum\Gamma_0(k): =\Omega_\mathfrak{n}+\Gamma_0(k_1)-\Gamma_0(k_2)+\Gamma_0(k_3)-\Gamma_0(k),$$
and  then have 
$$\Omega_\mathfrak{n}Tt_\mathfrak{n}+\sum\Gamma(\mathfrak{n},t_\mathfrak{n})=\overline\Omega_\mathfrak{n}Tt_\mathfrak{n}+\sum\Gamma_1(\mathfrak{n},t_\mathfrak{n}).$$
We  emphasize that the notation $\overline{\Omega}_{\mathfrak{n}}$ does not denote complex conjugation, but rather the renormalized dispersion relation defined above.

Then we have the following representation:
\begin{align*}
 \mathcal{A}_{\mathcal{T}}(t,\Theta[\mathcal{N}])=   \int_{\mathcal{D}}\prod_{\mathfrak{n}\in\mathcal{N}}  e^{\zeta_\mathfrak{n}2\pi i[\overline\Omega_\mathfrak{n}Tt_\mathfrak{n}+\sum\Gamma_1(\mathfrak{n},t_\mathfrak{n})] } \dif t_\mathfrak{n}.
\end{align*}
\begin{definition}
     We  fix $d_{\mathfrak{n}}\in\{0,1\}$ for each $\mathfrak{n}\in\mathcal{N}$, we can define $q_{\mathfrak{n}}$ for each $\mathfrak{n}\in\mathcal{T}$ inductively by
\begin{equation*}
q_{\mathfrak{n}}=0\mathrm{\ if\ }\mathfrak{n}\in\mathcal{L};\quad q_{\mathfrak{n}}=d_{\mathfrak{n}_1}q_{\mathfrak{n}_1}-d_{\mathfrak{n}_2}q_{\mathfrak{n}_2}+d_{\mathfrak{n}_3}q_{\mathfrak{n}_3}+\overline\Omega_{\mathfrak{n}}\mathrm{\ if\ }\mathfrak{n}\in\mathcal{N}.\end{equation*}
\end{definition}
\begin{proposition}\label{prop:est:af}
Let $\mathcal{T}$ be a nontrivial tree, and we assume there are some functions  $f_{\mathfrak{n}}:[0,1]\to \mathbb{C}$ satisfying $$\sup_{s\in[0,1]}(|f_{\mathfrak{n}}(s)|+|(\partial_tf_{\mathfrak{n}})(s)|)\leq C_{\mathfrak{n}}.$$
We define
\begin{align}
 \mathcal{A}^f_{\mathcal{T}}(t,\overline\Omega[\mathcal{N}],\Gamma_1[\mathcal{N}]):=   \int_{\mathcal{D}}\prod_{\mathfrak{n}\in\mathcal{N}}  f_{\mathfrak{n}}(t_{\mathfrak{n}})e^{\zeta_\mathfrak{n}2\pi i[\overline\Omega_\mathfrak{n}Tt_\mathfrak{n}+\sum\Gamma_1(\mathfrak{n},t_\mathfrak{n})] } \dif t_\mathfrak{n}.\label{def:af}
\end{align}
 
Then for any function $\Gamma_1(k,t)$ satisfying $$\sup_k\sup_{t\in[0,1]}|\partial_t\Gamma_1(k,t)|\leq 1,$$
it holds that
\begin{align*}
    |\mathcal{A}^f_{\mathcal{T}}(t,\overline\Omega[\mathcal{N}],\Gamma_1[\mathcal{N}])|\lesssim \sum_{d_\mathfrak{n}\in\{0,1\}}\prod_{\mathfrak{n}\in\mathcal{N}}C_\mathfrak{n}\langle Tq_\mathfrak{n}\rangle^{-1}.
\end{align*}  
\end{proposition}

\begin{proof}
  We prove this result by induction.
For $|\mathcal{T}|=1$, we may assume without loss of generality that $\zeta_\mathfrak{r}=+$. Then, for the case $|T\overline\Omega_k|\geq 1$, we have
   \begin{align*}
        \int_0^t &  f_{\mathfrak{r}}(s)e^{ 2\pi i[\overline\Omega_kTs+\sum\Gamma_1(k,s)] } \dif s\\
        &= \frac{f_{\mathfrak{r}}(t)e^{ 2\pi i[\overline\Omega_kTt+\sum\Gamma_1(k,t)] }- f_{\mathfrak{r}}(0)}{2\pi i\overline\Omega_kT}\\
        &\quad - \int_0^t \frac{f'_{\mathfrak{r}}(s)e^{ 2\pi i[\overline\Omega_kTs+\sum\Gamma_1(k,s)] }+f_{\mathfrak{r}}(s)e^{ 2\pi i[\overline\Omega_kTs+\sum\Gamma_1(k,s)] }2\pi i\sum \Gamma_1'(k,s)}{2\pi i\overline\Omega_kT}   \dif s.
   \end{align*}
  Using the assumptions on $f_{\mathfrak{r}}$ and $\Gamma_1$, we obtain 
  \begin{align*}
    |\mathcal{A}^f_{\mathcal{T}}(t,\overline\Omega[\mathcal{N}],\Gamma_1[\mathcal{N}])|\lesssim C_\mathfrak{n}| T\overline\Omega_k|^{-1}\lesssim C_\mathfrak{n}\langle T\overline\Omega_k\rangle^{-1}= C_\mathfrak{n}\langle Tq_\mathfrak{n}\rangle^{-1}.
\end{align*}
Moreover, in the case $|T\overline\Omega_k|\leq 1$, the bound follows directly from
   \begin{align*}
      \left|  \int_0^t   f_{\mathfrak{r}}(s)e^{ 2\pi i \sum\Gamma_1(k,s) } \dif s\right|\leq C_{\mathfrak{r}}.
   \end{align*}

From now on, we assume that the estimate holds for all trees of order at most $n$.
Let $\mathcal{T}$ be a tree with $|\mathcal{T}|=n+1$. Then there exists a branching node $\mathfrak{n}_0$ whose three children are leaves. Moreover, we can choose two sibling nodes (denoted by $\mathfrak{n}_1$ and $\mathfrak{n}_2$) such that each of them is either a leaf or has only leaves as children. We denote the parent of $\mathfrak{n}_0$ by $\mathfrak{n}_p$. The subtrees rooted at $\mathfrak{n}_i$, $i=0,1,2$, are denoted by $\mathcal{T}_i$, respectively. With this decomposition, we can write
   \begin{align*}
       \mathcal{A}^f_{\mathcal{T}}=   \int_{\mathcal{D}_1}\prod_{\mathfrak{n}\in\mathcal{N}_1}  f_{\mathfrak{n}}(t_{\mathfrak{n}})e^{\zeta_\mathfrak{n}2\pi i[\overline\Omega_\mathfrak{n}Tt_\mathfrak{n}+\sum\Gamma_1(\mathfrak{n},t_\mathfrak{n})] } \dif t_\mathfrak{n}
       \mathcal{A}^f_{\mathcal{T}_0}(t_{\mathfrak{n}_p})\mathcal{A}^f_{\mathcal{T}_1}(t_{\mathfrak{n}_p})\mathcal{A}^f_{\mathcal{T}_2}(t_{\mathfrak{n}_p}) ,
   \end{align*}
   where $\mathcal{N}_1$ are the set of branching nodes without the branching nodes in $\mathcal{T}_i,i=0,1,2$ and  \begin{align*}
\mathcal{D}_1 = \{t{[\mathcal{N}_1]} : 0 < t_{\mathfrak{n}'} < t_\mathfrak{n}<t  \text{ whenever } \mathfrak{n}' \text{ is a child of } \mathfrak{n} \}.
\end{align*}

First, we assume that $|T\overline\Omega_{\mathfrak{n}_i}|\geq 1$ for $i=0,1,2$ (if $\mathcal{T}_i$ is nontrivial). Then, for each subtree $\mathcal{T}_i$, the quantity   $\mathcal{A}^f_{\mathcal{T}_i}$ is either equal to $1=\langle Tq_{\mathfrak{n}_i}\rangle^{-1}$ in the trivial case $|\mathcal{T}_i|=0,q_{\mathfrak{n}_i}=0$, or can be written in the form:
\begin{align*}
    \frac{f_{\mathfrak{n}_i}(t_{\mathfrak{n}_p})e^{ \zeta_{\mathfrak{n}_i}2\pi i[\overline\Omega_{\mathfrak{n}_i}Tt_{\mathfrak{n}_p}+\sum\Gamma_1(\mathfrak{n}_i,t_{\mathfrak{n}_p})] }}{2\pi i\overline\Omega_{\mathfrak{n}_i}T}+\frac{\tilde{f}_{\mathfrak{n}_i}(t_{\mathfrak{n}_p})}{2\pi i\overline\Omega_{\mathfrak{n}_i}T},
\end{align*}
where the function $\tilde{f}_{\mathfrak{n}_i}$ arises from integration by parts and satisfies $$\sup_{s\in[0,1]}(|\tilde{f}_{\mathfrak{n}_i}(s)|+|\tilde{f}'_{\mathfrak{n}_i}(s)|)\lesssim |f_{\mathfrak{n}_i}(0)| +\sup_{s\in[0,1]}|f'_{\mathfrak{n}_i}(s)| +\sup_{s\in[0,1]}|f_{\mathfrak{n}_i}(s)\sum\Gamma_1'({\mathfrak{n}_i},s)|\lesssim C_{\mathfrak{n}_i}.$$
Then we  deduce that \begin{align*}
       \mathcal{A}^f_{\mathcal{T}}=\prod_{i:|\mathcal{T}_i|=1}\frac{1}{2\pi i\overline\Omega_{\mathfrak{n}_i}T}\sum_{d_i,\overline  f_{\mathfrak{n}_p}}   \int_{\mathcal{D}_1}\prod_{\mathfrak{n}\in\mathcal{N}_1}  \overline f_{\mathfrak{n}}(t_{\mathfrak{n}})e^{\zeta_\mathfrak{n}2\pi i[\tilde\Omega_\mathfrak{n}Tt_\mathfrak{n}+\sum \tilde \Gamma_1(\mathfrak{n},t_\mathfrak{n})] } \dif t_\mathfrak{n},
   \end{align*}
where the sum is taken over all choices $d_i\in\{0,1\}$ for those indices $i$ with $|\mathcal{T}_i|= 1$, corresponding to whether integration by parts is applied to $\mathcal{T}_i$. The quantities $\tilde\Omega_{\mathfrak{n}}$ and $\tilde\Gamma_1$ are defined as before, except at the node $\mathfrak{n}_p$, where  $$\tilde \Omega_{\mathfrak{n}_p}=q_{\mathfrak{n}_p},\  \sum \tilde{\Gamma}_1=\sum \Gamma_1(k_{\mathfrak{n}_p})+\sum_{i}d_i\sum \Gamma_1(k_{\mathfrak{n}_i}).$$
   The functions $\overline f_{\mathfrak{n}}$ are defined as before, with the exception of the node $\mathfrak{n}_p$, where $\overline f_{\mathfrak{n}_p}$ is given by a product of functions among $f_{\mathfrak{n}_p}$, $f_{\mathfrak{n}_i}$, and $\tilde{f}_{\mathfrak{n}_i}$, depending on the choice of $d_i$.
 
  Since the new tree $\mathcal{T}'$, obtained by removing the nodes in $\mathcal{T}_i$, $i=0,1,2$, satisfies $|\mathcal{T}'|\leq n$, we may apply the induction hypothesis to obtain
\begin{align*}
    |\mathcal{A}^f_{\mathcal{T}}|\lesssim\prod_{i} \frac{1}{\langle Tq_i\rangle}\sum_{d_\mathfrak{n}\in\{0,1\}}\prod_{\mathfrak{n}\in\mathcal{N}_1}C_\mathfrak{n}\langle T\tilde q_\mathfrak{n}\rangle^{-1}\lesssim \sum_{d_\mathfrak{n}\in\{0,1\}}\prod_{\mathfrak{n}\in\mathcal{N}}C_\mathfrak{n}\langle Tq_\mathfrak{n}\rangle^{-1}.
\end{align*}
Here, $\tilde q_\mathfrak{n}$ is defined analogously to $q_\mathfrak{n}$, with $\overline{\Omega}_{\mathfrak{n}}$ replaced by $\tilde{\Omega}_{\mathfrak{n}}$.
We also note that in the case $|\mathcal{T}_i|=0$, one has $q_{\mathfrak{n}_i}=0$, and hence $\langle Tq_{\mathfrak{n}_i}\rangle^{-1}=1$, so that the above estimate remains valid.
 
Finally, if there exists some $ |T\overline{\Omega}_{\mathfrak{n}_i}| \leq 1 $ for $ i = 0,1,2 $,
the proof proceeds in the same manner with even simpler estimates.
In this case, we simply bound the corresponding factor involving $T\overline{\Omega}_{\mathfrak{n}_i}$ by $1$, and the rest of the argument remains unchanged.
\end{proof}

In particular, by taking $f_{\mathfrak{n}}\equiv 1$ for all $\mathfrak{n}\in\mathcal{N}$, we obtain the following estimate for the integral in \eqref{def:a}:
\begin{proposition}
Let $\mathcal{T}$ be a nontrivial tree. 
Then for any function $\Gamma_1(k,t)$ satisfying $$\sup_k\sup_{t\in[0,1]}|\partial_t\Gamma_1(k,t)|\leq 1,$$
it holds that
\begin{align*}
  |\mathcal{A}_{\mathcal{T}}(t,\Theta[\mathcal{N}])|\lesssim \sum_{d_\mathfrak{n}\in\{0,1\}}\prod_{\mathfrak{n}\in\mathcal{N}}\langle Tq_\mathfrak{n}\rangle^{-1}.
\end{align*}
The constant may depend on the fixed maximal tree
order, but is uniform in $\Gamma_1$ under the stated bound.
\end{proposition}

\section{The Irregular Chains}\label{sec:irregular}
In this section, we introduce the notion of irregular chains, first introduced in \cite{DH23a}.
These chains give rise to the worst two-vector estimates, and our goal is to exploit a delicate cancellation mechanism among them.
Compared with previous works, an additional difficulty here comes from the symbols generated by derivatives, which we control by the
truncation functions.

\subsection{Cancellation of irregular chains}

\begin{definition} [Irregular Chains]\label{def:chain-couple}
 Consider an enhanced couple $\mathcal{Q}=\{\mathcal{T}^+,\mathcal{T}^-,\mathcal{P},\mathcal{N}_D\}$. A sequence of nodes $(\mathfrak{n}_0, \ldots, \mathfrak{n}_q)$ is called a  irregular chain if, for each $0 \leq j \leq q-1$, 
\begin{enumerate}
    \item $\mathfrak{n}_{j+1}$ is a child of $\mathfrak{n}_{j}$, and the other two children of $\mathfrak{n}_{j}$ are both leaves  for $j\ge1$; at $j=0$, $\mathfrak m_0$ is a leaf while the exterior child $\mathfrak p_0$ may carry an unchanged subtree,
    \item $\mathfrak{n}_{j}$ has a child $\mathfrak{m}_{j}$ paired with a child $\mathfrak{p}_{j+1}$ of $\mathfrak{n}_{j+1}$, and $\mathfrak m_i$ has opposite sign as $\mathfrak n_{i+1}$,
    \item  we denote $\mathfrak p_0$ to be the remaining child of $\mathfrak n_0$, and 
   $\mathfrak e$ and $\mathfrak f$ to be the remaining two children of $\mathfrak n_q$ with sign $+,-$ respectively.

\end{enumerate}
\end{definition}See Figure \ref{fig:irregular} for examples of two irregular chains. In this configuration, we note that degenerate nodes may also be present.

Let $\mathcal{H}=(\mathfrak{n}_0, \ldots, \mathfrak{n}_q)$ be an irregular chain. By the  definition of irregular chain, $\mathfrak{p}_i$ has the same sign as $\mathfrak{n}_i$ so there are only two choices (the left node or the right node). We denote $\zeta_{\mathfrak{n}_i}=\zeta_i$. We fix the degeneracy of each atoms $\mathfrak{n}_0, \ldots, \mathfrak{n}_q$. Then if we fix  these relative positions and $\zeta_0$, the structure of $\mathcal{H}$ is uniquely
 determined by the signs of $\zeta_i,1\leq i\leq q$.
 
\begin{definition}

\begin{enumerate}
\item (Congruent relation) 
 We define two irregular chains to be congruent (denoted $\mathcal{H}'\equiv \mathcal{H}$), if the sign $\zeta_0$ are the same and  the relative
 position of each $\mathfrak{p}_j$  are the same (both left or right) for these two irregular chains. Then the  the irregular chains $\mathcal{H}$  in this congruence class is uniquely determined
 by the signs $\zeta_{i},1\leq i\leq q$.
  We  call any two congruent irregular chains in a congruence class as twists of each  other.  See Figure \ref{fig:irregular} for an example of two irregular chains as twists.
 \item (Congruent couple) We can consider an irregular chain $\mathcal{H}$ in a couple $\mathcal{Q}$ and all its
 twists. Note that if we twist $\mathcal{H}$, then this twist can be trivially extended to $\mathcal{Q}$, as long as we
 keep the remaining part of the couple $\mathcal{Q}$ unchanged (note that the positions of $\mathfrak{e}$ and $\mathfrak{f}$ may be
 switched, but the sub-trees attached at $\mathfrak{e}$ and $\mathfrak{f}$ are unchanged).  For a given couple $\mathcal{Q}$ and we  suppose that $\mathcal{H}_i,1\leq i\leq r$ are disjoint irregular chains in $\mathcal{Q}$. Then we
  perform twists separately and arbitrarily to  each $\mathcal{H}_i$, starting from $\mathcal{Q}$, and obtain a family of couple $\mathcal{Q}'$, and call them congruent couples, with notation $\mathcal{Q}\equiv\mathcal{Q}'$.  
  \item (Decoration relation) We fix one irregular chain $\mathcal{H}$ in a couple $\mathcal{Q}$, and consider its decoration projected on $\mathcal{H}$.  Then the
  decoration on $\mathcal{H}$ can be uniquely determined by the vectors
 $(k_i, l_i), 0 \leq i\leq q +1$, which are defined as follows:\\ for   $0 \leq i\leq q $ we have $(k_i,l_i)=(k_{\mathfrak{n}_i},k_{\mathfrak{p}_i})$ if $\zeta_i=+$, and $(k_i,l_i)=(k_{\mathfrak{p}_i},k_{\mathfrak{n}_i})$ if $\zeta_i=-$; for $j=q+1$ we denote $(k_{q+1},l_{q+1}) = (k_{\mathfrak{e}},k_{\mathfrak{f}})$.  See Figure \ref{fig:decoration} for the corresponding decorations of irregular chains in Figure \ref{fig:irregular}.
 \item (Gaps) By direct observation,   we have  \begin{align*}
     k_0-l_0=k_1-l_1=\cdots=k_{q+1}-l_{q+1}:=h.
 \end{align*} We call the irregular chain large gap (LG) or small gap (SG) according as $|h|> g:=T^{-3/4}$ or $|h|\le g$, respectively.  
 \end{enumerate}
\end{definition}

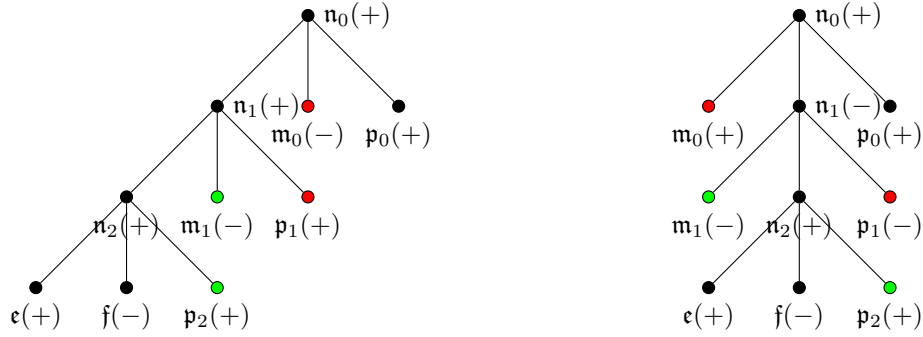
\begin{figure}
\begin{subfigure}{\linewidth}
\centering
\begin{tikzpicture}[level distance=1.2cm,
  level 1/.style={sibling distance=1.2cm},
  level 2/.style={sibling distance=1.2cm}]
\tikzstyle{hollow node}=[circle,draw,inner sep=1.6]
\tikzstyle{solid node}=[circle,draw,inner sep=1.6,fill=black]
\tikzset{
red node/.style = {circle,draw=black,fill=red,inner sep=1.6},
blue node/.style= {circle,draw = black, fill= blue,inner sep=1.6}, 
purple node/.style= {circle,draw = black, fill= purple,inner sep=1.6}, 
orange node/.style= {circle,draw = black, fill= orange,inner sep=1.6},
yellow node/.style= {circle,draw = black, fill= yellow,inner sep=1.6},
green node/.style = {circle,draw=black,fill=green,inner sep=1.6}}
\node[solid node, label = right:{$\mathfrak{n}_0(+)$}] at (-5.5,.5){}
    child{node[solid node, label = right: {$\mathfrak{n}_1(+)$}]{}
        child{node[solid node, label=below: $\mathfrak{n}_2(+)$]{}
             child{node[solid node, label=below: $\mathfrak{e}(+)$]{}}
             child{node[solid node, label=below: $\mathfrak{f}(-)$]{}}
             child{node[green node, label=below: $\mathfrak{p}_2(+)$]{}}
        }
        child{node[green node, label=below: $\mathfrak{m}_1(-)$]{}}
        child{node[red node, label=below: $\mathfrak{p}_1(+)$]{}}
    }
    child{node[red node, label=below: $\mathfrak{m}_0(-)$]{}}
    child{node[solid node, label=below: $\mathfrak{p}_0(+)$]{}}
    
;

\node[solid node, label = right: {$\mathfrak{n}_0(+)$}] at (1,0.5){}
    child{node[red node, label=below: $\mathfrak{m}_0(+)$]{}}
    child{node[solid node, label = right: {$\mathfrak{n}_1(-)$}]{}
        child{node[green node, label=below: $\mathfrak{m}_1(-)$]{}}
        child{node[solid node, label=below: $\mathfrak{n}_2(+)$]{}
             child{node[solid node, label=below: $\mathfrak{e}(+)$]{}}
             child{node[solid node, label=below: $\mathfrak{f}(-)$]{}}
             child{node[green node, label=below: $\mathfrak{p}_2(+)$]{}}
             }
        child{node[red node, label=below: $\mathfrak{p}_1(-)$]{}}
    }
    child{node[solid node, label=below: $\mathfrak{p}_0(+)$]{}}
    
;
\end{tikzpicture}
\end{subfigure}
\caption{Two examples of congruent chains. Points with the same color are paired.}
\label{fig:irregular}
\end{figure}

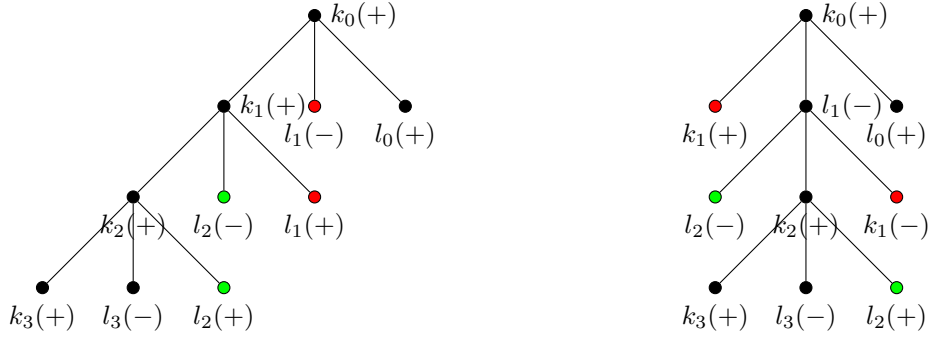
\begin{figure}
\begin{subfigure}{\linewidth}
\centering
\begin{tikzpicture}[level distance=1.2cm,
  level 1/.style={sibling distance=1.2cm},
  level 2/.style={sibling distance=1.2cm}]
\tikzstyle{hollow node}=[circle,draw,inner sep=1.6]
\tikzstyle{solid node}=[circle,draw,inner sep=1.6,fill=black]
\tikzset{
red node/.style = {circle,draw=black,fill=red,inner sep=1.6},
blue node/.style= {circle,draw = black, fill= blue,inner sep=1.6}, 
purple node/.style= {circle,draw = black, fill= purple,inner sep=1.6}, 
orange node/.style= {circle,draw = black, fill= orange,inner sep=1.6},
yellow node/.style= {circle,draw = black, fill= yellow,inner sep=1.6},
green node/.style = {circle,draw=black,fill=green,inner sep=1.6}}
\node[solid node, label = right:{$k_0(+)$}] at (-5.5,.5){}
    child{node[solid node, label = right: {$k_1(+)$}]{}
        child{node[solid node, label=below: $k_2(+)$]{}
             child{node[solid node, label=below: $k_3(+)$]{}}
             child{node[solid node, label=below: $l_3(-)$]{}}
             child{node[green node, label=below: $l_2(+)$]{}}
        }
        child{node[green node, label=below: $l_2(-)$]{}}
        child{node[red node, label=below: $l_1(+)$]{}}
    }
    child{node[red node, label=below: $l_1(-)$]{}}
    child{node[solid node, label=below: $l_0(+)$]{}}
    
;

\node[solid node, label = right: {$k_0(+)$}] at (1,0.5){}
    child{node[red node, label=below: $k_1(+)$]{}}
    child{node[solid node, label = right: {$l_1(-)$}]{}
        child{node[green node, label=below: $l_2(-)$]{}}
        child{node[solid  node, label=below: $k_2(+)$]{}
             child{node[solid node, label=below: $k_3(+)$]{}}
             child{node[solid node, label=below: $l_3(-)$]{}}
             child{node[green node, label=below: $l_2(+)$]{}}
             }
        child{node[red node, label=below: $k_1(-)$]{}}
    }
    child{node[solid node, label=below: $l_0(+)$]{}}
    
;
\end{tikzpicture}
\end{subfigure}
\caption{The corresponding decorated irregular chains are illustrated in Figure \ref{fig:irregular}.}\label{fig:decoration}
\end{figure}

In particular, for any congruent chain, we adopt the shorthand notation $\Omega_i=\Omega_{\mathfrak{n}_i}, \ \sum\Gamma(i,t)=\sum\Gamma(\mathfrak{n}_i,t)$  and $\Theta(i,t)=\Theta(\mathfrak{n}_i,t)$ for any congruent chains  for short. Moreover,  the expressions of   $\zeta_i\Omega_{ i}$ and $\zeta_i\sum\Gamma(i,t)$ are independent of the choice of $\zeta_i$: $$\zeta_i\Omega_i=|k_i|^\sigma- |k_i -h|^\sigma + |k_{i+1}- h|^\sigma-|k_{i+1}|^\sigma,$$
and
$$ \zeta_i\sum\Gamma(i,t)=\Gamma(k_i,t)- \Gamma(k_i -h,t)+ \Gamma(k_{i+1}- h,t)- \Gamma(k_{i+1},t).$$

 \begin{definition}
      (Splicing) Consider a couple $\mathcal{Q}$ and an irregular chain  $\mathcal{H}=(\mathfrak{n}_0,...,\mathfrak{n}_q)$ in it. We apply Splicing to this chain is to remove those nodes $(\mathfrak{n}_1,...,\mathfrak{n}_q)$ and their leaf
children  $(\mathfrak{m}_0,...,\mathfrak{m}_{q-1},\allowbreak\mathfrak{p}_1,...,\mathfrak{p}_{q})$ from $\mathcal{Q}$ and redefining the children of $\mathfrak{n}_0$ by $\mathfrak{p}_0,\mathfrak{e}$ and $\mathfrak{f}$
  with their position determined by  their
 relative positions.  See Figure \ref{fig:splicing} for an example. We denote the rest couple by $\mathcal{Q}^{sp}$.
 \end{definition}

\begin{figure}
\begin{subfigure}{\linewidth}
\centering
\begin{tikzpicture}[level distance=1.2cm,
  level 1/.style={sibling distance=1.2cm},
  level 2/.style={sibling distance=1.2cm}]
\tikzstyle{hollow node}=[circle,draw,inner sep=1.6]
\tikzstyle{solid node}=[circle,draw,inner sep=1.6,fill=black]
\tikzset{
red node/.style = {circle,draw=black,fill=red,inner sep=1.6},
blue node/.style= {circle,draw = black, fill= blue,inner sep=1.6}, 
purple node/.style= {circle,draw = black, fill= purple,inner sep=1.6}, 
orange node/.style= {circle,draw = black, fill= orange,inner sep=1.6},
yellow node/.style= {circle,draw = black, fill= yellow,inner sep=1.6},
green node/.style = {circle,draw=black,fill=green,inner sep=1.6}}
\node[solid node, label = right:{$\mathfrak{n}_0(+)$}] at (-5.5,.5){}
    child{node[solid node, label=below: $\mathfrak{e}(+)$]{}}
    child{node[solid node, label=below: $\mathfrak{f}(-)$]{}}
    child{node[solid node, label=below: $\mathfrak{p}_0(+)$]{}}
;

\node[solid node, label = right: {$k_0(+)$}] at (1,0.5){}
    child{node[solid node, label=below: $k_3(+)$]{}}
    child{node[solid node, label=below: $l_3(-)$]{}}
    child{node[solid node, label=below: $l_0(+)$]{}}   
;
\end{tikzpicture}
\end{subfigure}
\caption{The congruent couple in Figure \ref{fig:irregular} after splicing}
\label{fig:splicing}
\end{figure}
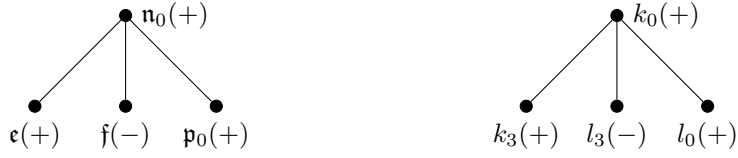

\begin{remark}
 For a non-degenerate atom $\mathfrak{n}_0$, after splicing, it may become degenerate. In this case, the two subtrees rooted at $\mathfrak{n}_2$ and $\mathfrak{n}_{nc}$ may become completely paired, so that the resulting couple is no longer enhanced in the sense of Definition \ref{def:couple}.
For example, in Figure \ref{fig:irregular} (or Figure \ref{fig:splicing}), if $k_0=k_3$ and $l_0=l_3$, then the node $\mathfrak{n}_0$ becomes degenerate, while the leaves in the subtree rooted at ${\mathfrak{f}_2, \mathfrak{p}_{0}}$ may be completely paired.

Similarly, for an originally degenerate node, the splicing procedure may also destroy the enhanced structure. We refer to such nodes as bad degenerate nodes.

\end{remark}  

 It is easy to see that for any $\mathcal{Q}'\equiv\mathcal{Q}$, one has $\mathcal{Q}^{sp}=\mathcal{Q}'^{sp}$.
We then define 
 \begin{align}
     \mathcal{G}_\mathcal{Q}(t,s,k):=\sum_{\mathcal{Q}'\equiv\mathcal{Q}} \mathcal{K}_{\mathcal{Q}'}(t,s,k),\label{def:gq}
 \end{align}
 where the sum is taken over all congruent couples $\mathcal{Q}\equiv\mathcal{Q}'$.  
\begin{lemma}\label{lem-splice}
For an enhanced couple $\mathcal Q$ and one SG irregular chain
$(\mathfrak n_0,\ldots,\mathfrak n_q)$ in $\mathcal Q$, we
 assume $q\ge1$, $\Gamma_1(k,0)=0$, and
$\sup_k\|\partial_t\Gamma_1(k,\cdot)\|_\infty\le1$. Then 
\begin{align*}
    \mathcal{G}_\mathcal{Q}(t,s,k) :&= \bigg(\frac{\alpha T}{L}\bigg)^{n_{sp}} \zeta(\mathcal{Q}^{sp}) \sum_{\mathcal{E}^{sp}}  
    \prod_{\mathfrak{n}\in\mathcal{N}^{sp}}B_\mathfrak{n}\prod_{\mathfrak{n}\in\mathcal{Q}^{sp}}\varphi_\mathfrak{n}
\int_0^1\int_{\mathcal{E}^{sp}_\tau} P_q(\tau, t_{\mathfrak n_0}, k[\mathcal N^{sp}])\\
&\quad\quad\quad\quad\quad\quad\quad\quad\times\prod_{\mathfrak{n}\in\mathcal{N}^{sp}}  e^{\zeta_\mathfrak{n}2\pi i\Theta_\mathfrak{n}(t_\mathfrak{n}) } \dif t_\mathfrak{n}\dif \tau
\prod_{\mathfrak{l} \in \mathcal{L}^{sp}}^+ {n_{{\rm in},\eta}}(k_l),
\end{align*}
 for some function $P_q$ satisfying 
\begin{equation}\label{eq-pestimate}\sup_{|k_{\mathfrak n_0} - k_{\mathfrak p_0}| \leq g} \left|\left|P_q(\tau,t_0,  k[\mathcal N^{sp}])\right|\right|_{l^\infty}+ \left|\left|\partial_{t_0}P_q(\tau,t_0,  k[\mathcal N^{sp}])\right|\right|_{l^\infty} 
 \le C_N(\alpha Tg)^q.
\end{equation}
For the difference of two kernels, multiply the same bound by
\[
 C_N\sup_k\|\partial_t(\Gamma_1-\Gamma'_1)(k,\cdot)\|_\infty.
\]

In the above expression,  $\mathcal{Q}^{sp}$ is denoted 
the remained couple   by splicing out its single irregular chain (we abuse this notation a little). Its corresponding decorations, nodes, leaves are defined by a similar manner. The time internal is defined as 
\begin{equation*}
\mathcal E^{sp}_{\tau} = \mathcal E^{sp} \cap \{ t_{\mathfrak n_0} \geq (t_{\mathfrak e}\vee t_{\mathfrak f} )+ \tau \}.
\end{equation*}
Furthermore, $
{\mathcal{E}^{sp}}$ permits degeneracies at $\mathfrak{n}_0$, while the subtrees permit complete pairing.
\end{lemma}
\begin{proof}
Without loss of generality, we assume that $\mathcal{Q}$ is the element satisfying  $\zeta_{i}=+$ for any  $0\leq i\leq q$. Then we have $k_{\mathfrak n_i} = k_i$, $k_{\mathfrak p_i} = k_{\mathfrak m_{i-1}} = l_i$, and $t_i = t_{\mathfrak n_i},0\leq i\leq q$,  $t_{q+1} = t_{\mathfrak e}\vee t_{\mathfrak f} $.  
  Then we define $\Theta(t)=\Omega Tt+\Gamma(t)$, where 
    \begin{align*}
    \Omega = \sum_{j = 0}^q \Omega_{j} = |k_{0}|^\sigma- |k_{0} -h|^\sigma + |k_{q+1}- h|^\sigma-|k_{q+1}|^\sigma.\\
   \Gamma(t)=\Gamma(k_0,t)- \Gamma(k_0 -h,t)+ \Gamma(k_{q+1}- h,t)- \Gamma(k_{q+1},t).
    \end{align*}
   
    For all $\mathcal Q'\equiv \mathcal{Q}$, the difference between  $\mathcal K_{\mathcal Q}$ and $\mathcal K_{\mathcal Q'}$ (run over all $\mathcal{Q}$) can be identified as:
 \begin{align*}
    \bigg(\frac{i\alpha T}{L}\bigg)^q&\sum\limits_{k_i \in \mathbb{Z}_L}  
    \prod_{i=0}^q |k_i|^\beta|k_{i+1}|^\beta|l_i|^\beta|l_{i+1}|^\beta\ \prod_{i=1}^q \varphi_{\leq K}({l_i})\varphi_{\leq K}({k_i})\notag\\ 
    &\prod_{i = 1}^q [\varphi_{\leq K}(l_i){n_{{\rm in},\eta}}(k_i)-\varphi_{\leq K}(k_i){n_{{\rm in},\eta}}(l_i)]
\int_{t_{q+1} < t_q < \ldots < t_1 < t_0}\prod_{i=0}^q  e^{\zeta_{i}2\pi i\Theta_i(t_i) } \dif t_1\dif t_2...\dif t_q,
 \end{align*}
    The resulting expressions are derived as follows:
    \begin{align*} 
    & 
    |k_0|^\beta|k_{q+1}|^\beta|l_0|^\beta|l_{q+1}|^\beta e^{ 2\pi i \Theta(t_0)} \bigg(\frac{i\alpha T}{L}\bigg)^q\sum\limits_{k_i \in \mathbb{Z}_L}
    \prod_{i=1}^q |k_i|^{2\beta}|l_i|^{2\beta}\prod_{i=1}^q \varphi_{\leq K}({l_i})\varphi_{\leq K}({k_i}) \\
     &\times\prod_{i = 1}^q [\varphi_{\leq K}(l_i){n_{{\rm in},\eta}}(k_i)-\varphi_{\leq K}(k_i){n_{{\rm in},\eta}}(l_i)]
\int_{t_{q+1} < t_q < \ldots < t_1 < t_0} \prod_{j = 1}^q e^{2\pi i( \Theta_{j}(t_{j})-\Theta_j(t_{0}))} \dif t_1\dif t_2...\dif t_q\\
    & = 
    |k_0|^\beta|k_{q+1}|^\beta|l_0|^\beta|l_{q+1}|^\beta e^{ 2\pi i \Theta(t_0)}  \int_{\tau\in [0,t_0-t_{q+1}]}P_q(\tau, t_0, k[\mathcal N^{sp}]) \dif \tau,
    \end{align*}
  and
    \begin{align*}
    P_q&(\tau, t_0, k[\mathcal N^{sp}]) :=\bigg(\frac{i\alpha T}{L}\bigg)^q\sum\limits_{k_i \in \mathbb{Z}_L} 
    \prod_{i=1}^q |k_i|^{2\beta}|l_i|^{2\beta}\prod_{i=1}^q \varphi_{\leq K}({l_i})\varphi_{\leq K}({k_i}) \notag\\
     &\times\prod_{i = 1}^q [{\varphi_{\leq K}}(l_i){n_{{\rm in},\eta}}(k_i)-{\varphi_{\leq K}}(k_i){n_{{\rm in},\eta}}(l_i)]
\notag\\
 &\quad\times \int_{0<s_1<\cdots<s_{q-1}<\tau}
 \prod_{j=1}^q e^{2\pi i(\Theta_j(t_0-s_j)-\Theta_j(t_0))}
 \,\dif s_1\cdots\dif s_{q-1}.
    \end{align*}  {Here $s_q:=\tau$.} Using the  decaying  of $  {n_{{\rm in},\eta}}$,   we have for the SG case where $|h|\leq g$,
    \begin{align*}
    \frac{\alpha T}{L} \sum_{k \in \mathbb{Z}_L} \left| \varphi(k){n_{{\rm in},\eta}}(k - h)-\varphi(k-h){n_{{\rm in},\eta}}(k)\right| &\lesssim \alpha  T g.
    \end{align*}
 Together with the fact that
    \begin{align*}
        |k_i|^{2\beta}|l_i|^{2\beta} \varphi_{\leq K}({l_i})\varphi_{\leq K}({k_i})\lesssim L^{\delta}, 
    \end{align*} we imply \eqref{eq-pestimate}. 
   Now we could write $ \mathcal{G}_\mathcal{Q}(t,s,k)$ as desired expression by noticing that $\tau=t_{\mathfrak{n}_0}-(t_{\mathfrak e}\vee t_{\mathfrak f} )\in(0,1]$. 
 
For the $t_0$ derivative, the static part of
$\Theta_j(t_0-s_j)-\Theta_j(t_0)$ is independent of $t_0$, and we only have the difference of $\partial_t\Gamma_1$ by \eqref{eq-pestimate}. 
\end{proof}
 
\subsection{The choice of irregular chains}
The remaining question is how to select such irregular chains and apply the splicing procedure to them. In general, we cannot simply splice all chains, as this may generate new irregular chains or additional terms, which in turn lead to unfavorable counting estimates. 
\begin{definition}
Consider  an enhanced couple $\mathcal{Q}$, and the corresponding molecule $\mathbb{M}(\mathcal{Q})$, we assume that there are two atoms, $v_1$ and $v_2$, connected by  two bonds. Then we denote by $\mathfrak{n}_j = \mathfrak{n}(v_j)$ for $j=1,2$.  Then we call the double bound is 

\begin{enumerate} 
\item type CL if the two bonds consist of one LP bond and one PC bond.
\item  type CN if both bonds are LP.  Pairs inside a triple bond are also included.
\end{enumerate}
According to different choice of paired leaves between $\mathfrak{n}_j$, the bonds can be the opposite directions or same directions.
\end{definition}

\begin{definition}
\begin{enumerate}
\item (Double Chains)
Consider an enhanced couple $\mathcal{Q}$ and the associated molecule $\mathbb{M}(\mathcal{Q})$. A sequence of nodes $\mathbb{H}=(v_0,\ldots,v_q)$ is called a double chain if, for each $0\leq i\leq q-1$, the pair $(v_i,v_{i+1})$ is connected by two bonds.
Given a molecule $\mathbb{M}(\mathcal{Q})$ and a chain $\mathbb{H}$, there is at most one CN double bond; see, for instance, \cite[Proposition 5.7]{DH23a}.
We further call $\mathbb{H}$ a hyperchain if $v_0$ and $v_q$ are connected by a single bond, and a  pseudo-hyperchain if $v_0$ and $v_q$ are connected through another common atom $v\notin\{v_0,\ldots,v_q\}$.
 
\item  (Irregular chain) A double chain is called an irregular chain if all its bonds are CL and are oriented in opposite directions. Similarly, we define irregular hyperchains and irregular pseudo-hyperchains in the same way.
\end{enumerate}
\end{definition}
In the following, we collectively refer to these structures as (irregular) chain-like objects. For clarity, the term “chain” will be reserved for chains in the strict sense, excluding hyperchains and pseudo-hyperchains.

In Definition \ref{def:chain-couple}, we also introduced irregular chains for an enhanced couple. In fact, given a couple $\mathcal{Q}$ and the corresponding molecule $\mathbb{M}(\mathcal{Q})$, each irregular chain in $\mathcal{Q}$ corresponds to an irregular chain in $\mathbb{M}(\mathcal{Q})$, and conversely, each irregular chain in $\mathbb{M}(\mathcal{Q})$ corresponds to a unique irregular chain in $\mathcal{Q}$.
Moreover, for any irregular chain $\mathcal{H}$ and the associated double-bond chain $\mathbb{H}$ in $\mathbb{M}$, if $\mathcal{Q}^{sp}$ denotes the couple after splicing, then the corresponding molecule $\mathbb{M}^{sp}=\mathbb{M}(\mathcal{Q}^{sp})$ is obtained from $\mathbb{M}$ by merging all atoms in $\mathbb{H}$ into a single atom. We refer to this procedure as splicing at the level of molecules.

 Consider an enhanced couple $\mathcal Q$ and the associated enhanced molecule $\mathbb{M}(\mathcal{Q})$. Let $\mathscr C$ denote the collection of all maximal small-gap (SG) irregular chain-like objects. Here, “maximal” means that the chain cannot be extended further by adding adjacent nodes. In particular, these objects are pairwise disjoint.
 
We then apply the splicing procedure to the irregular chains in the following manner:

\medskip

\boxed{\rm Operation\ Splicing}
We consider each $\mathcal{C}\in \mathscr C$ and proceed as follows:
\begin{enumerate}
\item If $\mathcal{C}$ is a chain, we apply the splicing procedure directly to the corresponding irregular chain.
\item If $\mathcal{C}$ is a hyperchain or a pseudo-hyperchain, we select a middle double bond. This decomposes $\mathcal{C}$ into two chains (possibly consisting of a single node) on each side of the middle node, and we apply splicing to each of these two chains.

After this procedure, in the case of a hyperchain, the remaining connection becomes a triple bond, while in the case of a pseudo-hyperchain, it reduces to a single double-bond pseudo-hyperchain. 
\end{enumerate}

\medskip
We note that no new double bonds are created during the splicing procedure.

\begin{remark}
  We remark that when applying the splicing procedure to a chain, its type should be determined with respect to the current molecule, rather than the original one, since splicing one chain may change the type of other chains. We refer to \cite{Vas24,Wu25} for a detailed discussion.
In particular, wide ladder structures may arise in the molecule; see, for example, Figure \ref{fig:ladder}. 
\begin{figure}
\begin{tikzpicture}[scale = .9]
    \tikzstyle{every node} = [circle, scale = .5,  draw = black]
    \node (1) at (-5,0) {};
    \node (2) at (-4,0) {};
    \node (3) at (-3,0) {};
    \node (4) at (-2,0) {};
    \node (5) at (-1,0) {};
    \node (6) at (-5,1) {};
    \node (7) at (-4,1) {};
    \node (8) at (-3,1) {};
    \node (9) at (-2,1) {};
    \node (10) at (-1,1) {};

    \draw[-,  thick] (1.25) -- (2.155);
    \draw[-, thick] (2.205) -- (1.335);
    \draw[-,  thick] (2.25) -- (3.155);
    \draw[-, thick] (3.205) -- (2.335);
    \draw[-,  thick] (3.25) -- (4.155);
    \draw[-, thick] (4.205) -- (3.335);
    \draw[-,  thick] (4.25) -- (5.155);
    \draw[-, thick] (5.205) -- (4.335);

    \draw[-,  thick] (1) -- (6);
    \draw[-,  thick] (5) -- (10);
    
     \draw[-,  thick] (6.25) -- (7.155);
    \draw[-, thick] (7.205) -- (6.335);
    \draw[-,  thick] (7.25) -- (8.155);
    \draw[-, thick] (8.205) -- (7.335);
    \draw[-,  thick] (8.25) -- (9.155);
    \draw[-, thick] (9.205) -- (8.335);
    \draw[-,  thick] (9.25) -- (10.155);
    \draw[-, thick] (10.205) -- (9.335);

\end{tikzpicture}
\hspace{1cm}
\begin{tikzpicture}[scale = .9]
    \tikzstyle{every node} = [circle, scale = .5,  draw = black]
    \node (1) at (-5,0) {};
    \node (2) at (-4,0) {};
    \node (3) at (-3,0) {};
    \node (4) at (-2,0) {};
    \node (5) at (-1,0) {};
    \node (6) at (-3,1) {};

    \draw[-,  thick] (1.25) -- (2.155);
    \draw[-, thick] (2.205) -- (1.335);
    \draw[-,  thick] (2.25) -- (3.155);
    \draw[-, thick] (3.205) -- (2.335);
    \draw[-,  thick] (3.25) -- (4.155);
    \draw[-, thick] (4.205) -- (3.335);
    \draw[-,  thick] (4.25) -- (5.155);
    \draw[-, thick] (5.205) -- (4.335);

    \draw[-,  thick] (1) -- (6);
    \draw[-,  thick] (6) -- (5);

\end{tikzpicture}
\caption{In such a wide ladder configuration, applying the splicing operation to the first chain may cause the second chain to transform from an ordinary chain into a pseudo-hyperchain.}
\label{fig:ladder}
\end{figure}
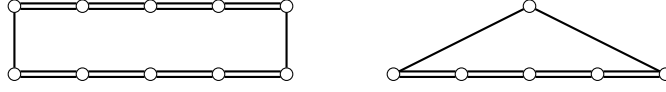
\end{remark}
After the splicing operation, the molecule $\mathbb{M}(\mathcal{Q}^{sp})$ remains connected. By the construction of the irregular chains, the only small-gap (SG) irregular chains that remain have length one and occur only within pseudo-hyperchains.
We also note that there may exist SG CN double chains with opposite orientations; however, these are not irregular chains.

\begin{lemma} \label{cor-M}
For an enhanced couple $\mathcal{Q}$, we apply splicing to all the  chain-like objects as above, and  consider the resulting couple $\mathcal Q^{sp}$. Then we have
\begin{align*}
    \mathcal{G}_\mathcal{Q}(t,s,k) :&= \bigg(\frac{\alpha T}{L}\bigg)^{n_{sp}} \zeta(\mathcal{Q}^{sp}) \sum_{\mathcal{E}^{sp}}  
    \prod_{\mathfrak{n}\in\mathcal{N}^{sp}}B_\mathfrak{n}\prod_{\mathfrak{n}\in\mathcal{Q}^{sp}}\varphi_\mathfrak{n}
\int_{[0,1]^{\mathcal{N}_0^{sp}}}\int_{\mathcal{E}^{sp}_\tau} \prod_{\mathfrak{n}_0\in \mathcal{N}_0^{sp}}P_{q_{\mathfrak{n}_0}}(\tau_{\mathfrak{n}_0}, t_{\mathfrak n_0}, k[\mathcal N^{sp}])\\
&\quad\quad\quad\quad\quad\quad\quad\quad\times\prod_{\mathfrak{n}\in\mathcal{N}^{sp}}  e^{\zeta_\mathfrak{n}2\pi i\Theta_\mathfrak{n}(t_\mathfrak{n}) } \dif t_\mathfrak{n}\prod_{\mathfrak{n}_0\in \mathcal{N}_0^{sp}}\dif \tau_{\mathfrak{n}_0}
\prod_{\mathfrak{l} \in \mathcal{L}^{sp}}^+ {n_{{\rm in},\eta}}(k_l),
\end{align*}

In the above expression,  $\mathcal N_0^{sp}$ are the nodes at which an irregular chain was spliced out and $P_{q_{\mathfrak n_0}}$ is given in Lemma \ref{lem-splice} with $q$ replaced by $q_{\mathfrak n_0}$, which is the length or irregular chain spliced out below $\mathfrak n_0$.    The time internal is defined as 
  $$\mathcal{E}^{sp}_{\pmb{\tau}} = \mathcal E \cap \{t_{\mathfrak n_0} \geq t_{\mathfrak n_0^{2}} + \tau_{\mathfrak n_0}, t_{\mathfrak n_0^{3}} + \tau_{\mathfrak n_0}\}_{\mathfrak n_0 \in \mathcal N_0^{sp}}.$$
  Furthermore, we allow degeneracies at each $\mathfrak{n}_0\in\mathcal{N}_0^{sp}$, while the corresponding subtrees may admit complete pairing. We refer to such nodes as bad degenerate nodes.
\end{lemma}
\begin{proof}
   This Lemma is  a direct consequence of Lemma  \ref{lem-splice} by taking all the chains into consideration.
\end{proof}

\section{The Algorithm on Molecules}\label{sec:Algorithm}

In this section, we consider the couple $\mathcal{Q}$ and the associated molecules $\mathbb{M}(\mathcal{Q})$. 
We first apply the splicing operation to the irregular chains, as in Lemma~\ref{cor-M}, 
to obtain $\mathcal{Q}^{sp}$. 
We then apply a refined algorithm introduced in \cite[Section 5]{DIP25} and \cite[Section~6]{VW26} to the molecule $\mathbb{M}(\mathcal{Q}^{sp})$.
This procedure allows us to control the number of decorations of a molecule in a step-by-step manner. 
Suppose that at some stage of the algorithm we have a molecule $\mathbb{M}^{pre}$, 
and after performing one operation we obtain $\mathbb{M}^{post}$. 
Let $\mathfrak{D}^{pre}$ and $\mathfrak{D}^{post}$ denote the corresponding numbers of decorations. 
We associate to this operation a counting factor $\mathfrak{C}$ such that
\begin{align*}
    \mathfrak{D}^{pre} \leq \mathfrak{C}\,\mathfrak{D}^{post}.
\end{align*}

For later use, we write $\overline\alpha:=T^{-3/4}L^{-2\delta}$, and then the gap $g=\overline \alpha L^{2\delta}$. By doing this, we have 
$$\frac{\alpha T}{L}= {\alpha}{T^{3/4}L^{2\delta}}\frac{\overline \alpha T}{L}\leq L^{-13\delta}\frac{\overline \alpha T}{L}$$
by choosing $\alpha T^{3/4}L^{15\delta}<1$. By doing this, we obtain an extra smallness. 
Then, in this section, for a molecule $\mathbb M$, we consider instead the following quantity:
 \begin{equation*} 
 \mathfrak U(\mathbb M)
   =\left(\frac{\overline \alpha T}{L}\right)^{\chi(\mathbb M)}
      \mathfrak D(\mathbb M),\qquad \chi=E-V+F.
\end{equation*}
For counting multiplicity $\mathfrak C$, write
$\Delta\chi=\chi_{\rm post}-\chi_{\rm pre}$.  Its deviation is
$ (\overline \alpha  T/L)^{-\Delta\chi}\mathfrak C$.

\subsection{Cutting operation}

\begin{definition}[Cutting an atom] 
Let $v$ be a atom with degree three or four. We replace $v$ by two atoms,
attaching a chosen pair of bonds to the first and the remaining
bonds to the second. If this cut  preserves the number of connected
components, we call it an $\alpha$-cut. For an  $\alpha$-cut,  the chosen pair must have opposite directions.  If this cut  increases that number of connected
components, we call it a $\beta$-cut.
The new atoms are called $\alpha$-atoms or $\beta$-atoms accordingly.
\end{definition}

Starting from $\mathbb M^{sp}$, we perform the following steps:
\begin{enumerate}[label=Step \arabic*:]
\item \emph{SG double bond.} We remove each SG bond  and call
its two remained endpoints $\alpha$-atoms.  We note that this
operation creates no new components, so we can call them $\alpha$-atoms.
\item \emph{SG, non-tame atoms.}
For an SG non-tame  atom of degree four, and $\sigma\in(0,1)$,  we assume this decorations are $k_1-k_2=k_4-k_3=h,|h|\leq g$, we choose $(x,x-h)\in\{(k_1,k_2),(k_4,k_3)\}$ containing the largest absolute frequency $\max\{|x|,|x-h|\}=\max\{|k_i|\}$. We define $\Phi_{\mathrm{hi}}:=f(x)-f(x-h)$, and $f(k):=|k|^\sigma+\Gamma_0(k)$. If additionally 
$$|\Phi_{\mathrm{hi}}|=|f(x)-f(x-h)|\le g,$$
we cut it at this atom.  
For an SG non-tame  atom in any other case,  we directly cut it retaining the SG gap.
\item \emph{Remaining $\beta$-cuts.}
For any remaining atom $v$ which is not tame,
perform a $\beta$-cut if it is possible.
\end{enumerate}
Only original atoms are cut, and the last stage is continued until
no such separating cut is possible.

\begin{proposition}[Cost of cutting]\label{prop:cutting1}
For any one of the cutting operations,
\begin{equation*} 
 \mathfrak U_{\rm pre}
 \lesssim  \mathfrak U_{\rm post}\, L^{\delta\Delta\chi}
 \prod_{\text{new }\alpha\text{-atoms}}d_\alpha\,
 d_\beta^{\Delta V_\beta-2\Delta F}.
\end{equation*}
Here we define \begin{equation*} 
 d_{\alpha,2}(\lambda)=\overline\alpha T\sqrt\lambda D L^{2\delta},
 \qquad d_{\alpha,1}=\overline\alpha\sqrt T DL^{2\delta}\log L,
 \qquad d_\beta=\overline\alpha\sqrt T L^{2\delta}.
\end{equation*}
If $\alpha$ atom $v$ has degree $i$, we choose $d_\alpha=d_{\alpha,i}$.
\end{proposition}
\begin{proof}
For Step~1, by deleting the pair, it gives $\Delta E=-2$ and
$\Delta V=\Delta F=0$.  If both endpoints had degree four,
one frequency has $O(L)$ choices and the gap has $O(L\lambda)$
choices.  The deviation is therefore
\[
 \left(\frac{\overline\alpha T}{L}\right)^2  L^2\lambda
 \lesssim (\overline\alpha T\sqrt\lambda)^2 \lesssim (\overline\alpha T\sqrt\lambda L^{2\delta})^2L^{-2\delta}.
\]
If an endpoint had degree three, the three-vector estimate  gives $O( L^2T^{-1}D^2\log L)$ choices, then we have 
\[
 \left(\frac{\overline\alpha T}{L}\right)^2  L^2T^{-1}D^2\log L\lesssim (\overline\alpha T^{1/2}D L^{2\delta}\log L)^2L^{-2\delta}.
\]
By  interpolation, it is also bounded by $Cd_{\alpha,2}d_{\alpha,1}$. 

In step 2, an $\alpha$-cut has
$(\Delta E,\Delta V,\Delta F)=(0,1,0)$ and hence
$\Delta\chi=-1$.    It suffices to fix the values $k_{v_1}$, $\alpha_{v_1}$, where $v_1$ is one of the newly created atoms
resulting from cutting at $v$.    If $v$ is degree 4, by
the SG assumption, there are $L\lambda$ choices of $k_{v_1}$.  For the choice of $\alpha_{v_1}$, in the case $\sigma>1$,   by  \[
|f(x)-f(x-h)|
\lesssim D^{\sigma-1}|h|+\alpha |h|^{2\beta},
\] we have $O(T\alpha\lambda^{2\beta}+TD^{\sigma-1}\lambda)$ choices of $\alpha_{v_1}$. 
In the case $\sigma<1$, by step 2, we have $O(Tg)$ choices of $\alpha_{v_1}$. 
So, we have 
\begin{align*}
     \frac{\overline\alpha T}{L} \cdot L\lambda \cdot Tg \lesssim (\overline\alpha T\sqrt\lambda D L^{2\delta})^2L^{-\delta},\\
     \frac{\overline\alpha T}{L} \cdot L\lambda \cdot (T\alpha\lambda^{2\beta}+TD^{\sigma-1}\lambda) \lesssim (\overline\alpha T\sqrt\lambda D L^{2\delta})^2L^{-\delta},
\end{align*} 
where we used $\lambda\leq g\leq \overline\alpha L^{2\delta}\leq1$.

At degree three, specifying the single remaining bond fixes its
phase as well. So we have
\[
   \frac{\overline\alpha T}{L} \cdot L\lambda\le (\overline\alpha T\sqrt\lambda L^{2\delta})(\overline\alpha \sqrt TD L^{2\delta})L^{-\delta},
\]
where we used $\lambda\leq \overline\alpha L^{2\delta} \leq \overline\alpha^2 TL^{2\delta}$.

Finally, for a $\beta$-cut, we have
$(\Delta E,\Delta V,\Delta F)=(0,1,1)$ and
$\Delta V_\beta=2$,  $\Delta\chi=\Delta V_\beta-2\Delta F=0$. Note that the values for $k_{v_1}$, $k_{v_2}$, $\alpha_{v_1}$ and $\alpha_{v_2}$
will be fixed by adding the equations
for all atoms in one of the components.
\end{proof}
  
\begin{proposition}[The reduced graph] 
Let $\mathbb M^{\mathrm{fin}}$ be the reduced molecule
obtained after all cutting steps. We call the remaining
original atoms $\epsilon$-atoms. Then the following
properties hold.
\begin{enumerate}[label=(\arabic*)]

\item Either $\mathbb M^{\mathrm{fin}}$ has only one
connected component, or every component contains
at least one $\beta$-atom.
 Moreover, if $F$ is the number of
components, then $V_\beta=2(F-1)$.

\item There is at most one component containing
atoms of odd degree, called the odd component,
which contains exactly two atoms of odd degree
(either one or three). All other components,
called even components, contain only atoms of
even degree. 

\item Every $\epsilon$-atom is either tame or
satisfies the following condition:
if $1<\sigma\le2$, it is LG; if $0<\sigma<1$,
then every admissible opposite partition of a
degree-four non-tame $\epsilon$-atom satisfies
\[
 |h|>g\qquad\text{or}\qquad
 |\Phi_{\mathrm{hi}}|>g.
\]
A surviving non-tame degree-three $\epsilon$-atom
is LG in both regimes. No non-tame $\epsilon$-atom
admits a separating $\beta$-cut.

\item If a component contains an $\epsilon$-atom,
the $\alpha$- and $\beta$-atoms form chains.
Each maximal such chain has
endpoints at $\epsilon$-atoms or degree-one cut
atoms. If both endpoints are the same
$\epsilon$-atom, that atom must be tame.
If a component contains no $\epsilon$-atoms,
it is either a cycle which is not an SG
opposite double bond or a self-loop, or the
odd component whose two odd-degree atoms
both have degree one.
Furthermore, $\mathbb M^{\mathrm{fin}}$
contains no self-loops.

\end{enumerate}
\end{proposition} 
 \begin{proof}
See \cite[Proposition~9.5]{DH23a} and
\cite[Proposition~5.22]{DIP25} for the proof.
 For $0<\sigma<1$, the additional phase condition
only modifies the characterization of uncut non-tame atoms:
$|h|>g$ or $|\Phi_{\mathrm{hi}}|>g$;
the combinatorial argument remains unchanged.
\end{proof}

\begin{proposition}[$\epsilon$-atom removal]
\label{prop:e-removal}
Let $\mathbb M'$ be a reduced component and let $v$ be an
$\epsilon$-atom.  Removing $v$ and its incident bonds gives
\[
 \mathfrak U_{\mathrm{pre}}
 \lesssim \mathfrak U_{\mathrm{post}}\,
L^{\delta \Delta\chi}.
\]
This estimate remains valid during the subsequent removals. 
\end{proposition}

\begin{proof}
Let $q=\Delta F+ 1$ be the number of components
remaining after the removal.  Write $A_j$ for the bonds attaching
$v$ to the $j$th component and $p_j=|A_j|$, ordered so that
$p_1\ge\cdots\ge p_q$.  Then
$\Delta\chi=q-\deg(v)$. 
If $\Delta\chi=0$, every part is a singleton and these equations
fix all incident frequencies.  We therefore only consider
$\deg(v)>1$ and partitions other than $(1,\ldots,1)$.

For $\deg(v)=2$, the only partition is $(2)$, with
$\Delta\chi=-1$.  Its deviation $\mathfrak D_2$ is estimated
according to the directions and the type of $v$:
\begin{itemize}
\item For two bonds with the same direction, Lemma~\ref{counting} gives
\[
 \mathfrak D_2\lesssim
  \frac{\overline \alpha T}{L}(LT^{-1/2}D)
 =\overline  \alpha\sqrt T D\lesssim L^{-\delta}.
\]
\item If $v$ is tame, the small-gap condition gives
\[
 \mathfrak D_2\lesssim\frac{\overline \alpha T}{L}(Lg)=\overline \alpha ^2TL^{2\delta}\le L^{-\delta}.
\] 
\item  For  the case that $v$ is LG,   Lemma~\ref{counting} gives
\[
  \mathfrak D_2\lesssim \frac{\overline \alpha T}{L}(1+LT^{-1/2}
                  + \frac{LD^{2-\sigma}}{T|h|}) \lesssim \overline \alpha T^{1/2}+D^2L^{-2\delta}\lesssim  L^{-\delta}.
\]
\item If \(|h|\le g\), but \(|\Phi_{\mathrm{hi}}|>g\), by the choice of $x$, we have $\max\{|x|,|x-h|\}=\max\{|k_i|\}$. Then, we have
$$g<|\Phi_{\mathrm{hi}}|\lesssim |h|\max\{|x|,|x-h|\}^{\sigma-1}+\alpha |h|^{2\beta}\lesssim  |h|\max\{|x|,|x-h|\}^{\sigma-1}+o(g).$$
By applying  Lemma~\ref{counting} here with $D=\max\{|x|,|x-h|\}$ we have 
\begin{align*}
    \mathfrak D_2&\lesssim \frac{\overline \alpha T}{L}(1+LT^{-1/2}
                  + \frac{L\max\{|x|,|x-h|\}^{2-\sigma}}{T|h|})\\
                  &\lesssim \overline \alpha T^{1/2}+\frac{\overline \alpha}{h}\max\{|x|,|x-h|\}^{2-\sigma} \lesssim  L^{-\delta}\\
                  &\lesssim\overline  \alpha T^{1/2}+\frac{\overline \alpha}{g}  \max\{|x|,|x-h|\} \lesssim L^{-2\delta}(1+D)\lesssim L^{-\delta}.
\end{align*}

\end{itemize}

For $\deg(v)=3$, there are two possibilities:
\begin{itemize}
\item $(2,1)$: here $\Delta\chi={-1}$ and
$\mathfrak D\lesssim\mathfrak D_2$.
\item $(3)$: here $\Delta\chi=-2$, and Lemma~\ref{counting} yields
\[
 \mathfrak D_3\lesssim
 \left(\frac{\overline \alpha T}{L}\right)^2
 L^2{T}^{-1}D^2\log L
 =\overline  \alpha ^2TD^2\log L .
\]
\end{itemize}

For $\deg(v)=4$, the possible partitions are:
\begin{itemize}
\item $(2,1,1)$: $\Delta\chi=-1$ and
$\mathfrak D\lesssim\mathfrak D_2$.
\item $(2,2)$: $\Delta\chi=-2$ and
$\mathfrak D\lesssim\mathfrak D_2^2$
 when $v$ is non-tame.  For a tame atom, the two
component momentum equations leave two free frequencies whose
difference is $O(g)$.  Their joint count is $O(L^2g)$, giving
\[
\mathfrak D\lesssim
 \left(\frac{\overline \alpha T }{L}\right)^2L^2g=\overline \alpha ^3T^2L^{2\delta}.
\]
\item $(3,1)$: $\Delta\chi=-2$ and
$\mathfrak D\lesssim\mathfrak D_3$.
\item $(4)$: $\Delta\chi=-3$, and Lemma~\ref{counting} gives
\[
 \mathfrak D_4\lesssim
 \left(\frac{\overline \alpha T}{L}\right)^3
 \frac{ L^3}{T}D^2\log L 
 = \overline \alpha ^3T^2D^2\log L\lesssim L^{-3\delta} .
\]
\end{itemize}  The degree-three improvement is the $(3)$ case.
 For the degree-four improvement, $\Delta F\le1$
leaves $(4)$, $(3,1)$ or $(2,2)$; the last is excluded at a
non-tame atom admitting no separating cut. 
\end{proof}

For each component which is not a cycle, use the following
operations in the stated priority order:
\begin{enumerate} 
\item Remove a degree-one $\alpha$- or $\beta$-atom.
\item Remove an $\epsilon$-atom of odd degree.
\item Remove an $\epsilon$-atom satisfying $\deg(v)=\deg_0(v)$.
\item Remove any remaining $\epsilon$-atom.
\end{enumerate}
Here $\deg_0(v)$ is the degree immediately after cutting.
Isolated atoms and Operation~1 incur no cost. 

\begin{proposition}[Counting estimates for reduced molecules]
\label{prop:cutting2}
For a reduced component $\mathbb M'$ containing $v_\beta$
$\beta$-atoms, the following estimates hold:
 \begin{align*}
 \mathfrak U(\mathbb M')
 &\lesssim \overline \alpha^2TL^{3\delta} (L^{-\delta})^{\chi-2} 
 \left(\prod_{\alpha\in\mathbb M'}d_\alpha\right)^{-1}
 d_\beta^{-V_\beta},
 &&\mathbb M'\text{ odd},\\
 \mathfrak U(\mathbb M')
 &\lesssim (L^{-\delta})^{\chi } 
 \left(\prod_{\alpha\in\mathbb M'}d_\alpha\right)^{-1}
 d_\beta^{2-V_\beta},
 &&\mathbb M'\text{ even}.
 \end{align*}
\end{proposition}

\begin{proof}
\emph{Case 1: $\mathbb M'$ is a double bond.}
Then $\chi=1$, $V_\alpha=0$, $V_\beta=2$, and the component is even.
For opposite directions, the gap $\lambda$ satisfies $\lambda>g$,
and the two-vector estimate in Lemma \ref{counting}  gives
\[
 \mathfrak U\lesssim \frac{\overline \alpha T}{L}
 \left( {LT^{-1/2}+}\frac{L}{T\lambda}\right) D^2\lesssim \overline \alpha T^{1/2}+L^{-2\delta}D^2\lesssim L^{-\delta}.
\] 

\emph{Case 2: $\mathbb M'$ is a larger cycle.}
The cycle has at least three atoms, $\Delta\chi=1$, and only even degrees.
If $V_\beta\ge3$, the trivial estimate is
\[
 \mathfrak U\lesssim\frac{\overline \alpha T}{L}L
 =\overline \alpha T\lesssim L^{-\delta}(\overline \alpha \sqrt TL^{\delta})^{-1},
\] 
where we used $\overline \alpha^2T^{3/2}\leq L^{-2\delta}$.

If $V_\beta=2$, there is at least one $\alpha$-atom, with recorded
gap $\lambda$.  Combining the trivial and two-vector bounds gives
\begin{align*}
    \mathfrak U
 &\lesssim \frac{\overline \alpha T}{L}
 \min\{L, LT^{-1/2}+ \frac{L}{T \lambda}\}D^2\notag\\
 &\lesssim \frac{\overline \alpha T}{L}(LT^{-1/4}+L(T\lambda)^{-1/2})D^2\notag\\
 &\lesssim   \overline \alpha  T^{3/4}D^2+\overline \alpha T^{1/2} \lambda^{-1/2} D^2\lesssim L^{-\delta} (\overline \alpha TL^{\delta}\lambda^{1/2}D)^{-1},\notag
\end{align*} 
where we used $\overline \alpha^2T^{3/2}D^2\leq L^{-2\delta}$.

If $V_\beta=1$, there are at least two $\alpha$-atoms, with recorded
gaps $\lambda_1,\lambda_2$.  Using either two-vector constraint,
\[
 \begin{aligned}
 \mathfrak U
 &\lesssim\frac{\overline \alpha T}{L}
 \left({LT^{-1/2}+}
 \frac{L}{T}\min\{\lambda_1^{-1},\lambda_2^{-1}\}\right)D^2\\
 &\lesssim  \overline \alpha T^{1/2}D^2+\overline  \alpha
 \lambda_1^{-1/2}\lambda_2^{-1/2}D^2
 \lesssim  L^{-\delta}\frac{d_\beta}
 {d_{\alpha,2}(\lambda_1)d_{\alpha,2}(\lambda_2)},
 \end{aligned}
\]
where we used $\overline \alpha T^{3/4}D^2<L^{-\delta}$. 

\emph{Case 3: $\mathbb M'$ is a chain which is not a cycle.}
All its atoms are $\alpha$- or $\beta$-atoms, and $\chi=0$.
The endpoint momentum constraints determine its decoration, so
$\mathfrak U\lesssim1$.  This is the odd component, and each of
its two endpoint charges is $\overline \alpha \sqrt T DL^{2\delta}\log L$.  Their inverse charges
cancel the prefactor $\overline \alpha^2TL^{5\delta}$ in the expression of $\mathfrak U(\mathbb M')$.

\emph{Case 4: $\mathbb M'$ contains an $\epsilon$-atom.}
By Proposition~\ref{prop:e-removal}, every removal of an
$\epsilon$-atom has deviation at most $ L^{-\delta}$.
If the component is even, then $V_\beta\ge1$.
The first original atom removed has degree four and, by parity,
partition $(4)$  or tame $(2,2)$.
Proposition~\ref{prop:e-removal} therefore gives deviation
at most $ \overline \alpha ^3T^2(D^2\log L+L^{2\delta})\lesssim \overline \alpha T^{1/2}L^{2\delta}$.

For an odd component, an uncut degree-three source is removed first
whenever it remains available after Operation~1.
Its partition is $(3)$  or tame $(2,1)$.
Its counting at most $\overline \alpha^2T (D^2\log L+L^{2\delta})$. 
 
Multiplying the removal costs proves the result.
\end{proof}

\begin{proposition} 
 We consider a enhanced couple $\mathcal{Q}^{sp}$ as in Proposition \ref{cor-M}, and its corresponding molecule $\mathbb M = \mathbb M(\mathcal Q^{sp})$, with scale $n$. For a $(k_v,C_v)$-decoration as in Definition \ref{def:decorations-molecules} for some fixed $C_v$, we additionally that there are $k_v^0\in\mathbb{Z}_L$ satisfying $|k_v^0|\leq D$ for some $D>1$, and for any $v\in\mathbb{M}$, $$|k_v-k_v^0|\leq 1.$$
 Then, it holds that  
$$\left(\frac{ \alpha T}{L}\right)^m\mathfrak D(\mathbb M^{sp})
 \le C_N \alpha ^2TL^{3\delta} L^{-\delta (m-2)} (\frac{\alpha} {\overline\alpha})^{m-2}.$$
 
\end{proposition}

\begin{proof}
We apply Proposition~\ref{prop:cutting1} and 
Proposition~\ref{prop:cutting2} together. We note that the initial rank is $m$,
and all $\alpha$-charges cancel. For the $F$ many components after cutting,
there is one odd component and $F-1$ even components.  The total
exponent of $d_\beta$ is therefore
\[
 [V_\beta-2(F-1)]+[2(F-1)-V_\beta]=0.
\]
The odd component contributes $\overline \alpha ^2TL^{3\delta}$. 
\end{proof}

\subsection{Bounds on couples} \label{sec-boundoncouple}
In previous sections, we have establish the estimates on the counting problem. In this section, we are in position to prove  Proposition \ref{prop:couple}.
\begin{proof}[Proof of Proposition \ref{prop:couple}]
First, we consider to prove \eqref{bd:couple}, the left hand side equals to
     \begin{align*}
     \sum_{\mathcal{Q}_{n}}\mathbf{E}[ c^{\mathcal{T}^+}_k (s)c^{\mathcal{T}^-}_k(s)] = \sum_{|\mathcal{Q}|=n}(\mathcal{K}_\mathcal{Q})(t,s,k).
 \end{align*}
 where the sum on the right hand side is taken over all the enhanced couples with $n$ nodes. Now since $n\leq N^3$, we know that there are at most  $O(C^{N^3}(N^3)!)$ such kinds of couples. By taking into account the different ways of choosing degenerate nodes and pairing, there are also at most $O(C^{N^3}(N^3)!)$ ways.  Then by classify the couple by the congruent relation, we also have at most $O(C^{N^3}(N^3)!)$ ways. Since this number is independent of $L$, in the following we could fix one single couple $\mathcal{Q}$ with fixed decoration, and   then  consider its  congruent class and the expression $\mathcal{G}_{\mathcal{Q}}$, which is defined in Lemma \ref{lem-splice}.

In the expression in    Lemma \ref{lem-splice}, by the choice of $K$, the symbols of derivatives  are bounded by \begin{align*}
\prod_{\mathfrak{n}\in\mathcal{N}^{sp}}B_\mathfrak{n}\prod_{\mathfrak{n}\in\mathcal{Q}^{sp}}\varphi_\mathfrak{n}\leq L^{n_{sp}\delta/10}.
    \end{align*}
    Then, we aim to bound the oscillatory integral for the fixed $\tau_{\mathfrak{n}_0}$:
    \begin{align*}
      \int_{\mathcal{E}^{sp}_\tau} \prod_{\mathfrak{n}_0\in \mathcal{N}_0^{sp}}P_{q_{\mathfrak{n}_0}}(\tau_{\mathfrak{n}_0}, t_{\mathfrak n_0}, k[\mathcal N^{sp}])\prod_{\mathfrak{n}\in\mathcal{N}^{sp}}  e^{\zeta_\mathfrak{n}2\pi i\Theta_\mathfrak{n}(t_\mathfrak{n}) } \dif t_\mathfrak{n} .
    \end{align*}
Since the condition $\sup_k\sup_{t\in[0,1]}|\partial_t{\Gamma_1}(k,t)|\leq 1$ asked    in Proposition \ref{prop:est:af} is satisfied.
    We apply    Proposition \ref{prop:est:af} with \begin{align*}f_{\mathfrak{n}_0}(t_{\mathfrak n_0})=\left(\alpha T |g| L^\delta\right)^{-q_{\mathfrak n_0}}P_{q_{\mathfrak n_0}}(\tau_{\mathfrak n_0}, t_{\mathfrak n_0}, k[\mathcal N^{sp}]),\end{align*} which satisfies the assumptions by \eqref{eq-pestimate}. Then we obtain
    \begin{align*}
    &\left| \sum_{\mathcal{Q}_{n}}\mathbf{E}[ c^{\mathcal{T}^+}_k (t)c^{\mathcal{T}^-}_k(s)]\right| \notag\\
    &\lesssim 
      C^{n}\left( \frac{\alpha T}{L}\right)^{n_{{{sp}}}}(\alpha  T |g| )^{\sum q_{\mathfrak{n}_0}} \sum_{\mathcal{E}^{sp}}  
      L^{n_{sp}\delta}\notag\\
&\quad\quad\quad  \times
\sup_{\pmb \tau \in [0,1]^{\mathcal N_0^{sp}}}\bigg{|}\int_{\mathcal{E}^{sp}_\tau} \prod_{\mathfrak{n}_0\in \mathcal{N}_0^{sp}}f_{\mathfrak{n}_0}(t_{\mathfrak n_0})\prod_{\mathfrak{n}\in\mathcal{N}^{sp}}  e^{\zeta_\mathfrak{n}2\pi i[\overline\Omega_\mathfrak{n}Tt_\mathfrak{n}+\sum\Gamma_1(\mathfrak{n},t_\mathfrak{n})]  } \dif t_\mathfrak{n}\bigg{|}
\prod_{\mathfrak{l} \in \mathcal{L}^{sp}}^+ {n_{{\rm in},\eta}}(k_l).
\end{align*}
Then using the decay of the initial data ${n_{{\rm in},\eta}}$ and the truncation functions $\prod_{\mathfrak{n}\in\mathcal{Q}^{sp}}\varphi_\mathfrak{n}$, we could assume that $|k_l-k_l^0|\leq1$ for $k_l^0\in \mathbb{Z}^d$, and only consider the expression under the restrictions $|k_l^0|\leq L^{\delta}$. 
 Moreover, by rewrite ${n_{{\rm in},\eta}}(k_l)=\langle k_l\rangle^{-50}{\widetilde n_{{\rm in},\eta}}(k_l)$, we have the above expression (under the restriction  $|k_l-k_l^0|\leq1$ ) is bounded by

\begin{align}
 &\prod_{l\in\mathcal{L}^{sp}}\langle k_l\rangle^{-30} C^{n}\left( \frac{\alpha T}{L}\right)^{n_{{sp}}}(\alpha  T|g| )^{\sum q_{\mathfrak{n}_0}} \sum_{{\lambda[\mathcal N^{sp}]}}\sum_{\mathcal{E}_\lambda^{sp}}  
 L^{n_{sp}\delta} \notag\\
      & \quad \quad\times 
\sup_{\pmb \tau \in [0,1]^{\mathcal N_0^{sp}}} \left|\int_{\mathcal{E}^{sp}_{\pmb{\tau}}}\prod_{\mathfrak{n}_0 \in \mathcal{N}^{sp}_0}f_{\mathfrak{n}_0}(t_{\mathfrak n_0})\prod_{\mathfrak{n} \in \mathcal{N}^{sp}}e^{\zeta_\mathfrak{n}2\pi i[\overline\Omega_\mathfrak{n}Tt_\mathfrak{n}+\sum\Gamma_1(\mathfrak{n},t_\mathfrak{n})]  } \dif  t_{\mathfrak{n}}\right|\notag\\
      &\lesssim \langle k\rangle^{-20} C^{n}\left( \frac{\alpha T}{L}\right)^{n_{{{sp}}}}(\alpha  T|g|)^{\sum q_{\mathfrak{n}_0}} \sup_{\lambda[\mathcal N^{{sp}}]}\sum_{\mathcal E_{\lambda }^{{sp}}}
      L^{n_{sp}\delta}\notag\\
      & \quad \quad\times\sum_{{\lambda[\mathcal N^{sp}]}}\sup_{\pmb \tau \in [0,1]^{\mathcal N_0^{sp}}}\left|\int_{\mathcal{E}^{sp}_{\pmb{\tau}}}\prod_{\mathfrak{n}_0 \in \mathcal{N}^{sp}_0}f_{\mathfrak{n}_0}(t_{\mathfrak n_0})\prod_{\mathfrak{n} \in \mathcal{N}^{sp}}e^{\zeta_\mathfrak{n}2\pi i[\overline\Omega_\mathfrak{n}Tt_\mathfrak{n}+\sum\Gamma_1(\mathfrak{n},t_\mathfrak{n})]  } \dif t_{\mathfrak{n}}\right|,\label{eq-sp1}
    \end{align}
    where $\lambda[\mathcal N^{sp}] \in \mathbb{Z}^{|\mathcal N^{sp}|}$ and ${\mathcal E}_{\lambda}^{sp}$ is taken over all the decorations of the couple satisfying $|T\overline\Omega_{\mathfrak n} - \lambda_{\mathfrak n}| \leq 1$ for each $\mathfrak n \in \mathcal N^{sp}$. Moreover, it holds that $\sum q_{\mathfrak{n}_0}+ n_{sp} = n$. We apply Proposition \ref{prop:est:af} to bound the last line above by
    \begin{align*}
        &\sum_{{\lambda[\mathcal N^{sp}]}}\sup_{\pmb \tau \in [0,1]^{\mathcal N_0^{sp}}}\sum_{d_\mathfrak{n}\in\{0,1\}}\prod_{\mathfrak{n}\in\mathcal{N}^{sp}}\langle Tq_\mathfrak{n}\rangle^{-1}
        \lesssim C^{n_{sp}} (\log L)^{n_{sp}},\notag
    \end{align*}
    where we used the fact that $d_\mathfrak{n}$ only has two choices, and since $k_l\in\mathbb{Z}_L$ with $|k_l-k_l^0|\leq1$, there are at most $L^{4}$ choices of $\lambda_\mathfrak{n}$.

    To bound the  term 
      $  \sup_{\lambda[\mathcal N^{{sp}}]}\sum_{\mathcal E_{\lambda }^{{sp}}}
     1 ,$
    we transfer the couples into molecule as in Section \ref{sec:Algorithm}. {For $\sigma\in(1,2]$, choose $\delta>0$ small enough that $\alpha T^{\sigma/2-\beta}L^{2\delta}<1$. No such condition is needed for $\sigma\in(0,1)$. Applying Lemma \ref{counting} with $D\leq L^{\delta/10}$, we obtain that \eqref{eq-sp1} is bounded by}
    \begin{align*}
     \langle k\rangle^{-20} &C^{n} \alpha ^2TL^{3\delta} L^{-\delta (n-\sum q_{\mathfrak{n}_0}-2)} (\frac{\alpha} {\overline\alpha})^{n-\sum q_{\mathfrak{n}_0}-2}\cdot (\frac{\alpha} {\overline\alpha}\cdot\overline \alpha  T|g|)^{\sum q_{\mathfrak{n}_0}}\cdot (\log L)^nL^{n\delta},\notag\\
     &\lesssim \langle k\rangle^{-20} C^{n}\alpha ^2T  L^{7\delta-12\delta (n-2)} 
    \end{align*}
    which implies \eqref{bd:couple}.

   Then we prove the next statement.  We recall that by \eqref{def:gamma1}, \eqref{bd:couple}
 \begin{align*}
   | \mathscr{P}( \Gamma_1)|&=\left|\frac{\alpha T}{\pi L}\sum_{k_1}\sum_{\mathcal{Q}_{\leq N}}\mathbf{E}[ c^{\mathcal{T}^+}_{k_1} (t)c^{\mathcal{T}^-}_{k_1}(t)] ( \Gamma_1)B_{k_1,k_1,k,k}\varphi_{\leq K}(k)\right|\\
   &\lesssim \alpha^3 T^{2}L^{20\delta} \cdot L^{-1}\sum_{k_1\in\mathbb{Z}_L}\langle k_1\rangle^{-20}   \leq L^{-3\delta},
 \end{align*}
 where we used $\alpha^3 T^2 L^{25\delta} \leq 1$.  Here we also use the fact that the sum over $k_1$ is restricted to $|k_1| \lesssim L^\delta$, so the symbol term is bounded together with the
  truncation function $\varphi_{\leq K}(k)$.

In the end,   it remains to  bound the last term, namely,  $\mathscr{P}[\Gamma_1]-\mathscr{P}[{\Gamma}'_1]$, under the assumption on $\Gamma_1$ and $\Gamma_1'$. By a similar argument as before, we only need to handle the term 
\begin{align*}
    \mathcal{G}_Q[\Gamma_1](t,k)-    \mathcal{G}_Q[{\Gamma}'_1](t,k),
\end{align*}
 where $ \mathcal{G}_Q$ is defined in \eqref{def:gq}. From the definition, the only terms that involve the  $\Gamma_1$  are in the time integrals, so we
 only need to estimate the term 
 \begin{align}
    \int_{\mathcal{D}}\prod_{\mathfrak{n}\in\mathcal{N}}  e^{\zeta_\mathfrak{n}2\pi i\overline\Omega_\mathfrak{n}Tt_\mathfrak{n}}e^{\zeta_\mathfrak{n}2\pi i\sum\Gamma_1(\mathfrak{n},t_\mathfrak{n}) }-\prod_{\mathfrak{n}\in\mathcal{N}}  e^{\zeta_\mathfrak{n}2\pi i\overline\Omega_\mathfrak{n}Tt_\mathfrak{n}}e^{\zeta_\mathfrak{n}2\pi i\sum\Gamma'_1(\mathfrak{n},t_\mathfrak{n}) } \dif t_\mathfrak{n}.\label{bd:inge_diff}
 \end{align}
  By taking difference, it suffices to estimate the following expression: for $|\overline\Omega_kT|\geq 1$,
 
  \begin{align*}
        \int_0^t &  e^{ 2\pi i[\overline\Omega_kTs+\sum\Gamma_1(k,s)] } - e^{ 2\pi i[\overline\Omega_kTs+\sum \Gamma'_1(k,s)] } \dif s\\
        &= \frac{ e^{ 2\pi i[\overline\Omega_kTt+\sum\Gamma_1(k,t)] }-e^{ 2\pi i[\overline\Omega_kTt+\sum  \Gamma'_1(k,t)] }}{2\pi i\overline\Omega_kT}\\
        &\quad - \int_0^t \frac{  e^{ 2\pi i[\overline\Omega_kTs+\sum\Gamma_1(k,s)] }2\pi i\sum \partial_s\Gamma_1(k,s)- e^{ 2\pi i[\overline\Omega_kTs+\sum \Gamma'_1(k,s)] }2\pi i\sum  \partial_s\Gamma_1'(k,s)}{2\pi i\overline\Omega_kT}   \dif s.
   \end{align*}

Now we go back to \eqref{bd:inge_diff}, by taking difference, at a loss of $|\mathcal{Q}|=n\leq N^3$, we only to deal with a similar expression as \eqref{def:af} with $f=1$, and there is a node $\mathfrak{n}$, such that the  term in the product $e^{ 2\pi i \sum\Gamma_1(k,s) } $ is replaced $ e^{ 2\pi i \sum\Gamma_1(k,s)} - e^{ 2\pi i\sum \Gamma'_1(k,s)} $, and for the remaining node, this term is $\Gamma_1$ or $ \Gamma'_1$. Then since $$\sup_k\sup_{t\in[0,1]}|\partial_t\Gamma_1(k,t)|+|\partial_t\Gamma'_1(k,t)|\leq 1,$$ we  could estimate this expression by a similar way as Lemma \ref{prop:est:af} and we only need to deal with the integral at node $\mathfrak{n}$. From the above expression,  the first term equals to
   \begin{align*}
       \frac{ e^{ 2\pi i\overline\Omega_kTt}}{2\pi i\overline\Omega_kT}(e^{2\pi i\sum\Gamma_1(k,t) }-e^{ 2\pi i \sum  \Gamma'_1(k,t) }).
   \end{align*}  
   Then it contribution is $\frac{ e^{ 2\pi i\overline\Omega_kTt}}{2\pi i\overline\Omega_kT}$ with an extra term $(e^{2\pi i\sum\Gamma_1(k,t) }-e^{ 2\pi i \sum  \Gamma'_1(k,t) })$.  Then for this term, we obtain the same bound as in Lemma \ref{prop:est:af}, with extra  product of $|\partial_t\Gamma_1-\partial_t \Gamma'_1|$.
For the second term, it is bounded by 
   \begin{align*}
  \frac{1}{ \overline\Omega_kT}   \left(  \left | \sum \partial\Gamma_1-\sum\partial \Gamma_1' \right|+ \left|e^{ 2\pi i \sum\Gamma_1} -e^{ 2\pi i \sum \Gamma_1'}  \right|\right)\lesssim \frac{1}{ \langle\overline\Omega_kT\rangle}|\partial_t\Gamma_1-\partial_t\Gamma_1'|.
   \end{align*}
   For the case $|\overline\Omega_kT|< 1$, we could directly consider the difference $$|  e^{ 2\pi i\sum\Gamma_1(k,s) } - e^{ 2\pi i\sum \Gamma'_1(k,s) }|  \lesssim\frac{1}{ \langle\overline\Omega_kT\rangle} |  e^{ 2\pi i\sum\Gamma_1(k,s) } - e^{ 2\pi i\sum \Gamma'_1(k,s) }|,$$ which is bounded by the same argument as before.
   Then for this term, we also obtain the same bound as in Lemma \ref{prop:est:af}, with extra  product of $|\partial_t\Gamma_1-\partial_t \Gamma'_1|$. 
   Then  in summary we obtain that \eqref{bd:inge_diff} is bounded by 
   \begin{align*}
       N^3\cdot\sum_{d_\mathfrak{n}\in\{0,1\}}\prod_{\mathfrak{n}\in\mathcal{N}}C_\mathfrak{n}\langle Tq_\mathfrak{n}\rangle^{-1} \sup_k\sup_{t\in[0,1]}|\partial_t\Gamma_1-\partial_t \Gamma'_1|.
   \end{align*}
   Then by following the same calculation as before, the term $   \mathcal{G}_Q[\Gamma_1](t,k)-    \mathcal{G}_Q[{\Gamma}'_1](t,k)$ is bounded by  $$ \langle k\rangle^{-20} C^{n}\alpha ^2T L^{7\delta-12\delta (n-2)} (\log L)^n\sup_k\sup_{t\in[0,1]}|\partial_t\Gamma_1-\partial_t \Gamma'_1|.$$
   Taking sum on $n$ and $k$ we obtain the desired bound by a similar argument as before.
 
\end{proof}

\section{The Remainder Term}\label{bd:remainder}
In this section, we study the remainder term $\mathcal{R}^{N+1}$, which solves the following fixed point problem:
\begin{equation}\mathcal{R}^{N+1}=(1-\mathscr{L})^{-1}[c_{\sim N}+\mathscr{Q}(\tilde\varphi_{\leq K}\mathcal{R}^{N+1})+\mathscr{C}(\tilde\varphi_{\leq K}\mathcal{R}^{N+1})].\label{def:rkn:fix}
\end{equation}
We refer to Section \ref{sec:rem} for the notation. Since $\mathscr{Q}$ and $\mathscr{C}$ are bilinear and trilinear operators, respectively, it suffices to estimate the linear operator $\mathscr{L}$.
In what follows, we show that, with a high probability, the operators $\mathscr{L}^n$ are bounded on $Z$ for all $n$, as stated in Proposition \ref{prop:l}. Consequently, the bound for $(1-\mathscr{L})^{-1}$ follows directly.

\subsection{The analysis of linear operator $\mathscr{L}$}\label{sec:analysisL}
To estimate the linear operator $\mathscr{L}$, we further analyze its structure through the tree expansion developed in the previous sections. 
\begin{proposition} \label{prop-Lm}
Let $\mathscr{ L}$ be defined as in \eqref{def:l}. Note that $\mathscr L^{n}$ is a  real-linear operator for $n \geq 0$. Define its kernels $(\mathscr L^{n})_{k l}^\zeta(t,s)$ for $\zeta \in \{\pm\}$ by 
\begin{equation*}
(\mathscr L^{n} b)_k (t) = \sum_{\zeta \in \{\pm\}} \sum_{l} \int_{\R} (\mathscr L^{n})_{kl}^{\zeta} (t,s) b_{l}(s)^{\zeta}\dif  s.
\end{equation*}
Then, for each $1 \leq n \leq N$ and $\zeta \in \{\pm\}$, we can decompose
\begin{equation*}
(\mathscr L^{n})^\zeta_{k, l} = \sum_{n \leq m \leq N^3} (\mathscr 
L^{n})_{k, l}^{m, \zeta},\notag
\end{equation*}
such that for any $n \leq m \leq N^3$ and $k,l \in \mathbb{Z}_L$ and $t,s \in [0,1]$ with $t > s$, we have 
\begin{equation} \label{eq-Lnm}
\mathbf{E} |(\mathscr L^{n})_{k, l}^{m, \zeta}(t,s)|^2 \lesssim {\langle k-\zeta l\rangle^{-20} L^{44} L^{-\delta m}}.
\end{equation}
    \end{proposition}
First, to analyze the expression of $(\mathscr L^{n})_{k, l}^{m, \zeta}$, we introduce the notion of flower trees, as in \cite[Section 11]{DH23a}.
\begin{definition}\label{def:flowtree}
A flower tree is a tree $\mathcal{T}$ with a   particular  leaf $\mathfrak{f}$  with is  called the  flower. Choosing a different leaf $\mathfrak{f}$ for the same tree $\mathcal{T}$ yields a different flower tree. We call the unique path connecting the root $\mathfrak{r}$ and the flower $\mathfrak{f}$  as the stem. A flower couple is a pair of flower trees whose flowers are coupled.

The  height  of a flower tree $\mathcal{T}_{\mathfrak{f}}$ is defined as the length of stem, i,e, number of branching nodes along its stem. A flower tree of height $n$ could be obtained by adding two sub-trees on each nodes on the stem by $n$ times. A flower tree is called an  admissible  tree if each of the added sub-trees has scale at most $N$.
\end{definition}

We then complete the proof of Proposition~\ref{prop-Lm} by expressing $\mathscr{L}^n$ as a sum over flower trees.
\begin{proof}[Proof of Proposition \ref{prop-Lm}]
  By the definition of the operator $\mathscr{L}$ in \eqref{def:l}, we observe that applying $\mathscr{L}$ corresponds to attaching two subtrees $\mathcal{T}_1$ and $\mathcal{T}_2$ to a single node, where the scales of these subtrees are at most $N$. Hence, the resulting object is an admissible tree in the sense of Definition \ref{def:flowtree}.   In particular, we remark that the absence of terms such as $\mathbf{E}[v_{k_1}\overline {c_{k_1}^{\leq N}}]$ reflects the fact that such contributions cannot be fully paired, in accordance with the definition of flower trees.  
Iterating this construction $n$ times, the operator $\mathscr{L}^n$ corresponds to building an admissible tree of height $n$. Therefore,
    \begin{align*}
        (\mathscr L^{n})^\zeta_{k, l} = \sum_{n \leq m \leq N^3} (\mathscr 
L^{n})_{k, l}^{m, \zeta}=\sum_{n \leq m \leq N^3} \sum_{\mathcal{T}}\tilde{c}_{kl}^{\mathcal{T}}(t,s), 
\end{align*}
where the sum over $\mathcal{T}$ runs over all flower trees with height $n$, scale $m$, and satisfying $\zeta_{\mathfrak{r}}=+$ and $\zeta_{\mathfrak{f}}=\zeta$.  The coefficient $\tilde{c}_{kl}^{\mathcal{T}}(t,s)$ is given by
\begin{align}
    \tilde{c}_{kl}^{\mathcal{T}}(t,s):=\bigg(\frac{\alpha T}{L}\bigg)^m\sum_{\mathcal{D}}\prod_{\mathfrak{n}\in\mathcal{N}}(i\zeta_{\mathfrak{n}}
    )\prod_{\mathfrak{n}\in\mathcal{N}}B_{k_{\mathfrak{n}_1},k_{\mathfrak{n}_2},k_{\mathfrak{n}_3},k_{\mathfrak{n}}}\prod_{\mathfrak{f}\neq \mathfrak{n}\in\mathcal{T}}\varphi_{\leq K}(k_\mathfrak{n})\notag\\
    \times \int_{\mathcal{D}_\mathfrak{n}}\prod_{\mathfrak{n}\in\mathcal{N}}  e^{\zeta_\mathfrak{n}2\pi i \Theta_\mathfrak{n}( {t_\mathfrak{n}}) } \dif t_\mathfrak{n} \delta({t_{\mathfrak{f}^p}-s)}\prod_{\mathfrak{f}\neq \mathfrak{l}\in\mathcal{L}}\sqrt{{n_{{\rm in},\eta}}(k_{\mathfrak{l}})}\tilde{\mathcal{B}}_{\mathcal{T}_\mathfrak{f}}(\omega)1_{k_\mathfrak{f}=l},\label{def:tildect}
\end{align}
where the sum over $\mathcal{D}$ runs over all $k$-decorations of $\mathcal{T}_{\mathfrak{f}}$, and $\mathfrak{f}^p$ denotes the parent of $\mathfrak{f}$. The term $\tilde{\mathcal{B}}_{\mathcal{T}_\mathfrak{f}}$ is defined analogously to \eqref{def:Bomega}, with $\eta_{\mathfrak{f}}(\omega)$ replaced by $1$.  

Next, we consider flower couples $\mathcal{Q}$ whose subtrees all have height $n$ and scale $m$. Then
 \begin{align*}
     \mathbf{E} |(\mathscr L^{n})_{k, l}^{m, \zeta}(t,s)|^2 =\sum_{\mathcal{Q}} (\tilde{\mathcal{K}}_\mathcal{Q})(t,s,k), 
 \end{align*}
 where the sum on $\mathcal{Q}$ runs over all the flower couples $\mathcal{Q} $ such that each substree is with  height $n$ and scale $m$. And 
 \begin{align*}
   (\tilde{\mathcal{K}}_\mathcal{Q})(t,s,k):&= \bigg(\frac{\alpha T}{L}\bigg)^{2m} \zeta(\mathcal{Q}) \sum_{\mathcal{E}}  
   \prod_{\mathfrak{n}\in\mathcal{N}}B_\mathfrak{n}\prod_{\mathfrak{f}\neq \mathfrak{n}\in\mathcal{Q}}\varphi_\mathfrak{n}\\
  &\quad\quad \times 
\int_{\mathcal{E}_\mathfrak{n}}\prod_{\mathfrak{n}\in\mathcal{N}}  e^{\zeta_\mathfrak{n}2\pi i\Theta_\mathfrak{n}(t_\mathfrak{n}) } \dif t_\mathfrak{n}\prod\delta({t_{\mathfrak{f}^p}-s)}
\prod_{\mathfrak{f}\neq \mathfrak{l} \in \mathcal{L}}^+ {n_{{\rm in},\eta}}(k_{\mathfrak{l}})1_{k_\mathfrak{f}=l}.
 \end{align*}
  For \eqref{eq-Lnm}, use the flower-couple argument  with the cutting estimates in Section~\ref{sec:Algorithm}. The squared kernel has original total order $2m$.  

We then proceed as in Proposition \ref{prop:couple}. By the choice of truncation functions,
 $$\prod_{\mathfrak{n}\in\mathcal{N}}B_\mathfrak{n}\prod_{\mathfrak{f}\neq \mathfrak{n}\in\mathcal{Q}}\varphi_\mathfrak{n}\lesssim L^{m\delta}.$$
 Most terms can be handled as before. For the special term $B_{\mathfrak{f}^p}$, let $\mathfrak{f}^{s,1},\mathfrak{f}^{s,2}$ denote the siblings of $\mathfrak{f}$. Using
 $k_{\mathfrak{f}^{p}}=k_{\mathfrak{f}}-k_{\mathfrak{f}^{s,1}}+k_{\mathfrak{f}^{s,2}}$, we obtain
\begin{align*}
B_{\mathfrak{f}^p}&\varphi_{\leq K}(k_{\mathfrak{f}^{p}})\varphi_{\leq K}(k_{\mathfrak{f}^{s,1}})\varphi_{\leq K}(k_{\mathfrak{f}^{s,2}})\notag\\
    &\leq |k_{\mathfrak{f}^p}|^{\beta}|k_{\mathfrak{f}^{s,1}}|^{\beta}|k_{\mathfrak{f}^{s,2}}|^{\beta}(|k_{\mathfrak{f}^p}|+|k_{\mathfrak{f}^{s,1}}|+|k_{\mathfrak{f}^{s,2}}|)^{\beta}\varphi_{\leq K}(k_{\mathfrak{f}^{p}})\varphi_{\leq K}(k_{\mathfrak{f}^{s,1}})\varphi_{\leq K}(k_{\mathfrak{f}^{s,2}})\lesssim L^{\delta/2}.
\end{align*}

 The decay $\langle k-\zeta l\rangle^{-20}$ follows from the indicator $1_{k_\mathfrak{f}=l}$ replacing ${n_{{\rm in},\eta}}(k_{\mathfrak{f}})$. Indeed, $k-\zeta l$ is a linear combination of $k_{\mathfrak{l}}$ for $\mathfrak{f}\neq \mathfrak{l}\in\mathcal{L}$, and the same argument as in \eqref{eq-sp1} applies. For the Dirac term $\delta(t_{\mathfrak{f}^p}-s)$, we return to \eqref{eq-sp1} and observe that omitting the integration over two flower nodes results in a loss of at most $L^{20}$. Finally, by the truncation condition and the relation
 $k_{\mathfrak{f}^{p}}=k_{\mathfrak{f}}-k_{\mathfrak{f}^{s,1}}+k_{\mathfrak{f}^{s,2}}$, we obtain $|k_\mathfrak{f}|\lesssim L^\delta$. Thus, as in Section~\ref{sec-boundoncouple}, we may assume $|k_l|\lesssim L^\delta$ for all $l\in\mathcal{L}$. The counting estimates from Section~\ref{sec:mainest} then remain valid.
\end{proof}

To prove Proposition~\ref{prop:l}, we apply large deviation estimates to obtain pathwise bounds with a high probability. We begin with the following result from \cite[Lemma~7.3]{VW26}:

\begin{lemma}[Gaussian Hypercontractivity]\label{lem:Gaussian Hypercontractivity}
     Let $\{\eta_k\}$ be i.i.d. Gaussian variables.  Given $\zeta_j\in\{+,-\}$ and define 
 $$X=\sum_{k_i,1\leq i\leq n}a_{k_1,...,k_n}\prod^n_{j=1}\eta^{\zeta_j}_{k_j}(\omega),$$
 where $a_{k_1,...,k_n}$ are constants. Then, for $q\geq2$, it holds
$$ \mathbf{E}[|X|^q]\leq (q-1)^{nq/2}  \mathbf{E}[|X|^2]^{q/2}.$$
\end{lemma}

We are now in a position to prove Proposition~\ref{prop:l}.
\begin{proof}[Proof of Proposition \ref{prop:l}]
    By definition, we have
\begin{align*}
\left\|\mathscr{L}^{n}(a)\right\|_{Z}^{2}= & \sup _{0 \leq t \leq 1} L^{-1} \sum_{k \in \mathbb{Z}_{L}}\langle k\rangle^{10}\left|\sum_{\zeta \in\{ \pm\}} \sum_{l} \int_{0 \leq s \leq t} \sum_{n \leq m \leq N^{3}}\left(\mathscr{L}^{n}\right)_{k, l}^{m, \zeta} a_{l}^{\zeta}(s) \mathrm{d} s\right|^{2} \\
\lesssim & \sup _{0 \leq s \leq t \leq 1} \sup _{\zeta} \sup _{m} L^{-1} N^{3} \sum_{k \in \mathbb{Z}_{L}}\langle k\rangle^{10}\left| \sum_{l \in \mathbb{Z}_{L}}\langle l\rangle^{-5}\left(\mathscr{L}^{n}\right)_{k, l}^{m, \zeta}(t, s)\langle l\rangle^{5} a_{l}^{\zeta}(s)\right|^{2} \\
\lesssim & \|a\|_{Z}^{2} N^{3} \sup _{0 \leq s \leq t \leq 1} \sup _{\zeta} \sup _{m} \sum_{k \in \mathbb{Z}_{L}}\sum_{l \in \mathbb{Z}_{L}}\langle k\rangle^{10}\langle l\rangle^{-10}\left|\left(\mathscr{L}^{n}\right)_{k, l}^{m, \zeta}(t, s)\right|^2 \\
\lesssim & \|a\|_{Z}^{2} N^{3} \sup _{0 \leq s \leq t \leq 1} \sup _{\zeta} \sup _{m} \sup _{k, l}\left(\langle k-\zeta  l\rangle^{9}\left|\left(\mathscr{L}^{n}\right)_{k, l}^{m, \zeta}(t, s)\right|\right)^{2}\\
& \quad\times {\sum_{\substack{k,l\in\mathbb Z_L }}}\langle k\rangle^{10}\langle l\rangle^{-10} \langle k-\zeta  l\rangle^{-18}\\
\lesssim & \|a\|_{Z}^{2}  {L^3}\sup _{0 \leq s \leq t \leq 1} \sup _{\zeta} \sup _{m} \sup _{k, l}\left(\langle k-\zeta  l\rangle^{9}\left|\left(\mathscr{L}^{n}\right)_{k, l}^{m, \zeta}(t, s)\right|\right)^{2}.
\end{align*}

It now suffices to bound the above expression with probability  $\geq 1 - L^{-40}$. From the definition of $\left(\mathscr{L}^{n}\right)_{k, l}^{m, \zeta}$ in \eqref{def:tildect}, we observe that  {there are $O((L^{1+\delta})^2)$ possible choices for $(k,l)$, by the support $|k|,|l|\le CL^\delta$ proved above. The preceding weighted double sum is $O(L^{2+\delta})\le L^3$.}
Moreover, the number of choices for $(\zeta,m)$ is finite and bounded by $2\cdot N^3$. Hence, it suffices to establish the desired bound for fixed $(\zeta,m,k,l)$.

Then, by \eqref{eq-Lnm} and an argument analogous to \cite[Proposition 12.1]{DH23a}, 
combining the Gaussian hypercontractivity (Lemma~\ref{lem:Gaussian Hypercontractivity}) 
with the Gagliardo--Nirenberg inequality, we obtain
\begin{align*}
  \mathbf{E}\left|  \sup _{0 \leq s \leq t \leq 1}   \sup _{k, l} \langle k-\zeta  l\rangle^{9}\left|\left(\mathscr{L}^{n}\right)_{k, l}^{m, \zeta}(t, s)\right|\right|^p\lesssim ( L^{40}  L^{-\delta m})^{p/2}.
\end{align*}
Consequently, by Chebyshev’s inequality, with probability  $\geq 1 - L^{-40}$, we have

\begin{align*}
   \sup _{0 \leq s \leq t \leq 1}   \sup _{k, l}  \langle k-\zeta  l\rangle^{9}\left|\left(\mathscr{L}^{n}\right)_{k, l}^{m, \zeta}(t, s)\right|\lesssim L^{45} L^{-\delta n/2}.
\end{align*} 

\end{proof}

\subsection{The estimate on the remainder}

\begin{proposition}\label{prop:ckn}
      With probability $\geq 1-L^{-40}$, it holds that 
      \begin{align*}
        \langle k\rangle^{9}| c_k^{n}(t)|\leq  (L^{-\delta})^{n}L,\ \ \langle k\rangle^{9}|(c_{\sim N})_k(t)|\leq (L^{-\delta})^{N}L,
      \end{align*}
     for any $k \in\mathbb{Z}_L$, $t \in[0,1]$, and all $0 \leq  n \leq N^3$.  We recall \eqref{def:ck} and \eqref{eqnr} for the notations. 
  \end{proposition}

\begin{proof}
   By a similar calculation as in  Proposition~\ref{prop:couple}, we obtain for $n\geq1$
\begin{align*}
 \mathbf{E}| c_k^{n}(t)|^2
&\leq C \langle k \rangle^{-20}  \alpha ^2T L^{7\delta-12\delta (2n-2)},\\
 \mathbf{E}|(c_{\sim N})_k(t)|^2
&\leq  C \langle k \rangle^{-20}  \alpha ^2T L^{7\delta-12\delta (2N-2)}.
\end{align*}
The estimate for $c_{\sim N}$ is obtained in the same way as for $c^n$, with $n$ replaced by $N$. Indeed, $c_{\sim N}$ is defined as the sum of $c_k^{\mathcal{T}^+}$ over all trees $\mathcal{T}^+$ of scale greater than $N$ whose three subtrees have scale at most $N$.
 Then, following the same argument as in \cite[Proposition 12.1]{DH23a}, we apply Gaussian hypercontractivity together with the Gagliardo–Nirenberg inequality to upgrade the above moment bounds to the desired   estimates. 
\end{proof}

The following proposition shows the existence of the remainder term $\mathcal{R}^{N+1}$.

\begin{proposition}\label{prop:rnsmall}
    With probability $\geq 1-L^{-40}$,  the mapping defined on the right hand  side of \eqref{def:rkn:fix} is a contraction mapping from the set $\{b : \|b\|_{ Z} \leq  L^{-500}\}$ to itself.
\end{proposition}
\begin{proof}
We first restrict to the complement of an exceptional set of probability at most $L^{-40}$.  
We then estimate the operators $\mathscr{Q}$ and $\mathscr{C}$ defined in \eqref{def:l}. By definition, we have
\begin{align*}
  \mathcal{W}_1&(\tilde\varphi_{\leq K}u,\tilde\varphi_{\leq K}v,\tilde\varphi_{\leq K}w)_k(t)\notag\\
  &=\frac{i\alpha T}{L}  \sum\limits_{(k_1,k_2, k_{3})}^{\times} u_{k_1}{v_{k_2}}  w_{k_3} e^{2\pi i \Theta(k_1,k_2,k_3,k,t)}B_{k_1,k_2,k_3,k}\tilde\varphi_{\leq K}(k_1)\tilde\varphi_{\leq K}(k_2)\tilde\varphi_{\leq K}(k_3)\varphi_{\leq K}(k)\notag\\
&\quad -\frac{i\alpha T}{L}u_{k}{v_{k}}w_{k}B_{k,k,k,k}\varphi_{\leq K}(k)\tilde\varphi_{\leq K}(k)^3\notag.    
\end{align*}
Here we recall that $\tilde\varphi_{\leq K}u=(\tilde\varphi_{\leq K}(k)u_k)_{k\in\mathbb{Z}_L}$ (see Section~\ref{sec:rem}). By the truncation properties of $\tilde\varphi_{\leq K}$, we obtain
 
\begin{align*}
B_{k_1,k_2,k_3,k}\tilde\varphi_{\leq K}(k_1)\tilde\varphi_{\leq K}(k_2)\tilde\varphi_{\leq K}(k_3)\varphi_{\leq K}(k)\leq L^\delta,\\
B_{k,k,k,k}\varphi_{\leq K}(k)\tilde\varphi_{\leq K}(k)^3\leq L^\delta.
\end{align*}
Consequently,
\begin{align*}
  &\bigg\| \int_0^t\mathcal{W}_1(\tilde\varphi_{\leq K}u,\tilde\varphi_{\leq K}v,\tilde\varphi_{\leq K}w)(s)\dif s\bigg\|_{Z}\notag\\
  & \leq  \int_0^t\big\|\mathcal{W}_1(\tilde\varphi_{\leq K}u,\tilde\varphi_{\leq K}v,\tilde\varphi_{\leq K}w)(s)\big\|_{Z}\dif s\leq L^{21}\|u\|_Z\|v\|_Z\|w\|_Z,
   \end{align*}
   and similarly,
   \begin{align*}
    \bigg\| \int_0^t(\mathscr{Q}_1)_k[\tilde\varphi_{\leq K}u](\tilde\varphi_{\leq K}v)\dif s\bigg\|_{Z}&\leq L^{21}\|u\|_Z\|v\|_Z^2,\\
    \bigg\| \int_0^t  (\mathscr{C}_1)_k(\tilde\varphi_{\leq K}v)\dif s\bigg\|_{Z}&\leq L^{21}\|v\|_Z^3.
\end{align*}
Here we recall the definition of $\mathscr{Q}_1$ and $\mathscr{C}_1$ in Section \ref{sec:rem}.

 Using that $\tilde\varphi_{\leq K}c^{\leq N}=c^{\leq N}$, we rewrite \begin{align*} 
\mathscr{Q}(\tilde\varphi_{\leq K}v)&:=\int_0^t \big(2\mathcal{W}_1(\tilde\varphi_{\leq K}v,\tilde\varphi_{\leq K}\overline{v}, \tilde\varphi_{\leq K}c^{\leq N})\\
&\quad\quad\quad +\mathcal{W}_1(\tilde\varphi_{\leq K}v, \tilde\varphi_{\leq K}\overline{c^{\leq N}},\tilde\varphi_{\leq K} v)+\mathscr{Q}_1[\tilde\varphi_{\leq K}c^{\leq N}](\tilde\varphi_{\leq K}v)\big)\dif s,\\
\mathscr{C}(\tilde\varphi_{\leq K}v)&:=\int_0^t \mathcal{W}_1(\tilde\varphi_{\leq K}v, \tilde\varphi_{\leq K}\overline{v}, \tilde\varphi_{\leq K}v)+\mathscr{C}_1(\tilde\varphi_{\leq K}v)\dif s.
\end{align*} 

Let $\|b\|_{Z}\leq L^{-500}$. By Proposition~\ref{prop:ckn} (for sufficiently large $N$), we obtain, with probability   $\geq 1-L^{-40}$,
\begin{align}
   \|c_{\sim N}& +\mathscr{Q}(\tilde\varphi_{\leq K}b)+\mathscr{C}(\tilde\varphi_{\leq K}b)\|_Z\notag\\
   &\leq \|c_{\sim N}\|_Z + L^{21}\|b\|_Z^2(\sum_{n\leq N}\|c^n\|_Z)+ L^{21}\|b\|_Z^3\leq  L^{-600}.\label{bd:jlc}
\end{align}
Then we write  
\begin{align*}
    (1 -\mathscr{L})^{-1} =(1-\mathscr{L}^N)^{-1}(1+\mathscr{L} +...+\mathscr{L}^{N-1}),
\end{align*}
which is a map on $Z$. Here we understand  $(1-\mathscr{L}^N)^{-1}$  by the Neumann series. By applying Proposition \ref{prop:l}, we have 
\begin{align*}
    \| (1 -\mathscr{L})^{-1} \|_{Z\to Z} \lesssim L^{70}.
\end{align*}
Together with \eqref{bd:jlc} we finish the proof.
\end{proof}
Then, by a fixed point argument together with Proposition~\ref{prop:rnsmall}, we obtain that, with probability   $\geq 1 - L^{-40}$, there exists a solution $\mathcal{R}^{N+1}$ to \eqref{def:rkn:fix} satisfying  \begin{align}
     \|\mathcal{R}^{N+1}\|_{ Z} \leq  L^{-500}.\label{bd:mathcalr}
\end{align}

\section{Proof of Theorem \ref{thm2}}\label{sec:proof of theorem 2}

We first introduce the following convergence of iterates.

\begin{proposition}
    \label{asymptotic}
    Assume $1\leq T<L^{1-10\delta}$ and
$\alpha T^{\sigma/2-\beta}D^2\leq L^{-4\delta}$
if $\sigma>1$. Let $\gamma>0$ small enough such that   $T  < L^{\frac{(1+\gamma)(1+2\beta)}{2}-10\delta}$.  Let $\phi\in\mathcal S(\mathbb R)$.
For any $t \in (0,T)$, let   \begin{align*}
 \mathscr S_t(\phi):&=\sum\limits_{\substack{ k-k_1+k_2-k_3=0,\\ \overline \Omega_k\neq0}}\phi_k \phi_{k_1} \phi_{k_2} \phi_{k_3}  \left[ \frac{1}{\phi_k} - \frac{1}{\phi_{k_1}} + \frac{1}{\phi_{k_2}} - \frac{1}{\phi_{k_3}} \right]\left| \frac{\sin(\pi t\overline \Omega_k)}{\pi t \overline \Omega_k} \right|^2\notag\\
  &\quad\quad\quad\times
|k_1|^{2\beta}|k_2|^{2\beta}|k_3|^{2\beta}|k|^{2\beta}\varphi_{\leq K}^2(k_1)\varphi_{\leq K}^2(k_2)\varphi_{\leq K}^2(k_3)\varphi_{\leq K}^2(k).\\
\mathscr K_t(\phi):&=L^{2} \int_{\xi_1-\xi_2+\xi_3=\xi }\phi(\xi) \phi(\xi_1) \phi(\xi_2) \phi(\xi_3)  \left[ \frac{1}{\phi(\xi)} - \frac{1}{\phi(\xi_1)} + \frac{1}{\phi(\xi_2)} - \frac{1}{\phi(\xi_3)} \right]    \left| \frac{\sin(\pi t\overline \Omega( \xi))}{\pi t\overline\Omega( \xi)} \right|^2 \notag\\ &\quad\quad\times |\xi_1|^{2\beta}|\xi_2|^{2\beta}|\xi_3|^{2\beta}|\xi|^{2\beta}\varphi_{\leq K}^2(\xi_1)\varphi_{\leq K}^2(\xi_2)\varphi_{\leq K}^2(\xi_3)\varphi_{\leq K}^2(\xi)   \dif \xi_1  \dif\xi_2 \, \dif\xi_3.
\end{align*}
 Then
$$
\mathscr S_t=\mathscr K_t +O(L^{2-\delta} t^{-1}).
$$
\end{proposition}

\begin{proof}
First, fix $k$ and define \begin{align*}
    g(x):&=\left|\frac{\sin \pi x}{\pi x}\right|^2\leq 1,\\
    F(k_1,k_3):&=\phi_k \phi_{k_1} \phi_{k_1+k_3-k} \phi_{k_3}  \left[ \frac{1}{\phi_k} - \frac{1}{\phi_{k_1}} + \frac{1}{\phi_{k_1+k_3-k}} - \frac{1}{\phi_{k_3}} \right]\\
    &\quad\quad\times
 |k_1|^{2\beta}|k_1+k_3-k|^{2\beta}|k_3|^{2\beta}|k|^{2\beta}\varphi_{\leq K}^2(k_1)\varphi_{\leq K}^2(k_1+k_3-k)\varphi_{\leq K}^2(k_3)\varphi_{\leq K}^2(k)\\
 &\quad\quad \times \chi(T^{\frac{1}{1+2\beta}}L^{2\delta}|k_1|)\chi(T^{\frac{1}{1+2\beta}}L^{2\delta}|k_1+k_3-k|)\chi(T^{\frac{1}{1+2\beta}}L^{2\delta}|k_3|),
\end{align*}
where   $\chi$ is a smooth cutoff satisfying $\chi(x)=0$ for $|x|\leq \frac12$ and $\chi(x)=1$ for $|x|\geq 1$. These truncations remove the singularity near $|k_i|=0$. 

Define the truncated quantities
\begin{align*}
   (\mathscr S_{\rm tr})_t&:=\sum_{(k_1, k_3)\in \mathbb{Z}_L^{2},\overline\Omega(k_1,k_3)\neq 0}F(k_1,k_3) g\left(t\overline \Omega(k_1,k_3)\right),\\
 (\mathscr K_{\rm tr})_t&:=L^{2} \int_{\xi_1,\xi_3}F(\xi_1,\xi_3) g\left(t\overline \Omega(\xi_1,\xi_3)\right)\dif \xi_1\dif \xi_3,
\end{align*}
where
$$\overline\Omega(x,y) =|x|^\sigma-|x+y-k|^\sigma+|y|^\sigma-|k|^\sigma+\Gamma_0(x)-\Gamma_0(x+y-k)+\Gamma_0(y)-\Gamma_0(k),$$
and
$ \Gamma_0(x)=C_K\alpha  \varphi_{\leq K}(x)|x|^{2\beta}.$

First, the condition  $\overline \Omega_k\neq 0$ in the expression of $\mathscr S_{\rm tr}$ can be removed.  Indeed, by  applying the 3-counting estimate in Lemma~\ref{counting} with $T$ replaced by $T L^{4\delta}$ and  $D\leq L^\delta$, 
the contribution of the sum over $\overline \Omega_k{=0}$ is $$L^2T ^{-1}L^{-4\delta}D^2\log L
 \lesssim L^{2-\delta}T^{-1}
 \leq L^{2-\delta}t^{-1}.$$

Then,  for  the truncated asymptotic functions, we show that $$
( \mathscr S_{\rm tr})_t= (\mathscr K_{\rm tr})_t +O(L^{2-\delta} t^{-1}).
$$
Since $\widehat g(\tau)=1-|\tau|$ on $[-1,1]$ and vanishes outside, we write $e(x):=e^{2\pi i x}$. Applying Poisson summation, we obtain 
\begin{align*}
 (  \mathscr S_{\rm tr})_t&=\sum_{(k_1, k_3)}F(k_1,k_3) g\left(t\overline \Omega(k_1,k_3)\right)=\sum_{(k_1, k_3)}F(k_1,k_3)\int_{-\infty}^{\infty}\hat{g}(\tau)e(\tau t\overline \Omega(k_1,k_3))\dif \tau\\
 &=\frac1t \sum_{(k_1, k_3)}F(k_1,k_3)\int_{-\infty}^{\infty}\hat{g}(\frac\tau t)e(\tau \overline \Omega(k_1,k_3))\dif \tau=\frac1t \int_{-\infty}^{\infty}\hat{g}(\frac\tau t)\sum_{(k_1, k_3)}F(k_1,k_3)e(\tau \overline \Omega(k_1,k_3))\dif \tau\\
 &=\frac1t \int_{-\infty}^{\infty}\hat{g}(\frac\tau t)\sum_{(k_1, k_3)\in\mathbb{Z}^{2}}F(\frac{k_1}{L},\frac{k_3}L)e(\tau \overline \Omega(\frac{k_1}L,\frac{k_3}L))\dif \tau\\
 &=\frac1t \int_{-\infty}^{\infty}\hat{g}(\frac\tau t) \sum_{(c,d)\in\mathbb{Z}^{2}}\int_{(x,y)\in\mathbb{R}^{2}}F(\frac{x}{L},\frac{y}L)e(\tau \overline \Omega(\frac{x}L,\frac{y}L))e(-c\cdot x-d\cdot y) \dif x\dif y\dif \tau\\
  &=\frac1t \int_{-\infty}^{\infty}\hat{g}(\frac\tau t) \int_{(x,y)\in\mathbb{R}^{2}}F(\frac{x}{L},\frac{y}L)e(\tau \overline \Omega(\frac{x}L,\frac{y}L))\dif x\dif y\dif \tau=:  (\mathscr S_{\rm tr})_1\\
   &\quad+\frac{L^{2}}t \int_{-\infty}^{\infty}\hat{g}(\frac\tau t) \sum_{(c,d)\in\mathbb{Z}^{2}/\{0,0\}}\int_{(x,y)\in\mathbb{R}^{2}}F(x,y)e(\tau \overline \Omega(x,y)-c\cdot Lx-d\cdot Ly) \dif x\dif y\dif \tau=:  (\mathscr S_{\rm tr})_2.
\end{align*}

For the main term, a change of variables yields
\begin{align*}
  ( \mathscr S_{\rm tr})_1=\frac{L^{2}}t \int_{-\infty}^{\infty}\hat{g}(\frac\tau t) \int_{(x,y)\in\mathbb{R}^{2}}F(x,y)e(\tau \overline \Omega(x,y))\dif x\dif y\dif \tau= (\mathscr K_{\rm tr})_t.
\end{align*}

For the second term $(\mathscr S_{\rm tr})_2$, in the case $\sigma\in(1,2]$, it holds that by the cut off function $\varphi$ and $\chi$,
\begin{align*}
   |\tau  \nabla  \overline \Omega(x,y)|\lesssim T(L^{\delta(\sigma-1)}+\alpha (T^{\frac{1}{1+2\beta}}L^{2\delta})^{1-2\beta})\leq L^{1-\delta/2},
\end{align*}
where we used $|\tau|\leq T < \min\{L^{1-10\delta},L^{\frac{(1+\gamma)(1+2\beta)}{2}-10\delta}\}$.

While in the case $\sigma\in(0,1)$, we have 
\begin{align*}
   |\tau  \nabla  \overline \Omega(x,y)|\lesssim T((T^{\frac{1}{1+2\beta}}L^{2\delta})^{1-\sigma}+\alpha (T^{\frac{1}{1+2\beta}}L^{2\delta})^{1-2\beta})\leq L^{1-\delta/2},
\end{align*}
where we used $|\tau|\leq T < L^{\frac{(1+\gamma)(1+2\beta)}{2}-10\delta}$ and $2\beta<\sigma$. In summary, we have 
\begin{align*}
   |\tau  \nabla  \overline \Omega(x,y)| \leq L^{1-\delta/2}.
\end{align*}

Then for $c\neq 0$ we have
\begin{align*}
|   \nabla_x[ \tau \overline\Omega(x,y)-c\cdot Lx-d\cdot Ly]|\geq L|c|-CL^{1-\delta/2}\geq \frac12|c|L,
\end{align*}
as $L$ goes large enough.

Now, from the definition of $F(x,y)$,  once we apply integration by parts, 
we notice that by the cut off function $\varphi$ and $\chi$,  $$|\partial_{x_i}|x|^{2\beta}|\leq 2\beta|x|^{2\beta-1}\lesssim (TL^{2\delta})^{1-2\beta},$$ 
we obtain a power of $TL^{2\delta}\cdot (|c|L)^{-1}$ from $F(x,y)$. By  integrating enough times, we have a gain of $L^{-2\delta}$,  then we obtain $(\mathscr S_{\rm tr})_2 =O(L^{2-\delta} t^{-1}).$

It remains to estimate the difference between truncated and original quantities. By the support of $1-\chi$, at least one of $|k_1|,|k_1+k_3-k|,|k_3|$ is $\lesssim T^{-{\frac{1}{1+2\beta}}}L^{-2\delta}$, and hence
\begin{align*}
 | ( \mathscr K_{\rm tr})_t- \mathscr K_t|&\leq   L^{2} \int_{\xi_1,\xi_3}\phi_\xi \phi_{\xi_1} \phi_{\xi_1+\xi_3-\xi} \phi_{\xi_3}  \left[ \frac{1}{\phi_\xi} - \frac{1}{\phi_{\xi_1}} + \frac{1}{\phi_{\xi_1+\xi_3-\xi}} - \frac{1}{\phi_{\xi_3}} \right](T^{-{\frac{1}{1+2\beta}}}L^{-2\delta} L^{\delta})^{1+2\beta}\dif \xi_1\dif \xi_3\\
 &=O(L^2(T^{-{\frac{1}{1+2\beta}}}L^{-\delta})^{1+2\beta})\leq O(L^{2-\delta}T^{-1})\leq O(L^{2-\delta}t^{-1}).
\end{align*} The same argument yields
\begin{align*}
 | (\mathscr S_{\rm tr})_t- \mathscr S_t|&\leq  O(L^{2-\delta}t^{-1}).
\end{align*}Combining the above estimates, we conclude the proof.
\end{proof}

\begin{proposition}\label{lem:converge}
    Let $T<L^{2-\delta}$. 
For any Schwartz function $\phi$, it holds that uniformly in $t\in[L^\epsilon,T]$
\begin{align*}
   t L^{-2}\mathscr K_t(\phi)(\xi)-\mathcal K (\phi)(\xi)=o_{l^\infty_\xi}(1).
\end{align*}

\end{proposition}
\begin{proof}
Let \begin{align*}
    \mathcal K_{\rm tr}(\phi)(\xi):=& \int_{\xi=\xi_1-\xi_2+\xi_3,\  \overline{\Omega}(\xi)=0}\varphi_{\leq K}^2(\xi_1)\varphi_{\leq K}^2(\xi_2)\varphi_{\leq K}^2(\xi_3)\varphi_{\leq K}^2(\xi)|\xi|^{2\beta}|\xi_1|^{2\beta}|\xi_2|^{2\beta}|\xi_3|^{2\beta}  \notag\\
   &\ \ \  \  
   \times \phi \phi_1 \phi_2 \phi_3\left(\frac{1}{\phi}-\frac{1}{\phi_1}+\frac{1}{\phi_2}-\frac{1}{\phi_3}\right)\, \dif \xi_1 \dif \xi_2 \dif \xi_3.\notag
\end{align*}

By writing $n_{\rm in }(\xi)=\langle \xi\rangle^{-100}\tilde{n}_{\rm in }(\xi)$, it suffices to consider the problem on a compact domain. 
Using the relation $\xi=\xi_1-\xi_2+\xi_3$, we rewrite the collision kernel as a function $W(\xi_1,\xi_3)$. We first prove that \begin{align*}
  \mathcal K (\phi)(\xi)  -\mathcal K_{\rm tr} (\phi)(\xi)=o_{l^\infty_\xi}(1).
\end{align*}

We recall that uniformly in $k\in\mathbb{R}$, $ \Gamma_0(k) \leq C\alpha L^\delta.$
We decompose the domain into
\begin{align*}
  E_1:=  &\{(\xi_1,\xi_3):\min\{|\xi_1|,|\xi_2|,|\xi_3|,|\xi|\}\leq L^{-\epsilon}  \},\\
   E_2:=  &\{(\xi_1,\xi_3):\min\{|\xi_1-\xi|,|\xi_3-\xi|\}\leq L^{-3\epsilon}  \},
\end{align*}
for some $\epsilon>0$ small enough.  On the set $E_1$, we use the symbol $|\xi_1|^\beta|\xi_2|^\beta|\xi_3|^\beta|\xi|^\beta$ to gain smallness,  which yields a contribution of order $o_{l^\infty_k}(1)$.  On the set $E_2\cap (E_1)^c$, we use the cancellation property $W(\xi,\xi_3)=0$. For instance, if $|\xi_1-\xi|\leq L^{-3\epsilon}$, then
\begin{align*}
     |W (\xi_1,\xi_3)|=|W (\xi_1,\xi_3)-W (\xi,\xi_3)|\leq \|W'\|_{C^1_{(\xi_1,\xi_3)}}\cdot L^{-3\epsilon} =o_{l^\infty_k}(1).
\end{align*}

We now consider this problem on the domain $(E_1\cup E_2)^c$. 
We further decompose the domain into finitely many subdomains where the signs of $\xi_i$ are fixed. 
On each subdomain, the functions $\Omega$, $\Gamma_0$, and $W$ are $C^1$, with derivatives bounded by $L^{C\epsilon}$. Define the perturbed resonance surface
 \[
S_\alpha := \{(\xi_1,\xi_3):\overline\Omega(\xi_1,\xi_3,\alpha)=0\}.
\]
On the domain $(E_1\cup E_2)^c$, it is easy to check the non-degeneracy condition 
\begin{align*}
  |\partial_i \overline \Omega(\xi_1,\xi_3,\alpha)|\gtrsim L^{-10\epsilon},\qquad  |\partial_i  \Omega(\xi_1,\xi_3)|=|\partial_i \overline \Omega(\xi_1,\xi_3,0)|\gtrsim L^{-10\epsilon}, 
\end{align*} 
the implicit function theorem implies that $S_\alpha$ can be parametrized as
\[
\xi_1=F(\xi_3,\alpha),
\]
with $F$ smooth.

 We then compare the integrals along $S_\alpha$ and $S_0$. 
\[
\Big|\int_{S_\alpha} W\,\dif S - \int_{S_0} W\,\dif S\Big| =\Big|\int_{-\delta}^\delta \frac{W(F\big(y,\alpha\big),y)}{|\partial_1\overline \Omega(F\big(y,\alpha\big),y)|}-\frac{W(F\big(y,0\big),y)}{|\partial_1\Omega(F\big(y,0\big),y)|}\dif y\Big|.
\]

Differentiating the identity ${\overline\Omega}(F(y,\alpha),y,\alpha)=0,$
we obtain
\[
 \partial_sF_s
 =-\frac{\partial_s\overline\Omega(F_s,y,s)}
          {\partial_1\overline\Omega(F_s,y,s)} \lesssim L^{10\epsilon}.
\] 
Then, we have 
\begin{align*}
    |W(F\big(y,\alpha\big),y)-W(F\big(y,0\big),y)|\leq \alpha \|W\|_{C^1}\|\partial_\alpha F\|_{C^0}\lesssim \alpha L^{10\epsilon}. 
\end{align*}

On the other hand, it holds that 
\begin{align*}
    \left|\frac{1}{|\partial_1\overline \Omega(F\big(y,\alpha\big),y)|}-\frac{1}{|\partial_1\Omega(F\big(y,0\big),y)|}\right|&\lesssim  L^{20\epsilon}\bigl(|\partial_1\overline \Omega(F\big(y,\alpha\big),y)-\partial_1\overline \Omega(F\big(y,0\big),y)|{+\|\partial_1(\overline\Omega-\Omega)\|_\infty}\bigr)\\
    &\lesssim L^{20\epsilon} (\alpha |\partial_1^2\overline\Omega| |\partial_\alpha F|+\|\partial_1(\overline\Omega-\Omega)\|_\infty)\leq \alpha L^{40\epsilon}.
\end{align*}

By choosing $\epsilon>0$ small enough, we have 
\[
\Big|\int_{S_\alpha} W\,\dif S - \int_{S_0} W\,\dif S\Big| \lesssim \alpha L^{50\epsilon}=o_{l^\infty_\xi}(1).
\]

By the Schwartz property of $\phi$, the cutoff function $\varphi_{\leq K}^2$ can be also removed.

Then, we prove that 

\begin{align*}
   t L^{-2}\mathscr K_t(\phi)(\xi)-\mathcal K_{\rm tr} (\phi)(\xi)=o_{l^\infty_\xi}(1).
\end{align*}
On the domain of $E_1\cup E_2$, noticing that $t \int \left| \frac{\sin(\pi t x)}{\pi t x} \right|^2 \dif x=1$, the estimate is similar as before. 
On the domain $(E_1\cup E_2)^c$, we know that  
\begin{align*}
  |\partial_i \overline \Omega(\xi_1,\xi_3,\alpha)|\gtrsim L^{-10\epsilon}>0, 
\end{align*}  which implies that 
\begin{align*}
\mathscr K_t(\phi):&=L^{2} \int   \left| \frac{\sin(\pi ts)}{\pi ts} \right|^2 G(s)\dif s,\qquad  
G(s)=\int
\frac{W(F(y,s),y)}
{|\partial_1\overline\Omega(F(y,s),y)|}\,dy. \notag
\end{align*}
Then 
the proof follows from the identity that, for any smooth function $f$,
\begin{align*}
t \int \left| \frac{\sin(\pi t x)}{\pi t x} \right|^2 f(x)\dif x =  f(0) + O(t^{-1}).
\end{align*}

\end{proof}

Now, with these two Proposition in hand, we are ready to prove Theorem \ref{thm2}.
\begin{proof}[Proof of Theorem \ref{thm2}]
We denote by $E$ the complement of all exceptional sets appearing in Proposition~\ref{prop:l} and Proposition~\ref{prop:rnsmall}, so that
     $$\mathbf{P}(E) \geq  1-L^{-40}.$$
 By the reductions in Section~\ref{reductions}, we have
 $$\mathbf{E}|\widehat{u}_{\rm tr}(k,t)|^2=\mathbf{E}|c_k(s)|^2,\ \ s = \frac tT.$$
 We further denote $\mathbf{E}_E[\ \cdot\ ]:=\mathbf{E}[ 1_E\cdot\ ]$. Then
 \begin{align*}
  \mathbf E|\widehat u_{\rm tr}(k, t)|^2&=   \mathbf{E}_E |c_k(s)|^2+O(L^{-40}) \\
  &=\mathbf{E}_E| c^{\leq N}_k (s)|^2 +2\mathrm{Re}\mathbf{E} _E[c^{\leq N}_ k (s)\overline{\mathcal{R}^{N+1}_k (s)}] +\mathbf{E}_E |\mathcal{R}^{N+1}_k (s)|^2+O(L^{-40}) .
 \end{align*}
 For the terms involving the remainder, by Proposition~\ref{prop:ckn} and \eqref{bd:mathcalr}, we obtain
 \begin{align*}
     2\mathrm{Re}\mathbf{E} _E[c^{\leq N}_ k (s)\overline{\mathcal{R}^{N+1}_k (s)}] +\mathbf{E}_E |\mathcal{R}^{N+1}_k (s)|^2=O(L^{-10}). 
 \end{align*}
Next, we have
\begin{align*}
    \mathbf{E}_E| c^{\leq N}_k (s)|^2&= \sum_{n=0}^{2N}\sum_{|\mathcal{T}^+|+|\mathcal{T}^-|=n}\mathbf{E}_E[c^{\mathcal{T}^+}_kc^{\mathcal{T}^-}_k(s)]\notag
\end{align*} 
Using Proposition~\ref{prop:couple}, the sum of all the couple with $n\geq3 $ is of order $O(L^{-\delta}\frac{t}{T_{\mathrm{kin}}})$. Here, we can account  for small $t$ by rescaling $T$.

It remains to estimate the lower-order terms appearing in the above expression.
For the case $n=0$, by definition,   $c_k^{\bullet}(s)= \varphi_{\leq K}(k)\sqrt{{n_{{\rm in},\eta}}(k)}\cdot\eta_k(\omega)^\zeta$, where $\zeta=\pm$ depends on the sign of the root.
Taking expectation, we obtain that the zeroth-order contribution is  $\varphi_{\leq K}^2(k)\, {n_{{\rm in},\eta}}(k)$.

 For the case $n=1$, we have $\mathbf{E}_E[c^{\mathcal{T}^+}_kc^{\mathcal{T}^-}_k(s)]=0$,  since there is no admissible  decorations in this case.

 It remains to consider the case $n=2$. In this case, we need to estimate the terms $\mathbf{E}|c^1_k(s)|^2$ and $2\mathrm{Re}\mathbf{E}\left[\overline{c^0_k(s)}\, c^2_k(s)\right]$. For the former, by definition, we have
 
\begin{align}
c_{k}^{1}(s)&=\frac{i\alpha T}{L}  \sum\limits_{(k_1,k_2, k_{3})}^{\times} (a_{\rm in})_{k_1}\overline {(a_{\rm in})_{k_2}}  (a_{\rm in})_{k_3} \int_0^se^{2\pi i[\overline\Omega_kT{u}+\sum\Gamma_1(k,{u})]}\dif {u}B_{k_1,k_2,k_3,k}\varphi_{\leq K}(k)\notag\\
&\quad+\frac{i\alpha Ts}{L} \left(|(a_{\rm in})_{k}|^2 -2\mathbf{E}|(a_{\rm in})_{k}|^2  \right) B_{k,k,k,k}\varphi_{\leq K}(k) (a_{\rm in})_{k}\notag\\
&\quad+2\frac{i\alpha Ts}{L}\sum_{k_1\neq k}\left(|(a_{\rm in})_{k_1}|^2 -\mathbf{E}|(a_{\rm in})_{k_1}|^2  \right) B_{k_1,k_1,k,k}\varphi_{\leq K}(k) (a_{\rm in})_{k}\notag\\
&=\frac{i\alpha T}{L} \sum\limits_{(k_1,k_2, k_{3}),\overline \Omega_k\neq 0}^{\times} (a_{\rm in})_{k_1}\overline {(a_{\rm in})_{k_2}}  (a_{\rm in})_{k_3} \frac{e^{2\pi i\overline \Omega_kTs}-1}{2\pi i\overline\Omega_kT}B_{k_1,k_2,k_3,k}\varphi_{\leq K}(k)\notag\\
&\quad+\frac{i\alpha T}{L}  \sum\limits_{(k_1,k_2, k_{3}),\overline \Omega_k\neq 0}^{\times} (a_{\rm in})_{k_1}\overline {(a_{\rm in})_{k_2}}  (a_{\rm in})_{k_3}\frac{Er_1(s)}{T\overline\Omega_k}B_{k_1,k_2,k_3,k}\varphi_{\leq K}(k)\notag\\
&\quad+\frac{i\alpha T}{L} \sum\limits_{(k_1,k_2, k_{3}),\overline \Omega_k= 0}^{\times} (a_{\rm in})_{k_1}\overline {(a_{\rm in})_{k_2}}  (a_{\rm in})_{k_3} \int_0^se^{2\pi i\sum\Gamma_1(k,u)}\dif uB_{k_1,k_2,k_3,k}\varphi_{\leq K}(k)\notag\\
&\quad+\frac{i\alpha Ts}{L} \left(|(a_{\rm in})_{k}|^2 -2\mathbf{E}|(a_{\rm in})_{k}|^2  \right) B_{k,k,k,k}\varphi_{\leq K}(k) (a_{\rm in})_{k}\notag\\
&\quad+2\frac{i\alpha Ts}{L}\sum_{k_1\neq k}\left(|(a_{\rm in})_{k_1}|^2 -\mathbf{E}|(a_{\rm in})_{k_1}|^2  \right) B_{k_1,k_1,k,k}\varphi_{\leq K}(k) (a_{\rm in})_{k},\label{bd:ck1}
\end{align}
where we define the error term by
\begin{align*}
    Er_1(s):= -\int_0^se^{2\pi i[ \overline\Omega_kTu+\sum\Gamma_1(k,u)]}\sum\partial_u\Gamma_1(k,u)\dif u+ \frac{e^{2\pi i[ \overline\Omega_kTs+\sum\Gamma_1(k,s)]}-e^{2\pi i \overline\Omega_kTs}}{2\pi i}.\notag
\end{align*}
By \eqref{bd:gamma'}, we obtain
\begin{align*}
   | Er_1(s)|&\lesssim s\sup_{u\in[0,s]}\sup_k|\partial_t\Gamma_1(k,u)|+|e^{2\pi i \sum\Gamma_1(k,s)}-1|\notag\\
   &\lesssim s \|\partial_t\Gamma_1(k,s)\|_{L_k^\infty C^0_s}+|  \sum\Gamma_1(k,s)|\lesssim s \|\partial_t\Gamma_1(k,s)\|_{L_k^\infty C^0_s}\lesssim sL^{-3\delta}.
\end{align*}

We denote each term on the right-hand side of \eqref{bd:ck1} by $C_i(s)$ for $1 \leq i \leq 5$, and estimate them in turn.   By Lemma~\ref{thm:Isserlis} together with a direct computation, we obtain

\begin{align*}
\mathbf{E}|C_1(s)|^2&=\frac{2\alpha^2T^2s^2
}{L^{2}}\sum_{(k_1, k_2, k_3); \overline \Omega_k\neq 0}^\times {n_{{\rm in},\eta}}(k_1){n_{{\rm in},\eta}}(k_2){n_{{\rm in},\eta}}(k_3)\left|\frac{\sin \pi\overline \Omega _kTs}{\pi \overline\Omega _kTs}\right|^2\\
&\quad\quad\times B^2_{k_1,k_2,k_3,k}\varphi_{\leq K}^2(k_1)\varphi_{\leq K}^2(k_2)\varphi_{\leq K}^2(k_3)\varphi_{\leq K}^2(k),\\
\mathbf{E}|C_4(s)|^2&=\frac{\alpha^2 T^2s^2}{L^{2}}\mathbf{E}[(|\eta_{k}(\omega)|^2 -2 )^2|\eta_{k}(\omega)|^2]B^2_{k,k,k,k}\varphi_{\leq K}^8(k){n_{{\rm in},\eta}}(k)|{n_{{\rm in},\eta}}(k)|^2\\
&=\frac{2\alpha^2 T^2s^2}{L^{2}}B^2_{k,k,k,k}\varphi_{\leq K}^8(k){n_{{\rm in},\eta}}^3(k),\\
    \mathbf{E}|C_5(s)|^2&=\frac{4\alpha^2 T^2s^2}{L^{2}}\sum_{k_1\neq k}\mathbf{E}(|\eta_{k_1}(\omega)|^2 -1 )^2 \mathbf{E}|\eta_{k}(\omega)|^2 B^2_{k_1,k_1,k,k}\varphi_{\leq K}^4(k_1)\varphi_{\leq K}^4(k){n_{{\rm in},\eta}}(k)|{n_{{\rm in},\eta}}(k_1)|^2\\
    &=\frac{4\alpha^2 T^2s^2}{L^{2}}\sum_{k_1\neq k}B^2_{k_1,k_1,k,k}\varphi_{\leq K}^4(k_1)\varphi_{\leq K}^4(k){n_{{\rm in},\eta}}(k){n_{{\rm in},\eta}}^2(k_1).
\end{align*}
Moreover, by Lemma~\ref{thm:Isserlis} and the resonance relations, we have for $1\leq i\leq 4$, \begin{align*}
   \mathrm{Re} \mathbf{E}[C_i(s)\overline{C_5(s)}]=0,
\end{align*} 
and   \begin{align*}
   \mathrm{Re} \mathbf{E}[(C_1(s)+C_2(s)+C_3(s))\overline{C_4(s)}]+\mathrm{Re} \mathbf{E}[(C_1(s)+C_2(s))\overline{C_3(s)}]=0.
\end{align*}  

For the term $C_3$, since $\overline \Omega_k=0$, we apply the 3-counting estimate in Lemma~\ref{counting} with $T$ replaced by $T L^{4\delta}$ and choose $D=L^\delta$. Then
\begin{align*}
   \mathbf{E}|C_3(s)|^2 &\lesssim \frac{\alpha^2 T^2s^2}{L^{2}} \sum\limits_{(k_1,k_2, k_{3}),\overline \Omega_k= 0}^{\times}  {n_{{\rm in},\eta}}(k_1){n_{{\rm in},\eta}}(k_2){n_{{\rm in},\eta}}(k_3)\notag\\
&\quad\times B^2_{k_1,k_2,k_3,k}\varphi_{\leq K}^2(k_1)\varphi_{\leq K}^2(k_2)\varphi_{\leq K}^2(k_3)\varphi_{\leq K}^2(k)\notag\\
   &\lesssim \frac{\alpha^2 T^2s^2}{L^{2}}\cdot  L^{2-4\delta}T^{-1}\log L\cdot L^\delta=O(\frac{t}{T_{\rm kin}}L^{-\delta}).
\end{align*}

Similarly, applying the 3-counting estimate in Lemma~\ref{counting} again, we obtain
\begin{align*}
\mathbf{E}&[C_1\overline{C_2}(s)+C_2\overline{C_1}(s)+|C_2(s)|^2]\\
&\lesssim\frac{\alpha^2T^2}{L^{2}}\sum_{(k_1, k_2, k_3); \overline \Omega_k\neq 0}^\times {n_{{\rm in},\eta}}(k_1){n_{{\rm in},\eta}}(k_2){n_{{\rm in},\eta}}(k_3)\frac{|Er_1|^2+|Er_1|}{|T\overline \Omega_k|^2} \notag\\
&\quad \quad \times B^2_{k_1,k_2,k_3,k}\varphi_{\leq K}^2(k_1)\varphi_{\leq K}^2(k_2)\varphi_{\leq K}^2(k_3)\varphi_{\leq K}^2(k)\\
&\lesssim\frac{\alpha^2T^2}{L^{2}} \cdot L^{2}T^{-1}\log L\cdot \sum_{j\neq 0}\frac{sL^{-3\delta}}{j^2}\cdot L^\delta=O(\frac{t}{T_{\rm kin}}L^{-\delta}).
\end{align*}

In summary, we obtain
\begin{align*}
     \mathbf{E}|c^1_k(s)|^2=\mathbf{E}|C_1(s)|^2+\mathbf{E}|C_4(s)|^2+\mathbf{E}|C_5(s)|^2+O(\frac{t}{T_{\rm kin}}L^{-\delta}).
\end{align*}

For the term $2\mathbf{E}\mathrm{Re}\overline{c^0_k(s)}c^2_k(s)$, we assume that the root and the other branching node are decorated by $k$ and $k'$, respectively. A direct inspection of the admissible decorations shows that the two branching nodes {are either both non-degenerate or both degenerate}. As a consequence, only the configurations illustrated in Figure~\ref{fig:type2tree} are possible for the corresponding enhanced couples $\mathcal{Q}$.
Moreover, in each of these cases, the resonance relations imply that  $\overline \Omega_k+\overline \Omega_{k'}=0$ and $\sum\Gamma_1(k,s)+\sum\Gamma_1(k',s)=0$.

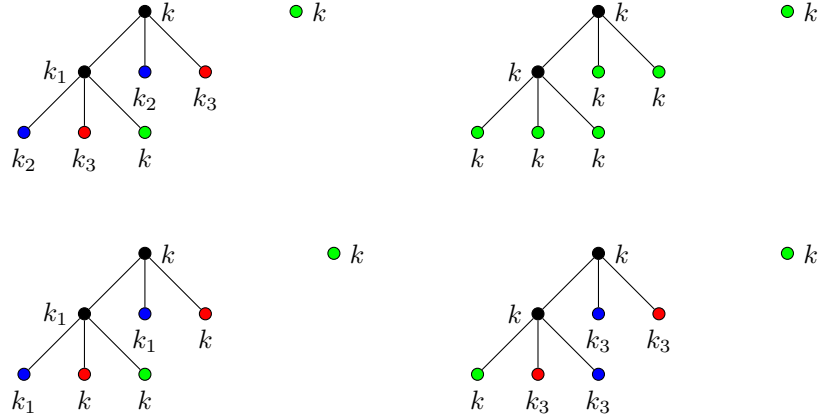
\begin{figure}
\begin{subfigure}[t]{\linewidth}
\centering
\begin{tikzpicture}[level distance=.8cm,
  level 1/.style={sibling distance=.8cm},
  level 2/.style={sibling distance=.8cm}]
\tikzstyle{hollow node}=[circle,draw,inner sep=1.6]
\tikzstyle{solid node}=[circle,draw,inner sep=1.6,fill=black]
\tikzset{
red node/.style = {circle,draw=black,fill=red,inner sep=1.6},
blue node/.style= {circle,draw = black, fill= blue,inner sep=1.6}, 
purple node/.style= {circle,draw = black, fill= purple,inner sep=1.6}, 
orange node/.style= {circle,draw = black, fill= orange,inner sep=1.6},
yellow node/.style= {circle,draw = black, fill= yellow,inner sep=1.6},
green node/.style = {circle,draw=black,fill=green,inner sep=1.6}}

\node[solid node, label = right: $k$ ] at (-5.5,.5){}
    child{node[solid node, label = left: $k_1$ ]{}
        child{node[blue node, label=below: $k_2$]{}}
        child{node[red node, label=below: $k_3$]{}}
        child{node[green node, label=below: $k$]{}}
    }
    child{node[blue node, label=below: $k_2$]{}}
    child{node[red node, label=below: $k_3$]{}}
;
\node[green node, label=right: $k$] at (-3.5,.5){};

\node[solid node, label = right: $k$ ] at (.5,.5){}
    child{node[solid node, label = left: $k$ ]{}
        child{node[green node, label=below: $k$]{}}
        child{node[green node, label=below: $k$]{}}
        child{node[green node, label=below: $k$]{}}
    }
    child{node[green node, label=below: $k$]{}}
    child{node[green node, label=below: $k$]{}}
;
\node[green node, label=right: $k$] at (3,.5){};

\node[solid node, label = right:{$k$}] at (-5.5,-2.7){}
    child{node[solid node, label = left: $k_1$ ]{}
        child{node[blue node, label=below: $k_1$]{}}
        child{node[red node, label=below: $k$]{}}
        child{node[green node, label=below: $k$]{}}
    }
    child{node[blue node, label=below: $k_1$]{}}
    child{node[red node, label=below: $k$]{}}
;
\node[green node, label=right: $k$] at (-3,-2.7){};

\node[solid node, label = right:{$k$}] at (.5,-2.7){}
    child{node[solid node, label = left: $k$ ]{}
        child{node[green node, label=below: $k$]{}}
        child{node[red node, label=below: $k_3$]{}}
        child{node[blue node, label=below: $k_3$]{}}
    }
    child{node[blue node, label=below: $k_3$]{}}
    child{node[red node, label=below: $k_3$]{}}
;
\node[green node, label=right: $k$] at (3,-2.7){};

\end{tikzpicture}
\end{subfigure}
\caption{All possible types of couples $\mathcal{Q}$ of order $n=2$, where the second branching node is the left child of the root. Symmetry should also be taken into account.
Nodes with the same color are paired with each other. The first configuration in the first row {corresponds} to the non-degenerate case, while {the remaining three correspond to degenerate cases}.}
\label{fig:type2tree}
\end{figure}

Now, similarly to \eqref{bd:ck1}, we distinguish four cases: the non-degenerate case with $\overline \Omega_k,\overline \Omega_{k'}\neq 0$; the non-degenerate case with $\overline \Omega_k=\overline \Omega_{k'}=0$; the fully degenerate case; and the partially degenerate case with $k\neq k_1$.

We begin with the non-degenerate case with $\overline \Omega_k,\overline \Omega_{k'}\neq 0$. 
In this case, the time integral $\mathcal{A}_{\mathcal{T}}$ takes the form
 \begin{align*}
     &\int_0^se^{2\pi i[\overline \Omega_kTs_1+\sum\Gamma_1(k,s_1)]}\int_0^{s_1}e^{2\pi i[\overline \Omega_{k'}Ts_2+\sum\Gamma_1(k',s_2)]}\dif s_2\dif s_1\\
     &=  \int_0^se^{2\pi i[\overline \Omega_kTs_1+\sum\Gamma_1(k,s_1)]}\notag\\ 
     &\quad\quad\quad\quad\times\Big[\frac{e^{2\pi i[\overline \Omega_{k'}Ts_1+\sum\Gamma_1(k',s_1)]}-1}{2\pi i\overline \Omega_{k'}T}-\int_0^{s_1}\frac{\partial_t\sum\Gamma_1(k',s_2)}{\overline \Omega_{k'}T}e^{2\pi i[\overline \Omega_{k'}Ts_2+\sum\Gamma_1(k',s_2)]}\dif s_2\Big]\dif s_1\\
      &=\frac1{2\pi i\overline \Omega_{k'}T}  \int_0^s1-e^{2\pi i[\overline \Omega_kTs_1+\sum\Gamma_1(k,s_1)]}\dif s_1+\widetilde {Er}_2=\frac{-s}{2\pi i\overline \Omega_{k}T}  -\frac{  e^{2\pi i\overline \Omega_kTs}-1}{4\pi^2|\overline \Omega_{k}T|^2}  +Er_2.
 \end{align*}
where we used the fact that $\overline \Omega_k+\overline \Omega_{k'}=0$ and $\sum\Gamma_1(k,s)+\sum\Gamma_1(k',s)=0$ among all the possible pairs. Here the error term  $Er_2$  is defined as 
 \begin{align*}
     Er_2&:=-\frac{  e^{2\pi i[\overline \Omega_kTs+\sum\Gamma_1(k,s)]}- e^{2\pi i\overline \Omega_kTs}}{4\pi^2|\overline \Omega_{k}T|^2}\\
     &\quad-\int_0^se^{2\pi i[\overline \Omega_kTs_1+\sum\Gamma_1(k,s_1)]} \int_0^{s_1}\frac{\partial_t\sum\Gamma_1(k',s_2)}{\overline \Omega_{k'}T}e^{2\pi i[\overline \Omega_{k'}Ts_2+\sum\Gamma_1(k',s_2)]}\dif s_2 \dif s_1\\
     &\quad+\int_0^s \frac1{2\pi i\overline \Omega_{k'}T}  \frac{\partial_t\sum\Gamma_1(k,s_1)}{\overline \Omega_{k}T}   e^{2\pi i[\overline \Omega_kTs_1+\sum\Gamma_1(k,s_1)]}\dif s_1.
 \end{align*}
We now estimate $Er_2$. For the second term above, we apply Proposition~\ref{prop:est:af} with
  $$f=\int_0^{s_1}\partial_t\sum\Gamma_1(k',s_2)e^{2\pi i[\overline \Omega_{k'}Ts_2+\sum\Gamma_1(k',s_2)]}\dif s_2.$$By \eqref{bd:gamma'}, we have
 \begin{align*}
     |f|+|\partial_tf|\lesssim \sup_k \|\partial_t\Gamma_1(k,s)\|_{C^0_s}\leq L^{-3\delta}.
 \end{align*}
 Hence, Proposition~\ref{prop:est:af} yields
\begin{align*}
   &\left| \int_0^se^{2\pi i[\overline \Omega_kTs_1+\sum\Gamma_1(k,s_1)]} \int_0^{s_1}\frac{\partial_t\sum\Gamma_1(k',s_2)}{\overline \Omega_{k'}T}e^{2\pi i[\overline \Omega_{k'}Ts_2+\sum\Gamma_1(k',s_2)]}\dif s_2 \dif s_1\right|\\
   &\qquad\lesssim s\frac{\sup_k \|\partial_t\Gamma_1(k,s)\|_{C^0_s}}{|T\overline\Omega_k|^2}..
\end{align*}
Consequently,
 \begin{align*}
     | Er_2|\lesssim \sup_k\sup_{s\in[0,1]}\frac{|\sum\Gamma_1(k,s)|+s\|\partial_t\Gamma_1(k,s)\|_{C^0_s}}{|\overline \Omega_{k}T|^2}\lesssim \frac{ s \|\partial_t\Gamma_1(k,s)\|_{C^0_sL_k^\infty}}{|\overline \Omega_{k}T|^2}\lesssim \frac{ sL^{-3\delta}}{|\overline \Omega_{k}T|^2}.
 \end{align*}
 
The contribution of this case to $2\,\mathbf{E}\mathrm{Re}\,\overline{c^0_k(s)}c^2_k(s)$ is

\begin{align*}
{-}2\bigg(\frac{\alpha T}{L}\bigg)^2&\sum\limits_{(k_1,k_2, k_{3}),\overline \Omega_k\neq  0}^{\times}{n_{{\rm in},\eta}}(k_1){n_{{\rm in},\eta}}(k_2){n_{{\rm in},\eta}}(k_3){n_{{\rm in},\eta}}(k)\left(\frac{1}{{n_{{\rm in},\eta}}(k_1)}-\frac{1}{{n_{{\rm in},\eta}}(k_2)}+\frac{1}{{n_{{\rm in},\eta}}(k_3)}\right)\\
&\times \Big(\left|\frac{\sin \pi \overline \Omega _kTs}{\pi\overline  \Omega _kT}\right|^2+2\mathrm{Re}(Er_2) \Big)
B^2_{k_1,k_2,k_3,k}\varphi_{\leq K}^2(k_1)\varphi_{\leq K}^2(k_2)\varphi_{\leq K}^2(k_3)\varphi_{\leq K}^2(k),
\end{align*}
where the term with $\mathrm{Re}(Er_2) $ is bounded by using 3 number counting estimate in Lemma \ref{counting} by:
\begin{align*}
    \frac{\alpha^2 T^2}{L^{2}}\cdot L^{2}T^{-1}\log L\cdot sL^{-3\delta}\cdot\sum_{j\neq 0}\frac{1}{j^2}\cdot L^\delta=O(\frac{t}{T_{\rm kin}}L^{-\delta}).
\end{align*}

Second, we consider the non-degenerate case with $\overline \Omega_k=\overline \Omega_{k'}=0$ and $k\neq k_1,k_3$. Its contribution is bounded by
\begin{align*}
\bigg(\frac{\alpha T}{L}\bigg)^2&\sum\limits_{(k_1,k_2, k_{3}),\overline \Omega_k= 0}^{\times}{n_{{\rm in},\eta}}(k_1){n_{{\rm in},\eta}}(k_2){n_{{\rm in},\eta}}(k_3){n_{{\rm in},\eta}}(k)\left(\frac{1}{{n_{{\rm in},\eta}}(k_1)}+\frac{1}{{n_{{\rm in},\eta}}(k_2)}+\frac{1}{{n_{{\rm in},\eta}}(k_3)}\right)\\
&\times 
B^2_{k_1,k_2,k_3,k}\varphi_{\leq K}^2(k_1)\varphi_{\leq K}^2(k_2)\varphi_{\leq K}^2(k_3)\varphi_{\leq K}^2(k),
\end{align*}
which is bounded by $O\!\left(\frac{t}{T_{\rm kin}}L^{-\delta}\right)$ by the same argument as for $C_3$. Here all terms are nonnegative, as we take absolute values. 

Third, we consider the fully degenerate case with  $k= k_1=k_3=k_2$, its contribution is 
\begin{align*}
   2i^2(1-1+1)\frac{\alpha^2 T^2s^2}{L^{2}}B^2_{k,k,k,k}\varphi_{\leq K}^8(k){n_{{\rm in},\eta}}(k)|{n_{{\rm in},\eta}}(k)|^2=-\mathbf{E}|C_4(s)|^2,
\end{align*}
where the factor $1-1+1$ corresponds to the three possible positions (left, middle, or right) of the second branching node relative to the root. The prefactor $2$ reflects that there are exactly two admissible pairings, since the root node can only pair with two leaves carrying the $+$ sign (see Figure~\ref{fig:type2tree}).

Finally, we consider the partially degenerate case with $k_1\neq k$ or $k_3\neq k$. Its contribution is
\begin{align*}
    4i^2(1-1+1)\frac{\alpha^2 T^2s^2}{L^{2}}\sum_{k_1\neq k}B^2_{k_1,k_1,k,k}\varphi_{\leq K}^4(k_1)\varphi_{\leq K}^4(k){n_{{\rm in},\eta}}(k)|{n_{{\rm in},\eta}}(k_1)|^2=-\mathbf{E}|C_5(s)|^2,
\end{align*}
where the factor $4$ comes from the four admissible pairings, corresponding to permutations of the two $+$ leaves in Figure~\ref{fig:type2tree}.

Combining the above estimates, we obtain
\begin{align*}\mathbf{E}|c^1_k(s)|^2+2\mathrm{Re}\mathbf{E}[\overline{c^0_k(s)}c^2_k(s)]=\frac{2\alpha^2t^2}{L^{2}}\cdot\mathscr{S}_{t}({n_{{\rm in},\eta}}) +O(\frac{t}{T_{\rm kin}}L^{-\delta}),\end{align*} where we recall  \begin{align*}
 \mathscr S_t(\phi):&=\sum\limits_{\substack{ k-k_1+k_2-k_3=0,\\ \overline \Omega_k\neq0}}\phi_k \phi_{k_1} \phi_{k_2} \phi_{k_3}  \left[ \frac{1}{\phi_k} - \frac{1}{\phi_{k_1}} + \frac{1}{\phi_{k_2}} - \frac{1}{\phi_{k_3}} \right]\left| \frac{\sin(\pi t\overline \Omega_k)}{\pi t \overline \Omega_k} \right|^2\notag\\
  &\quad\quad\quad\times
|k_1|^{2\beta}|k_2|^{2\beta}|k_3|^{2\beta}|k|^{2\beta}\varphi_{\leq K}^2(k_1)\varphi_{\leq K}^2(k_2)\varphi_{\leq K}^2(k_3)\varphi_{\leq K}^2(k).
\end{align*} 
Consequently,
\begin{align*}\mathbf{E}|\widehat{u}_{\rm tr}(k,t)|^2=\varphi_{\leq K}^2(k){n_{{\rm in},\eta}}(k)+\frac{2\alpha^2t^2}{L^{2}}\mathscr{S}_{t}({n_{{\rm in},\eta}})(k)+O(\frac{t}{T_{\rm kin}}L^{-\delta }).\end{align*}

By the asymptotic estimate in Proposition~\ref{asymptotic}, we further obtain \begin{align*}\mathbf{E}|\widehat{u}_{\rm tr}(k,t)|^2=\varphi_{\leq K}^2(k){n_{{\rm in},\eta}}(k)+\frac{2\alpha^2t^2}{L^{2}}\mathscr{K}_{t}({n_{{\rm in},\eta}})(k)+O(\frac{t}{T_{\rm kin}}L^{-\delta}),\end{align*} 
 By applying Proposition \ref{lem:converge}, we finish the proof.
\end{proof}

\section{An acceptable energy inequality}\label{sec:acceptineq}
In the previous sections, we  obtained the kinetic expansion for the truncated MMT equation. 
From the present section on, we shift our focus to the high-frequency component of the full MMT equation.
The analysis is based on a suitable energy inequality, as stated in Theorem~\ref{thm:energy inequality}. 
The key idea is to show that the increment of the high-frequency energy can be controlled by an $L^\infty$-type norm of the solution. 
This control will play a crucial role when dealing with solutions arising from random initial data.

Throughout this section and the next section, we denote the nonlinear term  by 
\begin{align*}
    \mathscr{B}(u,v,w):= -\lambda^2 |\nabla|^\beta\left[|\nabla|^\beta u \cdot|\nabla|^\beta\overline v   \cdot|\nabla|^\beta w\right].
\end{align*}

Let ${q\geq10}$ {be an integer}. We define the operator $J^{{{q}}}$ via the Fourier multiplier
\begin{align*}
    m_{{{q}}}(\xi):=\varphi_{\geq 0}(\xi)\langle\xi\rangle^{{{q}}}.
\end{align*}
We then introduce the high-frequency energy profile on the time interval $[T_1,T_2]$:
\begin{align*}
    A(t):=\int |J^{{{q}}}u(t)|^2\dif x\leq \|u(t)\|_{H^{{{q}}}}^2\leq A^2,
\end{align*}
where the last inequality holds under the assumption of 
Theorem~\ref{thm:energy inequality}.  

Our goal is to derive the following energy increment estimate for $A(t)$:
 \begin{proposition}\label{prop:highenergy} 
Under the assumption of Theorem \ref{thm:energy inequality}, it holds for $t_1\leq t_2\in[T_1,T_2]$ that: 
     \begin{align*}
  |   A(t_2)-A(t_1)|&\lesssim\lambda^2[\|u(t_1)\|_{H^{{{q}}}}\epsilon(t_1)]^2+\lambda^2[\|u(t_2)\|_{H^{{{q}}}}\epsilon(t_2)]^2\\
  &\quad+\lambda^4A^2\int_{t_1}^{t_2}\epsilon^4(s)\dif s+1_{\{\sigma\in(0,1)\}}{\lambda^2\int_{t_1}^{t_2}\epsilon^2(s)\|u(s)\|_{H^{{{q}}-\frac12+\beta}}^2\dif s}.
\end{align*}
 \end{proposition}

The proof of Proposition~\ref{prop:highenergy}, which forms the core of this section, will be presented later. 
Assuming this result, Theorem~\ref{thm:energy inequality} follows directly by combining it with the conservation of mass.

\begin{proof}[Proof of Theorem \ref{thm:energy inequality}]
In Proposition~\ref{prop:highenergy}, we derived an energy increment estimate 
for the high-frequency component. 
In the case $\sigma\in(1,2]$, using the smallness condition 
$\lambda^2 \epsilon^2(t)\ll 1$, we obtain 
 \begin{align*}
    A(t_2)&\leq A(t_1)+ \|u(t_1)\|_{H^{{q}}}^2+\frac12\|u(t_2)\|_{H^{{q}}}^2+\lambda^4A^2\int_{t_1}^{t_2}\epsilon^4(s)\dif s\\
    &\leq  2\|u(t_1)\|_{H^{{q}}}^2+\frac12\|u(t_2)\|_{H^{{q}}}^2+\lambda^4A^2\int_{t_1}^{t_2}\epsilon^4(s)\dif s.
\end{align*}

For the low-frequency part, we use the conservation of mass to deduce
\begin{align*}
    \|P_{\leq 4}u(t_2)\|_{H^{{q}}}^2&\lesssim \|u(t_2)\|_{L^2}^2=\|u(t_1)\|_{L^2}^2 \leq \|u(t_1)\|_{H^{{q}}}^2.
\end{align*}
 
Combining the high- and low-frequency estimates, we obtain
    \begin{align*}
        \|u(t_2)\|_{H^{{q}}}^2&\leq  A(t_2)+\|P_{\leq 4}u(t_2)\|_{H^{{q}}}^2\\
        &\leq   C\|u(t_1)\|_{H^{{q}}}^2+\frac12\|u(t_2)\|_{H^{{q}}}^2+\lambda^4A^2\int_{t_1}^{t_2}\epsilon^4(s)\dif s.
    \end{align*}
  Absorbing the term $\tfrac12\|u(t_2)\|_{H^{{q}}}^2$ into the left-hand side yields the desired estimate.

The case $\sigma\in(0,1)$ follows by the same argument.
\end{proof}

\subsection{ Proof of  Proposition \ref{prop:highenergy}}
 
Since $u$ solves \eqref{MMT}, a standard energy computation yields, for $t\in[T_1,T_2]$,
\begin{align*}
   i( A(t_2)-A(t_1))&=\int_{t_1}^{t_2}\langle i\partial_tJ^{{{q}}}u,J^{{{q}}}u\rangle+\langle J^{{{q}}}u, -i\partial_tJ^{{{q}}} u\rangle\dif t\\
    &=\int_{t_1}^{t_2} \langle \mathscr{B}(u,u,u),{(J^{{{q}}})^2}{u} \rangle- \langle {(J^{{{q}}})^2}u,\mathscr{B}({u},{u},u) \rangle\dif t.
\end{align*}
Here $\langle\cdot,\cdot\rangle$ denotes the standard complex $L^2$ inner product on $\mathbb{T}_L$. Passing to Fourier variables, we obtain
\begin{align*}
  |   A(t_2)-A(t_1)|&\lesssim\lambda^2\left|\int_{t_1}^{t_2} Q(u,\overline u,u,\overline u)\dif t\right|,
\end{align*}
where \begin{align*}
Q(u,\overline u,u,\overline u):=\frac{1}{L^{3}}\sum_{\xi_1-\xi_2+\xi_3-\xi_4=0}\widehat{u}_{\xi_1}\overline{\widehat{u}}_{\xi_2}\widehat{u}_{\xi_3}\overline{\widehat{u}}_{\xi_4}B_{\xi_1,\xi_2,\xi_3,\xi_4}(m_{{{q}}}(\xi_3)^2-m_{{{q}}}(\xi_4)^2).
\end{align*}
Here we recall  $B_{\xi_1,\xi_2,\xi_3,\xi_4}=|\xi_1|^\beta|\xi_2|^\beta|\xi_3|^\beta|\xi_4|^\beta.$ By symmetry under permutations of the indices, we rewrite $Q$ as
 \begin{align*}
Q(u,\overline u,u,\overline u)&=\frac{1}{2L^{3}}\sum_{\xi_1-\xi_2+\xi_3-\xi_4=0}\widehat{u}_{\xi_1}\overline{\widehat{u}}_{\xi_2}\widehat{u}_{\xi_3}\overline{\widehat{u}}_{\xi_4}B_{\xi_1,\xi_2,\xi_3,\xi_4}\\
&\quad\quad\quad\times(m_{{{q}}}(\xi_1)^2-m_{{{q}}}(\xi_2)^2+m_{{{q}}}(\xi_3)^2-m_{{{q}}}(\xi_4)^2).
\end{align*}

Then it suffices to establish an estimate for $Q$. 
We introduce the notation $$\varphi_{\underline{k}}(\xi):=\varphi_{k_1}(\xi_1)\varphi_{k_2}(\xi_2)\varphi_{k_3}(\xi_3)\varphi_{k_4}(\xi_4)$$ and let $\tilde{k}_1 \ge \tilde{k}_2 \ge \tilde{k}_3 \ge \tilde{k}_4$ 
denote the decreasing rearrangement of the set 
$\{k_1,k_2,k_3,k_4\} \subset \mathbb{Z}$.

Let $\epsilon' > 0$ be a small constant. 
For any $s \in [t_1,t_2]$ and $k \in \mathbb{Z}$, define
    \begin{align*}
        A_k(s):=\sum_{l\in\mathbb{Z}}2^{-\epsilon'|l-k| }\|P_lu(s)\|_{H^{{{q}}}}\sim \sum_{l\in\mathbb{Z}}2^{-\epsilon'|l-k|}\cdot 2^{\max\{l, 0\}{{q}}}\|P_lu(s)\|_{L^2}.
    \end{align*}
   By Young's inequality, for any $s\in[t_1,t_2]$,
    \begin{align*}
        \sum_{k\in\mathbb{Z}} |A_k(s)|^2\lesssim\|u(s)\|_{H^{{{q}}}}^2\lesssim A^2 .
    \end{align*}
In particular, for $\sigma \in (0,1)$, define
    \begin{align*}
        \tilde{A}_k(s):&=\sum_{l\in\mathbb{Z}}2^{-\epsilon'|l-k| }\|P_lP_{\geq-10}u(s)\|_{H^{{{q}}-\frac12+\beta}}\\
        &\sim \sum_{l\in\mathbb{Z}}2^{-\epsilon'|l-k|}\cdot 2^{\max\{l, 0\}\cdot ({{q}}-\frac12+\beta)}\|P_lP_{\geq-10}u(s)\|_{L^2}.
    \end{align*}
{
   By Young's inequality, for any $s\in[t_1,t_2]$,
    \begin{align*}
        \sum_{k\in\mathbb{Z}} |\tilde{A}_k(s)|^2
        \lesssim\|P_{\geq-10}u(s)\|_{H^{{{q}}-\frac12+\beta}}^2
        \lesssim\|u(s)\|_{H^{{{q}}-\frac12+\beta}}^2,
    \end{align*}
where the implicit constant is independent of the size of the torus $L$.
}

  By the same argument as \cite[Lemma 5.2, (5.20)]{DIP25a}, 
to obtain the desired bound in Proposition~\ref{prop:highenergy}, 
it suffices to prove that \begin{align}
2^{\epsilon'(\max\{\tilde k_3,0\}-\min\{\tilde k_4,0\})}&\left|\int_{t_1}^{t_2} Q(P_{k_1}u,P_{k_2}\overline u,P_{k_3}u,P_{k_4}\overline u)\dif t\right|\notag\\
&\lesssim [A_l(t_1)\epsilon(t_1)]^2+[A_l(t_2)\epsilon(t_2)]^2+\lambda^2\int_{t_1}^{t_2}A_l^2(s)\epsilon^4(s)\dif s\notag\\
&\quad+1_{\{\sigma\in(0,1)\}}\int_{t_1}^{t_2}{\tilde{A}_l^2(s)\epsilon^2(s)}\dif s,\label{bd:9.2mid}
\end{align}
for any $(k_1,k_2,k_3,k_4)\in\mathbb{Z}_l:=\{(k_1,k_2,k_3,k_4)\in\mathbb{Z}^4,l=\max\{k_1,k_2,k_3,k_4\} \}$.  Here we recall that $P_k$ are the  Littlewood-Paley projections.

In what follows, we use the shorthand

\begin{align}
   \Psi(\xi)&:=  \Psi(\xi_1,\xi_2,\xi_3,\xi_4)=m_{{{q}}}(\xi_1)^2-m_{{{q}}}(\xi_2)^2+m_{{{q}}}(\xi_3)^2-m_{{{q}}}(\xi_4)^2,\notag\\
   \Omega^{(\sigma)}(\xi)&:= \Omega^{(\sigma)}(\xi_1,\xi_2,\xi_3,\xi_4)=|\xi_1|^\sigma-|\xi_2|^\sigma+|\xi_3|^\sigma-|\xi_4|^\sigma,\label{def:psiphi}
\end{align}
and let \begin{align}
Q^{(\sigma)}(u,\overline u,u,\overline u):=\frac{1}{L^{3}}\sum_{\xi_1-\xi_2+\xi_3-\xi_4=0}\widehat{u}_{\xi_1}\overline{\widehat{u}}_{\xi_2}\widehat{u}_{\xi_3}\overline{\widehat{u}}_{\xi_4}B_{\xi_1,\xi_2,\xi_3,\xi_4}\frac{\Psi(\xi)}{\Omega^{(\sigma)}(\xi)}.\label{nomoralform}
\end{align}
By writing 
\begin{align*}
    Q^{(\sigma)}&(P_{k_1}u,P_{k_2}\overline u,P_{k_3}u,P_{k_4}\overline u)(t_2)- Q^{(\sigma)}(P_{k_1}u,P_{k_2}\overline u,P_{k_3}u,P_{k_4}\overline u)(t_1)\\
    &=\int_{t_1}^{t_2}\partial_tQ^{(\sigma)}(P_{k_1}u,P_{k_2}\overline u,P_{k_3}u,P_{k_4}\overline u)(s)\dif s,
\end{align*}
together with the fact that $u$  is the
solution to 
\eqref{MMT}, we obtain 
\begin{align}
& \left| \int_{t_1}^{t_2}Q (P_{k_1}u,P_{k_2}\overline u,P_{k_3}u,P_{k_4}\overline u)(s)\dif s\right|\notag\\
  &\leq | Q^{(\sigma)}(P_{k_1}u,P_{k_2}\overline u,P_{k_3}u,P_{k_4}\overline u)(t_2)- Q^{(\sigma)}(P_{k_1}u,P_{k_2}\overline u,P_{k_3}u,P_{k_4}\overline u)(t_1)|\notag\\
 &+  \bigg{|}\int_{t_1}^{t_2}Q^{(\sigma)}  (P_{k_1}\mathscr{B}(u,u,u),P_{k_2}\overline u,P_{k_3}u,P_{k_4}\overline u)(s)-Q^{(\sigma)}  (P_{k_1}u,P_{k_2}\mathscr{B}(\overline u,\overline u,\overline u),P_{k_3}u,P_{k_4}\overline u)(s)\notag\\
 &\quad +Q^{(\sigma)}  (P_{k_1} u,P_{k_2}\overline u,P_{k_3}\mathscr{B}(u,u,u),P_{k_4}\overline u)(s)-Q^{(\sigma)}  (P_{k_1} u ,P_{k_2}\overline u,P_{k_3}u,P_{k_4}\mathscr{B}(\overline u,\overline u,\overline u))(s)\dif s\bigg{|}.\label{bd:exp:Q}
\end{align}

We note that similar expressions appear in \cite[Section 5]{DIP25a} 
and \cite[Section 3.1]{MMT97}. 
In what follows, we estimate each term in \eqref{bd:exp:Q}, 
distinguishing between the cases $\sigma \in (0,1)$ and $\sigma \in (1,2]$.

\subsubsection{Case 1: $\sigma=2$.}

Although this case is closely related to the range $\sigma\in(1,2)$
considered below, we present a separate and more detailed argument for
clarity. Under the constraint
\[
    \xi_1-\xi_2+\xi_3-\xi_4=0,
\]
we have
\begin{align*}
    \Omega^{(2)}(\xi)
    =2(\xi_3-\xi_4)(\xi_4-\xi_1).
\end{align*}

We first consider the boundary terms in \eqref{bd:exp:Q}, which are of
the form
\begin{align*}
&Q^{(2)}
  (P_{k_1}u,P_{k_2}\overline u,P_{k_3}u,P_{k_4}\overline u)
\\
&\qquad
=\frac{1}{L^{3}}
\sum_{\xi_1-\xi_2+\xi_3-\xi_4=0}
\widehat{u}_{\xi_1}
\overline{\widehat{u}}_{\xi_2}
\widehat{u}_{\xi_3}
\overline{\widehat{u}}_{\xi_4}
B_{\xi_1,\xi_2,\xi_3,\xi_4}
\frac{\Psi(\xi)}{\Omega^{(2)}(\xi)}
\varphi_{\underline k}(\xi).
\end{align*}
The quotient $\Psi/\Omega^{(2)}$ is understood through the smooth
extension constructed below.

Using the constraint
$\xi_1-\xi_2+\xi_3-\xi_4=0$ and applying the mean value theorem twice,
we obtain, for ${{q}}\geq10$,
\begin{align}
\Psi(\xi)
&=(\xi_1-\xi_2)
  \int_0^1
  (m_{{{q}}}^2)'
  \bigl(\theta\xi_1+(1-\theta)\xi_2\bigr)
  \dif\theta 
-(\xi_4-\xi_3)
  \int_0^1
  (m_{{{q}}}^2)'
  \bigl(\theta\xi_4+(1-\theta)\xi_3\bigr)
  \dif\theta
\notag\\
&=(\xi_1-\xi_2)
  \int_0^1
  \Big[
      (m_{{{q}}}^2)'
      \bigl(\xi_2+\theta(\xi_1-\xi_2)\bigr) 
      -(m_{{{q}}}^2)'
      \bigl(\xi_3+\theta(\xi_1-\xi_2)\bigr)
  \Big]
  \dif\theta
\notag\\
&=(\xi_1-\xi_2)(\xi_2-\xi_3)
  \int_0^1\int_0^1
  (m_{{{q}}}^2)''
  \bigl(
      \eta\xi_2+(1-\eta)\xi_3
      +\theta(\xi_1-\xi_2)
  \bigr)
  \dif\eta\dif\theta
\notag\\
&=(\xi_1-\xi_2)(\xi_2-\xi_3)
  \int_0^1\int_0^1
  (m_{{{q}}}^2)''
  \Big(
      (1-\theta)\eta\xi_2
      +(1-\theta)(1-\eta)\xi_3
\notag\\
&\hspace{5.5cm}
      +\eta\theta\xi_1
      +\theta(1-\eta)\xi_4
  \Big)
  \dif\eta\dif\theta.
\label{bd:phi=}
\end{align}
 Hence, 
\begin{align*}
\frac{\Psi(\xi)}{\Omega^{(2)}(\xi)}
&=
\frac12
\int_0^1\int_0^1
(m_{{{q}}}^2)''
\Big(
    (1-\theta)\eta\xi_2
    +(1-\theta)(1-\eta)\xi_3
\\
&\hspace{4.6cm}
    +\eta\theta\xi_1
    +\theta(1-\eta)\xi_4
\Big)
\dif\eta\dif\theta.
\end{align*}
The right-hand side provides a smooth extension of
$\Psi/\Omega^{(2)}$ across the resonant set.

Together with the bound
\[
    |(m_{{{q}}}^2)''(\xi)|
    \lesssim |\xi|^{2{{q}}-2},
\]
this gives 
\begin{align*}
\left|
    \varphi_{\underline{k}}(\xi)
    \frac{\Psi(\xi)}{\Omega^{(2)}(\xi)}
\right|
\lesssim
2^{(2{{q}}-2)\widetilde{k}_1}.
\notag
\end{align*}
It follows that
\begin{align*}
&\left|
B_{\xi_1,\xi_2,\xi_3,\xi_4}
\frac{\Psi(\xi)}{\Omega^{(2)}(\xi)}
\varphi_{\underline k}(\xi)
\right| 
\lesssim
2^{(2{{q}}-2)\widetilde{k}_1}
2^{\beta(
    \widetilde{k}_1+\widetilde{k}_2
    +\widetilde{k}_3+\widetilde{k}_4)}
\lesssim
2^{(2{{q}}-2+2\beta)\widetilde{k}_1}
2^{\beta(\widetilde{k}_3+\widetilde{k}_4)}.
\notag
\end{align*}

We next verify the corresponding derivative bounds. Since
$(m_{{{q}}}^2)''$ has polynomial growth, for every derivative order
needed below,
\[
    \big|
        \xi^\alpha
        \partial_\xi^\alpha
        (m_{{{q}}}^2)''
    \big|
    \lesssim
    |\xi|^{2{{q}}-2}.
\]
Denoting the affine expression in \eqref{bd:phi=} by
$X_{\eta,\theta}(\xi)$, the chain rule gives
\begin{align*}
\left|
\xi_1\partial_{\xi_1}
\left[
    \frac{\Psi}{\Omega^{(2)}}
\right]
\right|
&\lesssim
\int_0^1\int_0^1
|\eta\theta\xi_1|\,
\left|
    (m_{{{q}}}^2)'''
    \bigl(X_{\eta,\theta}(\xi)\bigr)
\right|
\dif\eta\dif\theta
\\
&\lesssim
\int_0^1\int_0^1
|\eta\theta\xi_1|\,
|X_{\eta,\theta}(\xi)|^{2{{q}}-3}
\dif\eta\dif\theta \lesssim
2^{(2{{q}}-2)\widetilde{k}_1}.
\end{align*}
By repeated application of the chain rule, the same argument yields,
for every relevant multi-index $\alpha$,
\begin{align}
\left|
\xi^\alpha\partial_{\xi}^\alpha
\left[
    \frac{\Psi}{\Omega^{(2)}}
\right]
\right|
\lesssim
2^{(2{{q}}-2)\widetilde{k}_1}.
\label{bd:example}
\end{align}

Each localized factor
$|\xi_j|^\beta\varphi_{k_j}(\xi_j)$ satisfies the corresponding
symbol estimates. Hence, for every relevant multi-index $\alpha$,
\begin{align*}
&\left|
\xi^\alpha\partial_{\xi}^\alpha
\left(
B_{\xi_1,\xi_2,\xi_3,\xi_4}
\frac{\Psi}{\Omega^{(2)}}
\varphi_{\underline k}(\xi)
\right)
\right| 
\lesssim
2^{(2{{q}}-2+2\beta)\widetilde{k}_1}
2^{\beta(\widetilde{k}_3+\widetilde{k}_4)}.
\end{align*}
By Lemma~\ref{lem:sinfty}, it follows that
\begin{align}
\left\|
B_{\xi_1,\xi_2,\xi_3,\xi_4}
\frac{\Psi}{\Omega^{(2)}}
\varphi_{\underline k}(\xi)
\right\|_{\widetilde S^\infty}
\lesssim
2^{(2{{q}}-2+2\beta)\widetilde{k}_1}
2^{\beta(\widetilde{k}_3+\widetilde{k}_4)}.
\label{bd:bpsi/phi2}
\end{align}

By Lemma~\ref{lem:holder}, for $\epsilon'>0$ sufficiently small,
we obtain
\begin{align*}
&2^{\epsilon'
   (\max\{\widetilde k_3,0\}
    -\min\{\widetilde k_4,0\})}
\left|
Q^{(2)}
(P_{k_1}u,P_{k_2}\overline u,
 P_{k_3}u,P_{k_4}\overline u)(s)
\right|\notag
\\
&\quad
\lesssim
A_{\widetilde k_1}(s)
A_{\widetilde k_2}(s)
\epsilon^2(s)
2^{\beta\widetilde{k}_3
   -(2-\epsilon')\max\{\widetilde{k}_3,0\}} 
2^{\beta\widetilde{k}_4
   -2\max\{\widetilde{k}_4,0\}
   -\epsilon'\min\{\widetilde k_4,0\}}.
\end{align*}
Here we used the fact that the frequency constraint implies $ 2^{\widetilde k_1}
    \sim
    2^{\widetilde k_2}.$ 
Consequently, with
$l=\widetilde k_1=\max\{k_1,k_2,k_3,k_4\}$, the slow variation of the
frequency envelope gives
\[
    A_{\widetilde k_2}(s)\lesssim A_l(s).
\]
Choosing $0<\epsilon'<\beta$ sufficiently small, all the remaining
dyadic factors are bounded. Evaluating the above estimate at
$s=t_1$ and $s=t_2$, we obtain the first two terms on the right-hand
side of \eqref{bd:9.2mid}. 

Next, we consider the space-time integrals in \eqref{bd:exp:Q}.  
By symmetry, it suffices to focus on the first term:
 \begin{align*}
     \int_{t_1}^{t_2}Q^{(2)}  (P_{k_1}\mathscr{B}(u,u,u),P_{k_2}\overline u,P_{k_3}u,P_{k_4}\overline u)\dif s.
 \end{align*}
 We denote $\mathscr{B}(u):=\mathscr{B}(u,u,u)$, and establish the following estimates:
 \begin{align}
    2^{({{q}}-2\beta)\max\{k_1,0\}} \|P_{k_1}  \mathscr{B}(u)\|_{L^2}\lesssim   \lambda^2 A_{k_1}\epsilon^2,\ \
   \|P_{k_1}  \mathscr{B}(u)\|_{L^\infty}\lesssim \lambda^22^{\beta\max\{ k_1,0\}}  \epsilon^3,\label{bd:pkbu}
 \end{align}
 where we suppress the time dependence and write $\epsilon(t)=\epsilon$.

We first consider $k \geq 0$ and $p \in \{2, \infty\}$. Then we have
 \begin{align}
     \|P_k(uv)\|_{L^p}&\lesssim \|P_k(P'_kuP_{\leq k-4}v)\|_{L^p}+ \|P_k(P_{\leq k-4}uP'_kv)\|_{L^p}+\sum_{a,b\geq k-4,|a-b|\leq 8}\|P_k(P_auP_bv)\|_{L^p}\notag\\
     &\lesssim \sum_{a\geq k-4}\|P_au\|_{L^p}\|v\|_{\widetilde{W}^0}+\sum_{a\geq k-4}\|P_av\|_{L^p}\|u\|_{\widetilde{W}^0},\label{bd:pkuv}
 \end{align}
 where we place the low frequency terms  in $L^\infty$ and place the high frequency terms in $L^p$. We recall the definition of $\widetilde{W}^{0}$ in Section \ref{sec:pre}.
In particular, by \eqref{bd:thm9.1} we have 
 \begin{align*}
      \|P_k(|\nabla|^\beta u\cdot |\nabla|^\beta \overline u)\|_{L^p}\lesssim \sum_{a\geq k-4}\|P_a|\nabla|^\beta u\|_{L^p}\epsilon.
 \end{align*}
 
 This implies that
 \begin{align*}
    \|\ ||\nabla|^\beta u|^2\|_{\widetilde{W}^1}&\lesssim \epsilon^2+\sum_{k\geq0}2^k\sum_{a\geq k-4}\|P_a|\nabla|^\beta u\|_{L^\infty}\epsilon\\
    &\lesssim \epsilon^2+\sum_{a\geq-4}\sum_{0\leq k\leq a+4} 2^k \|P_a|\nabla|^\beta u\|_{L^\infty}\epsilon\lesssim \epsilon^2.
 \end{align*}

Then, returning to \eqref{bd:pkuv}, we obtain
\begin{align*}
     \|P_k(||\nabla|^\beta u|^2|\nabla|^\beta u)\|_{L^p}
     &\lesssim \sum_{a\geq k-4}\|P_a||\nabla|^\beta u|^2\|_{L^p}\epsilon+\sum_{a\geq k-4}\|P_a|\nabla|^\beta u\|_{L^p}\epsilon^2 \notag\\
     &\lesssim \sum_{a\geq k-4}\sum_{b\geq a-4}\|P_b|\nabla|^\beta u\|_{L^p}\epsilon^2.\notag
 \end{align*}
 
This implies that
\begin{align*}
    2^{({{q}}-\beta)k}\|P_{k}(||\nabla|^\beta u|^2|\nabla|^\beta u)\|_{L^2}&\lesssim 2^{({{q}}-\beta)k}\sum_{a\geq k-4}\sum_{b\geq a-4}2^{-b({{q}}-\beta)}\|P_b  u\|_{H^{{{q}}}}\epsilon^2\\
    &\lesssim \sum_{b\geq k-8}\sum_{b+4\geq a\geq k-4}2^{({{q}}-\beta)(k-b)}\|P_b  u\|_{H^{{{q}}}}\epsilon^2\\
    &\lesssim \sum_{b\geq k-8}2^{({{q}}-\beta-\epsilon'')(k-b)}\|P_b  u\|_{H^{{{q}}}}\epsilon^2\lesssim A_k\epsilon^2,\\
     \|P_{k}(||\nabla|^\beta u|^2|\nabla|^\beta u)\|_{L^\infty}&\lesssim \sum_{a\geq k-4}\sum_{b\geq a-4}2^{b\beta}\|P_b  u\|_{L^\infty}\epsilon^2\\
    &\lesssim \sum_{b\geq k-8}(b+8-k)2^{b\beta}\|P_b  u\|_{L^\infty}\epsilon^2\lesssim \epsilon^3.
\end{align*}

By \eqref{bd:thm9.1} and the definition of $\mathscr{B}(u)$, we have for $k_1\geq0$
\begin{align*}
2^{({{q}}-2\beta)k_1} \|P_{k_1}  \mathscr{B}(u)\|_{L^2}
&\lesssim 
\lambda^2 2^{({{q}}-\beta)k_1}
\|P_{k_1}(\big||\nabla|^\beta u\big|^2\,|\nabla|^\beta u)\|_{L^2}\lesssim 
\lambda^2 A_{k_1}\epsilon^2,
\end{align*}
and
  \begin{align*}
   \|P_{k_1}  \mathscr{B}(u)\|_{L^\infty}\lesssim \lambda^22^{\beta \max\{ k_1,0\}}  \epsilon^3.
 \end{align*}
This proves \eqref{bd:pkbu} for $k_1\geq 0$.

For $k_1<0$, we argue similarly. For $p\in\{2,\infty\}$, we have
\begin{align*}
  \|  P_{\leq 0}(uv)\|_{L^p}\lesssim \sum_{a,b\geq0,|a-b|\leq 8}\|P_0^*(P_a^*uP_b^*v)\|_{L^p}\lesssim \sum_{a \geq0}\|P_a^*u\|_{L^p}\|v\|_{\widetilde{W}^0}+\sum_{a \geq0}\|P_a^*v\|_{L^p}\|u\|_{\widetilde{W}^0},
\end{align*}
where we define $P_0^*=P_{\leq 0},P_a^*=P_a$ for $a\geq1$. Since no frequency growth occurs in the low-frequency regime, 
the same argument as above applies and yields \eqref{bd:pkbu} for $k_1<0$.

 Then, with \eqref{bd:pkbu} in hand, for $\epsilon'>0$ sufficiently small, it follows from Lemma~\ref{lem:holder} and \eqref{bd:bpsi/phi2} that
 \begin{align*}
    & 2^{\epsilon'(\max\{\tilde k_3,0\}-\min\{\tilde k_4,0\})}|  Q^{(2)}  (P_{k_1}\mathscr{B}(u),P_{k_2}\overline u,P_{k_3}u,P_{k_4}\overline u)(s)|\\
     &\quad\lesssim \lambda^2 A_{\tilde k_1}(s)A_{\tilde k_2}(s)\epsilon^4(s)2^{\beta \tilde{k}_3-2\max\{\tilde{k}_3,0\}{+\epsilon'\max\{\tilde k_3,0\}}-\epsilon' \min\{\tilde k_3,0\}}\\
     &\qquad\times 2^{\beta \tilde{k}_4-2\max\{\tilde{k}_4,0\}-\epsilon' \min\{\tilde k_4,0\}}\cdot   2^{(4\beta-2)\max\{\tilde{k}_1,0\}}.
 \end{align*}
    By using the fact that $\beta\in(0,\frac12)$, we obtain the third term on the right side of \eqref{bd:9.2mid}.

 \subsubsection{Case 2: $\sigma\in(1,2)$.}
 In this case, the estimate is similar to Case 1. It suffices to establish an analogue of \eqref{bd:bpsi/phi2} 
for the regime $\sigma \in (1,2)$. 
However, the analysis is slightly more involved and requires a case-by-case discussion.

We start from the resonance condition
$\xi_1 - \xi_2 + \xi_3 - \xi_4 = 0.$
For convenience, we perform the change of variables 
$\xi_2 \mapsto -\xi_2$ and $\xi_4 \mapsto -\xi_4$. 
Under this transformation, the expressions in \eqref{def:psiphi} remain unchanged, 
and the resonance condition becomes
\begin{align*}
\xi_1 + \xi_2 + \xi_3 + \xi_4 = 0.
\end{align*}
We note that if $\xi_i = 0$ for some $i \in \{1,2,3,4\}$, 
then $Q^{(\sigma)}$ vanishes due to the factor 
$B_{\xi_1,\xi_2,\xi_3,\xi_4} = |\xi_1|^\beta|\xi_2|^\beta|\xi_3|^\beta|\xi_4|^\beta$. 
Hence, in the sequel we assume $\xi_i \neq 0$ for all $i$.

Let $\iota_i = \mathrm{sgn}(\xi_i)$ and denote by 
$\iota = (\iota_1,\iota_2,\iota_3,\iota_4)$ 
the associated sign pattern. 
It is clear that there is no solution when 
$\iota = (+,+,+,+)$ or $\iota = (-,-,-,-)$, 
since in that case the identity 
$\xi_1+\xi_2+\xi_3+\xi_4=0$ cannot hold.

Taking symmetry into account, it remains to consider the following three cases:

\textbf{Case 2.1: $\sigma\in(1,2)$ and
$\iota=(-,-,+,+)$.}
We replace $\xi_1$ and $\xi_2$ by $-\xi_1$ and $-\xi_2$,
respectively, so that all frequencies are nonnegative.
Under this convention, the resonance condition becomes
\[
    \xi_1+\xi_2=\xi_3+\xi_4.
\]
By simultaneously exchanging $(\xi_1,\xi_4)$ and
$(\xi_3,\xi_2)$ if necessary, we may assume that $ \xi_1-\xi_4=\xi_3-\xi_2>0.$
Then
\begin{align}
|\Omega^{(\sigma)}|
&=\xi_1^\sigma-\xi_4^\sigma
  +\xi_3^\sigma-\xi_2^\sigma
\notag\\
&=(\xi_1-\xi_4)
  \int_0^1
  \sigma\bigl(\theta\xi_1+(1-\theta)\xi_4\bigr)^{\sigma-1}
  \dif\theta 
 +(\xi_3-\xi_2)
  \int_0^1
  \sigma\bigl(\theta\xi_3+(1-\theta)\xi_2\bigr)^{\sigma-1}
  \dif\theta
\notag\\
&=(\xi_3-\xi_2)
\left[
  \int_0^1
  \sigma\bigl(\theta\xi_1+(1-\theta)\xi_4\bigr)^{\sigma-1}
  \dif\theta 
 +\int_0^1
  \sigma\bigl(\theta\xi_3+(1-\theta)\xi_2\bigr)^{\sigma-1}
  \dif\theta
\right].
\label{bd:omega2.2}
\end{align}

Since all frequencies are nonnegative and $\sigma-1>0$, we have
\begin{align*}
&\int_0^1
\sigma\bigl(\theta\xi_1+(1-\theta)\xi_4\bigr)^{\sigma-1}
\dif\theta
\geq
\int_0^1\sigma(\theta\xi_1)^{\sigma-1}\dif\theta
=\xi_1^{\sigma-1},
\\
&\int_0^1
\sigma\bigl(\theta\xi_1+(1-\theta)\xi_4\bigr)^{\sigma-1}
\dif\theta
\geq
\int_0^1\sigma\bigl((1-\theta)\xi_4\bigr)^{\sigma-1}
\dif\theta
=\xi_4^{\sigma-1}.
\end{align*}
Consequently,
\begin{align*}
&\int_0^1
\sigma\bigl(\theta\xi_1+(1-\theta)\xi_4\bigr)^{\sigma-1}
\dif\theta
\geq
\frac12\bigl(
\xi_1^{\sigma-1}+\xi_4^{\sigma-1}
\bigr).
\end{align*}
Applying the same argument to the second integral in
\eqref{bd:omega2.2}, we obtain
\begin{align}
|\Omega^{(\sigma)}|
&\gtrsim_\sigma
(\xi_1^{\sigma-1}
 +\xi_2^{\sigma-1}
 +\xi_3^{\sigma-1}
 +\xi_4^{\sigma-1})
(\xi_3-\xi_2)
\notag\\
&\gtrsim_\sigma
(\xi_1+\xi_2+\xi_3+\xi_4)^{\sigma-1}
(\xi_3-\xi_2).
\label{eq:omega-lower-case21}
\end{align}
In the last inequality, we used $0<\sigma-1<1$.

After the above changes of variables, the representation
\eqref{bd:phi=} gives
\begin{align*}
\left|
\varphi_{\underline{k}}(\xi)\Psi(\xi)
\right|
\lesssim
\varphi_{\underline{k}}(\xi)
|(\xi_1+\xi_2)(\xi_2-\xi_3)|
2^{(2{{q}}-2)\tilde{k}_1}.\notag
\end{align*}  Thus,
\begin{align*}
\left|
\varphi_{\underline{k}}(\xi)
\frac{\Psi(\xi)}{\Omega^{(\sigma)}(\xi)}
\right|
&\lesssim
\varphi_{\underline{k}}(\xi)
\frac{
(\xi_1+\xi_2)
2^{(2{{q}}-2)\tilde{k}_1}
}{
(\xi_1+\xi_2+\xi_3+\xi_4)^{\sigma-1}
} \lesssim
2^{(2{{q}}-\sigma)\tilde{k}_1}.\notag
\end{align*}
Here we used the fact that, on the support of
$\varphi_{\underline{k}}$, $ \xi_1+\xi_2+\xi_3+\xi_4
    \sim 2^{\tilde{k}_1}.$ 

We next establish the corresponding derivative bounds. Define
\begin{align*}
G(\xi)
:=
\frac{\Omega^{(\sigma)}(\xi)}{\xi_3-\xi_2}.
\end{align*}
By \eqref{bd:omega2.2},
\begin{align*}
G(\xi)
&=
\int_0^1
\sigma\bigl(\theta\xi_1+(1-\theta)\xi_4\bigr)^{\sigma-1}
\dif\theta +
\int_0^1
\sigma\bigl(\theta\xi_3+(1-\theta)\xi_2\bigr)^{\sigma-1}
\dif\theta.
\end{align*} We know that
\begin{align*}
\left|
\xi_1\partial_{\xi_1}
\bigl[
(\theta\xi_1+(1-\theta)\xi_4)^{\sigma-1}
\bigr]
\right|
&\lesssim
\theta\xi_1
(\theta\xi_1+(1-\theta)\xi_4)^{\sigma-2}
\notag\\
&\lesssim
(\theta\xi_1+(1-\theta)\xi_4)^{\sigma-1}.\notag
\end{align*}
The same estimate holds for derivatives with respect to the other
frequency variables. Repeated application of the chain rule therefore
gives, for every relevant multi-index $\alpha$,
\begin{align}
\left|
\xi^\alpha\partial_\xi^\alpha G(\xi)
\right|
\lesssim_\alpha G(\xi).
\label{eq:G-symbol-case21}
\end{align}

 Then, we introduce the Fa\`a di Bruno's formula:
\begin{align*}
    \frac{d^n}{dx^n} f(g(x)) = \sum C_{m_1,m_2,...,m_n,n} f^{(m_1 + \cdots + m_n)}(g(x)) \cdot \prod_{j=1}^{n} \bigl(g^{(j)}(x)\bigr)^{m_j},
\end{align*}
where the sum is over all $n$-tuples of nonnegative integers 
$(m_{1},\ldots ,m_{n})$ satisfying the constraint
$1\cdot m_{1}+2\cdot m_{2}+3\cdot m_{3}+\cdots +n\cdot m_{n}=n.$ 
Using \eqref{eq:G-symbol-case21}, we conclude that
\begin{align}
\left|
\xi^\alpha\partial_\xi^\alpha
\left[
G(\xi)^{-1}
\right]
\right|
\lesssim_\alpha
G(\xi)^{-1}.
\label{eq:G-inverse-symbol-case21}
\end{align}

On the other hand, using \eqref{bd:phi=} after the above changes of
variables and applying the chain rule as in \eqref{bd:example}, we
obtain
\begin{align}
\left|
\xi^\alpha\partial_\xi^\alpha
\left[
\varphi_{\underline{k}}(\xi)
\frac{\Psi(\xi)}{\xi_3-\xi_2}
\right]
\right|
\lesssim
2^{(2{{q}}-1)\tilde{k}_1}.
\label{eq:psi-divided-symbol-case21}
\end{align}
Indeed, after cancellation of $\xi_3-\xi_2$, the remaining expression
contains one factor of size at most $2^{\tilde{k}_1}$ and one factor
of $(m_{{{q}}}^2)''$, whose size is bounded by
$2^{(2{{q}}-2)\tilde{k}_1}$.

Combining \eqref{eq:omega-lower-case21},
\eqref{eq:G-inverse-symbol-case21}, and
\eqref{eq:psi-divided-symbol-case21}, and also using the standard
derivative bounds for the Littlewood--Paley cutoffs, we conclude that
\begin{align*}
\left|
\xi^\alpha\partial_\xi^\alpha
\left[
\varphi_{\underline{k}}(\xi)
\frac{\Psi(\xi)}{\Omega^{(\sigma)}(\xi)}
\right]
\right|
&=
\left|
\xi^\alpha\partial_\xi^\alpha
\left[
\varphi_{\underline{k}}(\xi)
\frac{\Psi(\xi)}{\xi_3-\xi_2}
\left(
\frac{\Omega^{(\sigma)}(\xi)}
     {\xi_3-\xi_2}
\right)^{-1}
\right]
\right| \lesssim
2^{(2{{q}}-\sigma)\tilde{k}_1}.\notag
\end{align*}

\textbf{Case 2.2: $\sigma\in(1,2)$ and
$\iota=(-,+,+,+)$.}
We replace $\xi_1$ by $-\xi_1$, so that all frequencies are
nonnegative. Under this convention, the resonance condition becomes
\[
    \xi_1=\xi_2+\xi_3+\xi_4.
\]
We use the elementary inequality
\[
    a^\sigma-b^\sigma
    \geq a^{\sigma-1}(a-b),
    \qquad a>b>0.
\] 

By symmetry between $\xi_2$ and $\xi_4$, we may assume that
$\xi_2\geq\xi_4$. We then have
\begin{align*}
\Omega^{(\sigma)}
&=(\xi_2+\xi_3+\xi_4)^\sigma
  -\xi_2^\sigma+\xi_3^\sigma-\xi_4^\sigma
\\
&\geq
(\xi_2+\xi_3+\xi_4)^{\sigma-1}
(\xi_3+\xi_4)-\xi_4^\sigma
\\
&=
\Big(
(\xi_2+\xi_3+\xi_4)^{\sigma-1}
-\xi_4^{\sigma-1}
\Big)\xi_4
+
(\xi_2+\xi_3+\xi_4)^{\sigma-1}\xi_3.
\end{align*}
Since $\xi_2\geq\xi_4$, we have  
$
     \xi_2+\xi_3+\xi_4\geq2\xi_4,
$
and hence
$
    \xi_4^{\sigma-1}
    \leq
    2^{1-\sigma}\xi_1^{\sigma-1}.
$
Therefore, setting
$
    \epsilon_\sigma:=1-2^{1-\sigma}>0,
$
we obtain
\begin{align}
\Omega^{(\sigma)}
&\geq
\epsilon_\sigma
(\xi_2+\xi_3+\xi_4)^{\sigma-1}\xi_4
+
(\xi_2+\xi_3+\xi_4)^{\sigma-1}\xi_3 \geq
\epsilon_\sigma
\xi_1^{\sigma-1}(\xi_3+\xi_4).
\label{eq:omega-lower-case22}
\end{align}

After the above changes of variables, the representation
\eqref{bd:phi=} yields
\begin{align*}
\left|
\varphi_{\underline k}(\xi)\Psi(\xi)
\right|
\lesssim
(\xi_2+\xi_3)(\xi_3+\xi_4)
2^{(2{{q}}-2)\tilde k_1}.\notag
\end{align*}  Moreover,
\begin{align*}
\left|
\varphi_{\underline{k}}(\xi)
\frac{\Psi(\xi)}{\Omega^{(\sigma)}(\xi)}
\right|
&\lesssim
\frac{
(\xi_2+\xi_3)(\xi_3+\xi_4)
2^{(2{{q}}-2)\tilde{k}_1}
}{
\xi_1^{\sigma-1}(\xi_3+\xi_4)
}  \lesssim
2^{(2{{q}}-\sigma)\tilde{k}_1}.\notag
\end{align*} 

We next establish the corresponding derivative bounds. We first claim
that, for $i=2,3,4$,
\begin{align*}
    \left|
    \xi_i\partial_{\xi_i}\Omega^{(\sigma)}
    \right|
    \lesssim
    \Omega^{(\sigma)}.\notag
\end{align*}
Since the roles of $\xi_2$ and $\xi_4$ are symmetric, it suffices to
consider $i=2,3$. We compute
\begin{align*}
\frac1\sigma
\xi_2\partial_{\xi_2}\Omega^{(\sigma)}
&=
\xi_2(\xi_2+\xi_3+\xi_4)^{\sigma-1}
-\xi_2^\sigma
\\
&\leq
(\xi_2+\xi_3+\xi_4)^\sigma
-\xi_2^\sigma+\xi_3^\sigma-\xi_4^\sigma,
\\
\frac1\sigma
\xi_3\partial_{\xi_3}\Omega^{(\sigma)}
&=
\xi_3(\xi_2+\xi_3+\xi_4)^{\sigma-1}
+\xi_3^\sigma
\\
&\leq
(\xi_2+\xi_3+\xi_4)^\sigma
-\xi_2^\sigma+\xi_3^\sigma-\xi_4^\sigma.
\end{align*}
For the second inequality, we used
\begin{align*}
    \xi_2^\sigma+\xi_4^\sigma
    \leq
    (\xi_2+\xi_4)
    (\xi_2+\xi_3+\xi_4)^{\sigma-1}.
\end{align*} 
The estimate for $i=4$ follows by symmetry.

For higher-order derivatives, we use the mean-value representation
\begin{align*}
\frac1\sigma
\partial_{\xi_2}\Omega^{(\sigma)}
&=
(\xi_2+\xi_3+\xi_4)^{\sigma-1}
-\xi_2^{\sigma-1}
\notag\\
&=
(\xi_3+\xi_4)
\int_0^1
(\sigma-1)
\bigl(
\xi_2+\theta(\xi_3+\xi_4)
\bigr)^{\sigma-2}
\,\dif\theta.\notag
\end{align*}
Since each component in the expression are of power law,  by the application of the chain rule gives, for every relevant
multi-index $\alpha$,
\begin{align}
\left|
\xi_2\xi^\alpha
\partial_\xi^\alpha
\left[
\partial_{\xi_2}\Omega^{(\sigma)}
\right]
\right|
\lesssim_\alpha
\left|
\xi_2\partial_{\xi_2}\Omega^{(\sigma)}
\right|
\lesssim
\Omega^{(\sigma)}.
\label{eq:higher-derivative-xi2-case22}
\end{align}
The same conclusion holds for $\xi_4$ by symmetry.

For the $\xi_3$ derivative, we use
\[
    \frac1\sigma
    \partial_{\xi_3}\Omega^{(\sigma)}
    =
    (\xi_2+\xi_3+\xi_4)^{\sigma-1}
    +\xi_3^{\sigma-1}.
\]
Both terms satisfy the standard homogeneous symbol bounds. Thus,
for every relevant multi-index $\alpha$,
\begin{align}
\left|
\xi_3\xi^\alpha
\partial_\xi^\alpha
\left[
\partial_{\xi_3}\Omega^{(\sigma)}
\right]
\right|
\lesssim_\alpha
\left|
\xi_3\partial_{\xi_3}\Omega^{(\sigma)}
\right|
\lesssim
\Omega^{(\sigma)}.
\label{eq:higher-derivative-xi3-case22}
\end{align}
It follows from
\eqref{eq:higher-derivative-xi2-case22}--%
\eqref{eq:higher-derivative-xi3-case22} that
\begin{align*}
\left|
\xi^\alpha\partial_\xi^\alpha
\Omega^{(\sigma)}
\right|
\lesssim_\alpha
\Omega^{(\sigma)}
\end{align*}
for every nonzero relevant multi-index $\alpha$. 

On the other hand, the representation \eqref{bd:phi=} and the chain
rule imply
\begin{align}
\left|
\xi^\alpha\partial_\xi^\alpha
\left[
\varphi_{\underline k}(\xi)\Psi(\xi)
\right]
\right|
\lesssim_\alpha
(\xi_2+\xi_3)(\xi_3+\xi_4)
2^{(2{{q}}-2)\tilde k_1}.
\label{eq:psi-derivative-case22}
\end{align} 

Combining \eqref{eq:omega-lower-case22}  and
\eqref{eq:psi-derivative-case22}, and applying Leibniz's rule, we
conclude that
\begin{align*}
\left|
\xi^\alpha\partial_\xi^\alpha
\left[
\varphi_{\underline{k}}(\xi)
\frac{\Psi(\xi)}{\Omega^{(\sigma)}(\xi)}
\right]
\right|
&\lesssim
\frac{
(\xi_2+\xi_3)(\xi_3+\xi_4)
2^{(2{{q}}-2)\tilde k_1}
}{
\xi_1^{\sigma-1}(\xi_3+\xi_4)
} \lesssim
2^{(2{{q}}-\sigma)\tilde k_1}.\notag
\end{align*}
  
\textbf{Case 2.3: $\sigma\in(1,2)$ and
$\iota=(-,+,-,+)$.}
We replace $\xi_1$ and $\xi_3$ by $-\xi_1$ and $-\xi_3$,
respectively, so that all frequencies are nonnegative.
Under this convention, the resonance condition becomes
\[
    \xi_1+\xi_3=\xi_2+\xi_4.
\]
Using the mean value theorem twice, we write
\begin{align}
\Omega^{(\sigma)}
&=\xi_1^\sigma-\xi_2^\sigma+\xi_3^\sigma-\xi_4^\sigma
\notag\\
&=(\xi_1-\xi_2)
  \int_0^1
  \sigma\bigl(\xi_2+\theta(\xi_1-\xi_2)\bigr)^{\sigma-1}
  \dif\theta 
-(\xi_4-\xi_3)
  \int_0^1
  \sigma\bigl(\xi_3+\theta(\xi_4-\xi_3)\bigr)^{\sigma-1}
  \dif\theta
\notag\\
&=(\xi_1-\xi_2)\sigma
  \int_0^1
  \Big[
      \bigl(\xi_2+\theta(\xi_1-\xi_2)\bigr)^{\sigma-1}
 -
      \bigl(\xi_3+\theta(\xi_1-\xi_2)\bigr)^{\sigma-1}
  \Big]
  \dif\theta
\notag\\
&=(\xi_1-\xi_2)(\xi_2-\xi_3)\sigma(\sigma-1)
\notag\\
&\quad\times
\int_0^1\int_0^1
\Big(
    (1-\theta)\eta\xi_2
    +(1-\theta)(1-\eta)\xi_3
    +\eta\theta\xi_1
    +\theta(1-\eta)\xi_4
\Big)^{\sigma-2}
\dif\eta\dif\theta.
\label{bd:omega2.3}
\end{align}
Here we used $ \xi_1-\xi_2=\xi_4-\xi_3.$ 
For brevity, set
\begin{align*}
X_{\eta,\theta}(\xi)
:=
(1-\theta)\eta\xi_2
+(1-\theta)(1-\eta)\xi_3
+\eta\theta\xi_1
+\theta(1-\eta)\xi_4.
\end{align*}
The coefficients in this expression are nonnegative and sum to one.
Hence, 
\[
    X_{\eta,\theta}(\xi)^{\sigma-2}
    \gtrsim
    2^{(\sigma-2)\widetilde k_1}.
\]
Therefore,
\begin{align}
\left|\Omega^{(\sigma)}(\xi)\right|
\gtrsim_\sigma
\left|
(\xi_1-\xi_2)(\xi_2-\xi_3)
\right|
2^{(\sigma-2)\widetilde k_1}.
\label{eq:omega-lower-case23}
\end{align}

By \eqref{bd:phi=}, $\Psi$ contains the same two factors
$(\xi_1-\xi_2)(\xi_2-\xi_3)$. Consequently,  
\begin{align*}
\left|
\varphi_{\underline{k}}(\xi)
\frac{\Psi(\xi)}{\Omega^{(\sigma)}(\xi)}
\right|
&\lesssim
\frac{
\left|
(\xi_1-\xi_2)(\xi_2-\xi_3)
\right|
2^{(2{{q}}-2)\widetilde{k}_1}
}{
\left|
(\xi_1-\xi_2)(\xi_2-\xi_3)
\right|
2^{(\sigma-2)\widetilde{k}_1}
} \lesssim
2^{(2{{q}}-\sigma)\widetilde{k}_1}.\notag
\end{align*}

We next establish the corresponding derivative estimates. Define
\begin{align*}
G_\sigma(\xi)
:=
\frac{\Omega^{(\sigma)}(\xi)}
     {(\xi_1-\xi_2)(\xi_2-\xi_3)}.
\end{align*}
By \eqref{bd:omega2.3},
\begin{align*}
G_\sigma(\xi)
=
\sigma(\sigma-1)
\int_0^1\int_0^1
X_{\eta,\theta}(\xi)^{\sigma-2}
\dif\eta\dif\theta>0.
\end{align*}
For each $i\in\{1,2,3,4\}$, we have
\[
    |\xi_i\partial_{\xi_i}X_{\eta,\theta}(\xi)|
    \leq X_{\eta,\theta}(\xi).
\] 
Repeated application of the chain rule therefore yields, for every
relevant multi-index $\alpha$,
\begin{align*}
\left|
\xi^\alpha\partial_\xi^\alpha
G_\sigma(\xi)
\right|
\lesssim_\alpha
G_\sigma(\xi).\notag
\end{align*}
Since $G_\sigma>0$, the multivariable Fa\`a di Bruno formula gives
\begin{align}
\left|
\xi^\alpha\partial_\xi^\alpha
\left[
G_\sigma(\xi)^{-1}
\right]
\right|
\lesssim_\alpha
G_\sigma(\xi)^{-1}.
\label{eq:G-inverse-case23}
\end{align}
Furthermore, \eqref{eq:omega-lower-case23} implies
\begin{align}
G_\sigma(\xi)^{-1}
\lesssim
2^{(2-\sigma)\widetilde k_1}.
\label{eq:G-inverse-size-case23}
\end{align}

On the other hand, using \eqref{bd:phi=} and the chain rule as in
\eqref{bd:example}, we obtain the absolute dyadic bound
\begin{align}
\left|
\xi^\alpha\partial_\xi^\alpha
\left[
\varphi_{\underline k}(\xi)
\frac{\Psi(\xi)}
     {(\xi_1-\xi_2)(\xi_2-\xi_3)}
\right]
\right|
\lesssim_\alpha
2^{(2{{q}}-2)\widetilde k_1}.
\label{eq:psi-divided-symbol-case23}
\end{align}
Notice that we use an absolute dyadic bound here, rather than an
estimate relative to the size of $\Psi$, since the latter may vanish.

Combining \eqref{eq:G-inverse-case23},
\eqref{eq:G-inverse-size-case23}, and
\eqref{eq:psi-divided-symbol-case23}, and applying Leibniz's rule, we
obtain
\begin{align*}
\left|
\xi^\alpha\partial_\xi^\alpha
\left[
\varphi_{\underline k}(\xi)
\frac{\Psi(\xi)}{\Omega^{(\sigma)}(\xi)}
\right]
\right|
&=
\left|
\xi^\alpha\partial_\xi^\alpha
\left[
\varphi_{\underline k}(\xi)
\frac{\Psi(\xi)}
     {(\xi_1-\xi_2)(\xi_2-\xi_3)}
G_\sigma(\xi)^{-1}
\right]
\right| \lesssim
2^{(2{{q}}-\sigma)\widetilde k_1}.
\end{align*}

Combining the estimates obtained in Cases 2.1--2.3 and applying
Lemma~\ref{lem:sinfty}, we conclude that
\begin{align*}
\left\|
B_{\xi_1,\xi_2,\xi_3,\xi_4}
\frac{\Psi(\xi)}{\Omega^{(\sigma)}(\xi)}
\varphi_{\underline k}(\xi)
\right\|_{\widetilde S^\infty}
&\lesssim
2^{(2{{q}}-\sigma)\widetilde{k}_1}
2^{\beta(
\widetilde{k}_1+\widetilde{k}_2
+\widetilde{k}_3+\widetilde{k}_4)}
\notag\\
&\lesssim
2^{(2{{q}}-\sigma+2\beta)\widetilde{k}_1}
2^{\beta(\widetilde{k}_3+\widetilde{k}_4)}.
\notag
\end{align*}

We are now in a position to estimate the terms in
\eqref{bd:exp:Q}. For the boundary terms, Lemma~\ref{lem:holder}
gives
\begin{align*}
&2^{\epsilon'
(\max\{\widetilde k_3,0\}
-\min\{\widetilde k_4,0\})}
\left|
Q^{(\sigma)}
(P_{k_1}u,P_{k_2}\overline u,
 P_{k_3}u,P_{k_4}\overline u)(s)
\right|
\\
&\quad\lesssim
A_{\widetilde k_1}(s)
A_{\widetilde k_2}(s)
\epsilon^2(s)
2^{\beta\widetilde{k}_3
-(2-\epsilon')\max\{\widetilde{k}_3,0\}}
\\
&\qquad\qquad\times
2^{\beta\widetilde{k}_4
-2\max\{\widetilde{k}_4,0\}
-\epsilon'\min\{\widetilde k_4,0\}}
2^{(4\beta-\sigma)\max\{\widetilde{k}_1,0\}}.
\end{align*}
Here the factor $  2^{(4\beta-\sigma)\max\{\widetilde k_1,0\}}$ 
is a harmless weakening of the sharper boundary estimate, chosen so
that the boundary and space-time contributions have the same dyadic
form.

For the space-time contributions, using \eqref{bd:pkbu},
Lemma~\ref{lem:holder}, and the above multiplier bound, we similarly
obtain
\begin{align*}
&2^{\epsilon'
(\max\{\widetilde k_3,0\}
-\min\{\widetilde k_4,0\})}
\left|
Q^{(\sigma)}
(P_{k_1}\mathscr B(u),P_{k_2}\overline u,
 P_{k_3}u,P_{k_4}\overline u)(s)
\right|
\\
&\quad\lesssim
\lambda^2
A_{\widetilde k_1}(s)
A_{\widetilde k_2}(s)
\epsilon^4(s)
2^{\beta\widetilde{k}_3
-(2-\epsilon')\max\{\widetilde{k}_3,0\}}
\\
&\qquad\qquad\times
2^{\beta\widetilde{k}_4
-2\max\{\widetilde{k}_4,0\}
-\epsilon'\min\{\widetilde k_4,0\}}
2^{(4\beta-\sigma)\max\{\widetilde{k}_1,0\}}.
\end{align*} 

Since $4\beta<\sigma$ and $0<\epsilon'<\beta$, the dyadic factors are
bounded by $1$. Therefore,  integration over $[t_1,t_2]$ yields $\lambda^2
    \int_{t_1}^{t_2}
    A_l(s)^2\epsilon^4(s)\dif s,$ 
which is the third term on the right-hand side of
\eqref{bd:9.2mid}. 

\subsubsection{Case 3: $\sigma\in(0,1)$.}

In this regime, the analysis becomes more delicate due to the presence
of nontrivial resonant interactions. As in Case 2, we first replace
$\xi_2$ and $\xi_4$ by $-\xi_2$ and $-\xi_4$, respectively. The
momentum constraint then takes the form
\[
    \xi_1+\xi_2+\xi_3+\xi_4=0.
\]

We decompose
\begin{align*}
&Q(P_{k_1}u,P_{k_2}\overline u,
   P_{k_3}u,P_{k_4}\overline u)
\\
&\qquad
=
Q_1(P_{k_1}u,P_{k_2}\overline u,
    P_{k_3}u,P_{k_4}\overline u)
+
Q_2(P_{k_1}u,P_{k_2}\overline u,
    P_{k_3}u,P_{k_4}\overline u),
\end{align*}
where
\begin{align*}
&Q_1(P_{k_1}u,P_{k_2}\overline u,
     P_{k_3}u,P_{k_4}\overline u)
\\
&\qquad
:=
 \frac{1}{2L^3}
\sum_{\mathcal Q_1}
\widehat u_{\xi_1}
\overline{\widehat u}_{-\xi_2}
\widehat u_{\xi_3}
\overline{\widehat u}_{-\xi_4}
B_{\xi_1,\xi_2,\xi_3,\xi_4}
\Psi(\xi_1^*,\xi_2,\xi_3^*,\xi_4)
\varphi_{\underline k}(\xi),
\\
&Q_2:=Q-Q_1.
\end{align*}
Here $\sum_{\mathcal Q_1}$ denotes summation over
$\xi\in(L^{-1}\mathbb Z)^4$ satisfying
\[
    \xi_1+\xi_2+\xi_3+\xi_4=0
\]
and whose sign configuration belongs to
\begin{align*}
\{
&(-,+,+,+),(+,-,+,+),(+,+,-,+),(+,+,+,-),
\\
&(+,-,-,-),(-,+,-,-),(-,-,+,-),(-,-,-,+)
\}.
\end{align*}

We next define the modified frequencies. {For example, consider} $\iota=(-,+,+,+).$ 
We set $ \xi_1^*
    :=
    -(\xi_2+\xi_3^*+\xi_4),$ 
where $\xi_3^*$ is the unique positive solution of
\begin{align}
    (x+\xi_2+\xi_4)^\sigma
    +x^\sigma
    -\xi_2^\sigma
    -\xi_4^\sigma
    =0.
\label{eq:resonant-root}
\end{align} 
Since $0<\sigma<1$, the strict concavity of $x\mapsto x^\sigma$
implies  \eqref{eq:resonant-root} admits a unique solution
satisfying
\[
    0<\xi_3^*<\min\{\xi_2,\xi_4\}.
\]

We aim to prove that
\begin{align}
&2^{\epsilon'
(\max\{\tilde k_3,0\}-\min\{\tilde k_4,0\})}
\left|
\int_{t_1}^{t_2}
Q_1(P_{k_1}u,P_{k_2}\overline u,
    P_{k_3}u,P_{k_4}\overline u)(s)
\,\dif s
\right|
\notag\\
&\qquad
\lesssim
\int_{t_1}^{t_2}
{\widetilde A_l^2(s)\epsilon^2(s)}
\,\dif s,
\label{bd:q1}
\end{align}
and
\begin{align}
&2^{\epsilon'
(\max\{\tilde k_3,0\}-\min\{\tilde k_4,0\})}
\left|
\int_{t_1}^{t_2}
Q_2(P_{k_1}u,P_{k_2}\overline u,
    P_{k_3}u,P_{k_4}\overline u)(s)
\,\dif s
\right|
\notag\\
&\qquad
\lesssim
[A_l(t_1)\epsilon(t_1)]^2
+
[A_l(t_2)\epsilon(t_2)]^2
+
\lambda^2
\int_{t_1}^{t_2}
A_l^2(s)\epsilon^4(s)
\,\dif s.
\label{bd:q2}
\end{align}

The estimate for $Q_1$ is treated in Case 3.0 below. 
For $Q_2$, the argument is similar to Case 2, and it suffices to establish a symbol bound analogous to \eqref{bd:bpsi/phi2}.
Taking symmetry into account, we consider the sign configurations in Cases 3.1–3.3.

\textbf{Case 3.0: $\sigma\in(0,1)$ and the proof of
\eqref{bd:q1}.}
By symmetry, it suffices to consider the sign configuration
$\iota=(-,+,+,+)$. Using \eqref{bd:phi=} after the corresponding
changes of signs, we obtain
\begin{align*}
&\left|
\Psi(\xi_1^*,\xi_2,\xi_3^*,\xi_4)
\varphi_{\underline{k}}(\xi)
\right|
\lesssim
(\xi_2+\xi_3^*)
(\xi_3^*+\xi_4)
(\xi_2+\xi_3^*+\xi_4)^{2{{q}}-2}
\varphi_{\underline{k}}(\xi)
\\
&\qquad
\lesssim
\xi_2\xi_4
(\xi_2+\xi_4)^{2{{q}}-2}
\varphi_{\underline{k}}(\xi) 
\lesssim
2^{\widetilde{k}_1+\widetilde{k}_3
  +(2{{q}}-2)\widetilde{k}_1}.
\end{align*}
Here we used $ 0<\xi_3^*<\min\{\xi_2,\xi_4\}.$ 
Moreover,  
\[
    \max\{\xi_2,\xi_4\}\lesssim 2^{\widetilde k_1},
    \qquad
    \min\{\xi_2,\xi_4\}\lesssim 2^{\widetilde k_3}.
\]
It follows that
\begin{align*}
&\left|
B_{\xi_1,\xi_2,\xi_3,\xi_4}
\Psi(\xi_1^*,\xi_2,\xi_3^*,\xi_4)
\varphi_{\underline{k}}(\xi)
\right| 
\lesssim
2^{(2{{q}}-1+2\beta)\widetilde{k}_1}
2^{(1+\beta)\widetilde{k}_3}
2^{\beta\widetilde{k}_4}.
\end{align*}

We next establish the corresponding derivative estimates. Regard $ 
    \xi_3^*=\xi_3^*(\xi_2,\xi_4)$
as a function of $(\xi_2,\xi_4)$, and write
$
    S:=\xi_2+\xi_3^*(\xi_2,\xi_4)+\xi_4.
$
By definition,
\[
    S^\sigma+(\xi_3^*)^\sigma
    -\xi_2^\sigma-\xi_4^\sigma=0.
\]
Implicit differentiation with respect to $\xi_2$ gives
\begin{align*}
\partial_{\xi_2}\xi_3^*
&=
\frac{
\xi_2^{\sigma-1}
-
S^{\sigma-1}
}{
S^{\sigma-1}
+
(\xi_3^*)^{\sigma-1}
} =
\frac{
(1-\sigma)(\xi_3^*+\xi_4)
\displaystyle\int_0^1
\bigl(
\xi_2+\theta(\xi_3^*+\xi_4)
\bigr)^{\sigma-2}
\,\dif\theta
}{
S^{\sigma-1}
+
(\xi_3^*)^{\sigma-1}
}.
\end{align*} 
We claim that
\begin{align*}
|\xi_4\partial_{\xi_4}\xi_3^*|+|\xi_2\partial_{\xi_2}\xi_3^*|\lesssim \xi_3^*  . \notag
\end{align*} We verify the $\xi_2\partial_{\xi_2}$ case. It suffices to show
\begin{align*}
    \xi_2^{\sigma}-  \xi_2(\xi_2+\xi_3^*+\xi_4)^{\sigma-1}\leq  \xi_3^*(\xi_2+\xi_3^*+\xi_4)^{\sigma-1}+(\xi_3^*)^{\sigma}.
\end{align*}
By the definition of $\xi_3^*$, it suffices to check
\begin{align*}
    (\xi_2+\xi_3^*+\xi_4)^\sigma-\xi_4^\sigma-  \xi_2(\xi_2+\xi_3^*+\xi_4)^{\sigma-1}\leq  \xi_3^*(\xi_2+\xi_3^*+\xi_4)^{\sigma-1}.
\end{align*}
Since $\sigma\in(0,1)$,    we use  $(\xi_2+\xi_3^*+\xi_4)^\sigma=(\xi_2+\xi_3^*+\xi_4)\cdot (\xi_2+\xi_3^*+\xi_4)^{\sigma-1}$ to conclude the desired bound. The $\xi_4$ case follows similarly. 

For higher order derivatives, $\partial^\alpha \xi_3^*$ can be written as a ratio of expressions with homogeneous  power law structure.
More generally, for every multi-index
$\alpha=(\alpha_2,\alpha_4)$ needed below, repeated implicit
differentiation yields
\begin{align}
\left|
\xi_2^{\alpha_2}
\xi_4^{\alpha_4}
\partial_{\xi_2}^{\alpha_2}
\partial_{\xi_4}^{\alpha_4}
\xi_3^*
\right|
\lesssim_\alpha
\xi_3^*.
\label{bd:part2xi3}
\end{align}

We now estimate derivatives of the modified symbol. Set
\[
\Psi^*(\xi_2,\xi_4)
:=
\Psi\bigl(
-\xi_2-\xi_3^*(\xi_2,\xi_4)-\xi_4,\,
\xi_2,\,
\xi_3^*(\xi_2,\xi_4),\,
\xi_4
\bigr).
\]
By the chain rule,
\begin{align*}
\xi_2\partial_{\xi_2}\Psi^*
&=
\xi_2
\left[
-(1+\partial_{\xi_2}\xi_3^*)\partial_1
+\partial_2
+(\partial_{\xi_2}\xi_3^*)\partial_3
\right]
\Psi
\bigl(
\xi_1^*,\xi_2,\xi_3^*,\xi_4
\bigr).
\end{align*}
Using   the factorized
representation of $\Psi$ above, and the symbol estimates for
$(m_{{{q}}}^2)''$, we obtain
\begin{align*}
\left|
\xi_2\partial_{\xi_2}\Psi^*
\right|
&\lesssim
2^{\widetilde k_1}
2^{\widetilde k_3
  +(2{{q}}-2)\widetilde k_1}
\lesssim
2^{\widetilde k_3
  +(2{{q}}-1)\widetilde k_1}.
\end{align*}
Higher-order derivatives are controlled in the same way by repeated
application of the chain rule and \eqref{bd:part2xi3}. Consequently,
for every relevant multi-index $\alpha$,
\begin{align*}
\left|
\xi^\alpha\partial_\xi^\alpha
\left[
B_{\xi_1,\xi_2,\xi_3,\xi_4}
\Psi(\xi_1^*,\xi_2,\xi_3^*,\xi_4)
\varphi_{\underline{k}}(\xi)
\right]
\right|
\lesssim_\alpha
2^{(2{{q}}-1+2\beta)\widetilde{k}_1}
2^{(1+\beta)\widetilde{k}_3}
2^{\beta\widetilde{k}_4}.
\end{align*}
Here derivatives with respect to $\xi_1$ and $\xi_3$ only act on the
remaining localized factors, while the modified symbol itself depends
on $\xi_2$ and $\xi_4$.

Lemma~\ref{lem:sinfty} now gives
\begin{align}
\left\|
B_{\xi_1,\xi_2,\xi_3,\xi_4}
\Psi(\xi_1^*,\xi_2,\xi_3^*,\xi_4)
\varphi_{\underline{k}}(\xi)
\right\|_{\widetilde S^\infty}
\lesssim
2^{(2{{q}}-1+2\beta)\widetilde{k}_1}
2^{(1+\beta)\widetilde{k}_3}
2^{\beta\widetilde{k}_4}.
\label{eq:q1-sinfty}
\end{align}

We next return to
\[
    Q_1(P_{k_1}u,P_{k_2}\overline u,
        P_{k_3}u,P_{k_4}\overline u).
\]
We may insert an additional cutoff on the largest-frequency factor:
\begin{align*}
&Q_1(P_{k_1}u,P_{k_2}\overline u,
     P_{k_3}u,P_{k_4}\overline u)
\\
&\quad
=
{\frac{1}{2L^3}}
\sum_{-\xi_1=\xi_2+\xi_3+\xi_4}
\varphi_{\geq-10}(\xi_1)
\widehat u_{\xi_1}
\overline{\widehat u}_{-\xi_2}
\widehat u_{\xi_3}
\overline{\widehat u}_{-\xi_4}
B_{\xi_1,\xi_2,\xi_3,\xi_4} 
\Psi(\xi_1^*,\xi_2,\xi_3^*,\xi_4)
\varphi_{\underline{k}}(\xi).
\end{align*}
Indeed, on the support of $1-\varphi_{\geq-10}(\xi_1)$, all the
actual frequencies satisfy
\[
    |\xi_j|\leq|\xi_1|,
\]
while
\[
    0<\xi_3^*<\min\{\xi_2,\xi_4\},
    \qquad
    |\xi_1^*|
    =
    \xi_2+\xi_3^*+\xi_4
    \lesssim|\xi_1|.
\]
Thus all four arguments of $\Psi$ lie in the region where
$m_{{{q}}}$ vanishes, and hence
\[
    \Psi(\xi_1^*,\xi_2,\xi_3^*,\xi_4)=0.
\] 
 Moreover, under the constraint $-\xi_1 = \xi_2 + \xi_3 + \xi_4$, we have
\[
2^{k_1}\sim 2^{\tilde{k}_1}\sim2^{\tilde{k}_2}.
\]  Applying Lemma~\ref{lem:holder}
and \eqref{eq:q1-sinfty}, we obtain
\begin{align*}
 &|Q_1(P_{k_1}u,P_{k_2}\overline u,P_{k_3}u,P_{k_4}\overline u)(s)|\\
 &\quad\lesssim
 \tilde A_{\tilde k_1}(s)\tilde A_{\tilde k_2}(s)\epsilon^2(s)
 2^{(1+\beta)\tilde k_3-2\max\{\tilde k_3,0\}}
 2^{\beta\tilde k_4-2\max\{\tilde k_4,0\}}.
\end{align*}   This completes
the proof of \eqref{bd:q1}.

\textbf{Case 3.1: $\sigma\in(0,1)$ and
$\iota=(-,+,+,+)$.}
We replace $\xi_1$ by $-\xi_1$, so that all frequencies are
nonnegative. Under this convention, the resonance condition becomes
\[
    \xi_1=\xi_2+\xi_3+\xi_4.
\]

To estimate $Q_2$, we apply a normal-form transformation analogous to
\eqref{bd:exp:Q}. We are led to study the multiplier
\begin{align}
\varphi_{\underline k}(\xi)
\frac{\Psi(\xi)-\Psi(\xi^*)}{\Omega^{(\sigma)}(\xi)}
&=
\varphi_{\underline{k}}(\xi)
\frac{\Psi(\xi)-\Psi(\xi^*)}
     {\Omega^{(\sigma)}(\xi)-\Omega^{(\sigma)}(\xi^*)},
\label{eq:case31-difference-quotient}
\end{align}
where $ \xi^*=(\xi_1^*,\xi_2,\xi_3^*,\xi_4)$ 
and, by construction,
\[
    \Omega^{(\sigma)}(\xi^*)=0.
\]

Fix $(\xi_2,\xi_4)$ and define
\begin{align*}
\widetilde\Psi(x)
&:=
\Psi\bigl(
    -\xi_2-x-\xi_4,\,
    \xi_2,\,
    x,\,
    \xi_4
\bigr),
\\
\widetilde\Omega(x)
&:=
(\xi_2+x+\xi_4)^\sigma
-\xi_2^\sigma+x^\sigma-\xi_4^\sigma.
\end{align*}
 
By the mean value theorem,
\begin{align*}
\left|
\widetilde\Psi(\xi_3)
-\widetilde\Psi(\xi_3^*)
\right|
&=
|\xi_3-\xi_3^*|
\left|
\int_0^1
\widetilde\Psi'
\bigl(
    \theta\xi_3+(1-\theta)\xi_3^*
\bigr)
\,\dif\theta
\right|
\\
&\lesssim
|\xi_3-\xi_3^*|
\bigl(
\max\{\xi_3,\xi_3^*,\xi_2,\xi_4\}
\bigr)^{2{{q}}-1}
\\
&\lesssim
|\xi_3-\xi_3^*|
2^{(2{{q}}-1)\widetilde{k}_1}.
\end{align*}
Similarly, 
\begin{align*}
\left|
\widetilde\Omega(\xi_3)
-\widetilde\Omega(\xi_3^*)
\right|
&\gtrsim_\sigma
|\xi_3-\xi_3^*|
\bigl(
\max\{\xi_3,\xi_3^*,\xi_2,\xi_4\}
\bigr)^{\sigma-1}
\\
&\gtrsim_\sigma
|\xi_3-\xi_3^*|
2^{(\sigma-1)\widetilde{k}_1}.
\end{align*}
Here we used
$
    0<\xi_3^*<\min\{\xi_2,\xi_4\}$
and
$
    \max\{\xi_3,\xi_2,\xi_4\}
    \lesssim  \xi_1
    \sim 2^{\widetilde k_1}.
$

Consequently, the quotient in
\eqref{eq:case31-difference-quotient} admits a continuous extension
across $\xi_3=\xi_3^*$ and satisfies
\begin{align}
\left|
\varphi_{\underline{k}}(\xi)
\frac{\Psi(\xi)-\Psi(\xi^*)}
     {\Omega^{(\sigma)}(\xi)}
\right|
\lesssim
2^{(2{{q}}-\sigma)\widetilde{k}_1}.
\label{bd:case3.1}
\end{align}

We next establish the corresponding derivative bounds. Set
\begin{align*}
z_\theta
:=
\theta\xi_3+(1-\theta)\xi_3^*,
\end{align*}
and write
\begin{align*}
D(\xi)
&:=
\frac{
\widetilde\Omega(\xi_3)
-\widetilde\Omega(\xi_3^*)
}{
\xi_3-\xi_3^*
}=
\sigma
\int_0^1
\left[
    (\xi_2+z_\theta+\xi_4)^{\sigma-1}
    +z_\theta^{\sigma-1}
\right]
\,\dif\theta.
\end{align*}

By \eqref{bd:part2xi3}, the dependence of $\xi_3^*$ on
$(\xi_2,\xi_4)$ satisfies the required symbol-type estimates. Hence,
for every relevant multi-index $\alpha$,
\begin{align*}
\left|
\xi^\alpha\partial_\xi^\alpha
D(\xi)
\right|
\lesssim_\alpha
D(\xi).
\end{align*}
For example, the integrands that arise after differentiation are
controlled by
\begin{align*}
&\left|
\xi^\alpha\partial_\xi^\alpha
\left[
(\xi_2+z_\theta+\xi_4)^{\sigma-1}
+z_\theta^{\sigma-1}
\right]
\right| 
\lesssim_\alpha
(\xi_2+z_\theta+\xi_4)^{\sigma-1}
+z_\theta^{\sigma-1}.
\end{align*}
The multivariable Fa\`a di Bruno formula therefore gives
\begin{align}
\left|
\xi^\alpha\partial_\xi^\alpha
\left[D(\xi)^{-1}\right]
\right|
\lesssim_\alpha
D(\xi)^{-1}.
\label{eq:D-inverse-case31}
\end{align}

For the numerator, define
\begin{align*}
N(\xi)
&:=
\frac{
\widetilde\Psi(\xi_3)
-\widetilde\Psi(\xi_3^*)
}{
\xi_3-\xi_3^*
} =
\int_0^1
\widetilde\Psi'(z_\theta)
\,\dif\theta.
\end{align*}
Using the symbol estimates for $m_{{{q}}}^2$, the chain rule, and
\eqref{bd:part2xi3}, we obtain
\begin{align}
\left|
\xi^\alpha\partial_\xi^\alpha
\left[
\varphi_{\underline k}(\xi)N(\xi)
\right]
\right|
\lesssim_\alpha
2^{(2{{q}}-1)\widetilde k_1}.
\label{eq:N-symbol-case31}
\end{align}
Combining  
\eqref{eq:D-inverse-case31}, and
\eqref{eq:N-symbol-case31}, we conclude that
\begin{align*}
\left|
\xi^\alpha\partial_\xi^\alpha
\left[
\varphi_{\underline{k}}(\xi)
\frac{\Psi(\xi)-\Psi(\xi^*)}
     {\Omega^{(\sigma)}(\xi)}
\right]
\right|
\lesssim_\alpha
2^{(2{{q}}-\sigma)\widetilde{k}_1}.
\end{align*}
Thus, the bound \eqref{bd:case3.1} remains valid under all derivatives
required in Lemma~\ref{lem:sinfty}.

\medskip

\textbf{Case 3.2: $\sigma\in(0,1)$ and
$\iota=(-,-,+,+)$.}
We replace $\xi_1$ and $\xi_2$ by $-\xi_1$ and $-\xi_2$,
respectively, so that all frequencies are nonnegative. Under this
convention, the resonance condition becomes
\[
    \xi_1+\xi_2=\xi_3+\xi_4.
\]
By symmetry, we may assume that $\xi_1-\xi_4=\xi_3-\xi_2>0.$ 
By the mean value theorem,
\begin{align}
|\Omega^{(\sigma)}|
&=
\xi_1^\sigma-\xi_4^\sigma
+\xi_3^\sigma-\xi_2^\sigma
\notag\\
&=
(\xi_1-\xi_4)
\left[
\int_0^1
\sigma
\bigl(
    \theta\xi_1+(1-\theta)\xi_4
\bigr)^{\sigma-1}
\,\dif\theta
\right.
\notag\\
&\hspace{3.6cm}\left.
+
\int_0^1
\sigma
\bigl(
    \theta\xi_3+(1-\theta)\xi_2
\bigr)^{\sigma-1}
\,\dif\theta
\right].
\label{eq:omega-case32-mvt}
\end{align}
Since $\sigma-1<0$,
we have
\begin{align*}
\int_0^1
\sigma
\bigl(
    \theta\xi_1+(1-\theta)\xi_4
\bigr)^{\sigma-1}
\,\dif\theta
\geq
\sigma(\xi_1+\xi_4)^{\sigma-1}.
\end{align*} 
\begin{align*}
\int_0^1
\sigma
\bigl(
    \theta\xi_3+(1-\theta)\xi_2
\bigr)^{\sigma-1}
\,\dif\theta
\geq
\sigma(\xi_3+\xi_2)^{\sigma-1}.
\end{align*}
Consequently,
\begin{align*}
|\Omega^{(\sigma)}|
&\gtrsim_\sigma
(\xi_1-\xi_4)\left[
    (\xi_1+\xi_4)^{\sigma-1}
    +(\xi_3+\xi_2)^{\sigma-1}
\right]
\notag\\
&\gtrsim_\sigma
(\xi_1-\xi_4)
(\xi_1+\xi_2+\xi_3+\xi_4)^{\sigma-1}.\notag
\end{align*}

On the other hand, \eqref{bd:phi=} gives
\begin{align}
\left|
\varphi_{\underline{k}}(\xi)\Psi(\xi)
\right|
\lesssim
|(\xi_1+\xi_2)(\xi_2-\xi_3)|
2^{(2{{q}}-2)\widetilde{k}_1}.
\label{eq:psi-case32}
\end{align}  It follows that 
\begin{align*}
\left|
\varphi_{\underline{k}}(\xi)
\frac{\Psi(\xi)}{\Omega^{(\sigma)}(\xi)}
\right|
&\lesssim
\frac{
(\xi_1+\xi_2)
2^{(2{{q}}-2)\widetilde{k}_1}
}{
(\xi_1+\xi_2+\xi_3+\xi_4)^{\sigma-1}
} \lesssim
2^{(2{{q}}-\sigma)\widetilde{k}_1}.\notag
\end{align*} 

For the derivative estimates, define
\begin{align*}
G(\xi)
:=
\frac{\Omega^{(\sigma)}(\xi)}{\xi_1-\xi_4}.
\end{align*}
By \eqref{eq:omega-case32-mvt}, $G(\xi)>0$. Moreover, since each term  in the expression is of power law, repeated
application of the chain rule gives, for every relevant multi-index
$\alpha$,
\begin{align*}
\left|
\xi^\alpha\partial_\xi^\alpha G(\xi)
\right|
\lesssim_\alpha
G(\xi).\notag
\end{align*} 

Since $G>0$, the multivariable Fa\`a di Bruno formula yields
\begin{align}
\left|
\xi^\alpha\partial_\xi^\alpha
\left[G(\xi)^{-1}\right]
\right|
\lesssim_\alpha
G(\xi)^{-1}.
\label{eq:G-inverse-case32}
\end{align} 

For the numerator, \eqref{eq:psi-case32} and the chain rule give the
absolute dyadic estimate
\begin{align}
\left|
\xi^\alpha\partial_\xi^\alpha
\left[
\varphi_{\underline k}(\xi)
\frac{\Psi(\xi)}{\xi_1-\xi_4}
\right]
\right|
\lesssim_\alpha
2^{(2{{q}}-1)\widetilde k_1}.
\label{eq:psi-divided-case32}
\end{align}
Combining \eqref{eq:G-inverse-case32}  and
\eqref{eq:psi-divided-case32}, we obtain
\begin{align*}
\left|
\xi^\alpha\partial_\xi^\alpha
\left[
\varphi_{\underline{k}}(\xi)
\frac{\Psi(\xi)}{\Omega^{(\sigma)}(\xi)}
\right]
\right|
\lesssim_\alpha
2^{(2{{q}}-\sigma)\widetilde{k}_1}.\notag
\end{align*} 

\textbf{Case 3.3: $\sigma\in(0,1)$ and
$\iota=(-,+,-,+)$.}
We replace $\xi_1$ and $\xi_3$ by $-\xi_1$ and $-\xi_3$,
respectively, so that all frequencies are nonnegative.
Under this convention, the resonance condition becomes
\[
    \xi_1-\xi_2=\xi_4-\xi_3.
\]
Using \eqref{bd:omega2.3}, we write
\begin{align*}
\Omega^{(\sigma)}
&=
(\xi_1-\xi_2)(\xi_2-\xi_3)\sigma(\sigma-1)
\notag\\
&\quad\times
\int_0^1\int_0^1
\Big(
    (1-\theta)\eta\xi_2
    +(1-\theta)(1-\eta)\xi_3
    +\eta\theta\xi_1
    +\theta(1-\eta)\xi_4
\Big)^{\sigma-2}
\,\dif\eta\,\dif\theta.\notag
\end{align*}
For brevity, set
\[
X_{\eta,\theta}(\xi)
:=
(1-\theta)\eta\xi_2
+(1-\theta)(1-\eta)\xi_3
+\eta\theta\xi_1
+\theta(1-\eta)\xi_4.
\]
Since the coefficients in $X_{\eta,\theta}$ are nonnegative and sum
to one, on the support of $\varphi_{\underline k}$ we have 
\[
    X_{\eta,\theta}(\xi)^{\sigma-2}
    \gtrsim
    2^{(\sigma-2)\widetilde k_1}.
\]
Consequently,
\begin{align*}
\left|\Omega^{(\sigma)}(\xi)\right|
\gtrsim_\sigma
\left|
(\xi_1-\xi_2)(\xi_2-\xi_3)
\right|
2^{(\sigma-2)\widetilde{k}_1}.\notag
\end{align*}

On the other hand, by \eqref{bd:phi=},
\begin{align*}
\left|
\varphi_{\underline{k}}(\xi)\Psi(\xi)
\right|
\lesssim
\left|
(\xi_1-\xi_2)(\xi_2-\xi_3)
\right|
2^{(2{{q}}-2)\widetilde{k}_1}.
\end{align*}  Moreover,
\begin{align*}
\left|
\varphi_{\underline{k}}(\xi)
\frac{\Psi(\xi)}{\Omega^{(\sigma)}(\xi)}
\right|
&\lesssim
\frac{
\left|
(\xi_1-\xi_2)(\xi_2-\xi_3)
\right|
2^{(2{{q}}-2)\widetilde{k}_1}
}{
\left|
(\xi_1-\xi_2)(\xi_2-\xi_3)
\right|
2^{(\sigma-2)\widetilde{k}_1}
} \lesssim
2^{(2{{q}}-\sigma)\widetilde{k}_1}.\notag
\end{align*}

We next establish the corresponding derivative estimates. Define
\begin{align*}
G_\sigma(\xi)
:=
-\frac{\Omega^{(\sigma)}(\xi)}
{(\xi_1-\xi_2)(\xi_2-\xi_3)} 
=
\sigma(1-\sigma)
\int_0^1\int_0^1
X_{\eta,\theta}(\xi)^{\sigma-2}
\,\dif\eta\,\dif\theta
>0.
\end{align*} Repeated application of the chain rule yields, for
every relevant multi-index $\alpha$,
\begin{align*}
\left|
\xi^\alpha\partial_\xi^\alpha
G_\sigma(\xi)
\right|
\lesssim_\alpha
G_\sigma(\xi).\notag
\end{align*}
Indeed, each derivative produces a factor of the form
$X_{\eta,\theta}^{\sigma-2-r}$, while the corresponding scaled
frequency factors contribute at most $X_{\eta,\theta}^r$.

Since $G_\sigma>0$, the multivariable Fa\`a di Bruno formula implies
\begin{align}
\left|
\xi^\alpha\partial_\xi^\alpha
\left[
G_\sigma(\xi)^{-1}
\right]
\right|
\lesssim_\alpha
G_\sigma(\xi)^{-1}.
\label{eq:G-inverse-case33}
\end{align} 

For the numerator, \eqref{bd:phi=} and the chain rule give the absolute
dyadic estimate
\begin{align}
\left|
\xi^\alpha\partial_\xi^\alpha
\left[
\varphi_{\underline k}(\xi)
\frac{\Psi(\xi)}
{(\xi_1-\xi_2)(\xi_2-\xi_3)}
\right]
\right|
\lesssim_\alpha
2^{(2{{q}}-2)\widetilde k_1}.
\label{eq:psi-divided-case33}
\end{align} 

Finally,
\[
\frac{\Psi(\xi)}{\Omega^{(\sigma)}(\xi)}
=
-
\frac{\Psi(\xi)}
     {(\xi_1-\xi_2)(\xi_2-\xi_3)}
G_\sigma(\xi)^{-1}.
\]
Combining \eqref{eq:G-inverse-case33} and
\eqref{eq:psi-divided-case33}, and applying Leibniz's rule, we obtain
\begin{align*}
\left|
\xi^\alpha\partial_\xi^\alpha
\left[
\varphi_{\underline{k}}(\xi)
\frac{\Psi(\xi)}{\Omega^{(\sigma)}(\xi)}
\right]
\right|
\lesssim_\alpha
2^{(2{{q}}-\sigma)\widetilde{k}_1}.\notag
\end{align*}

In summary, under all circumstances (case 3.1-3.3), by Lemma~\ref{lem:sinfty}, we obtain 

\begin{align}
  \| B_{\xi_1,\xi_2,\xi_3,\xi_4}\frac{\tilde{\Psi}(\xi)}{\Omega^{(\sigma)}(\xi)}\varphi_{\underline k}(\xi)\|_{\tilde{S}^\infty}\lesssim 2^{(2{{q}}-\sigma)\tilde{k}_1}2^{\beta(\tilde{k}_1+\tilde{k}_2+\tilde{k}_3+\tilde{k}_4)}\lesssim2^{(2{{q}}-\sigma+2\beta)\tilde{k}_1}2^{\beta(\tilde{k}_3+\tilde{k}_4)}.\label{bd:9.m3}
\end{align} 
Here $\tilde{\Psi}(\xi)$ denotes either $\Psi(\xi)$ or $\Psi(\xi)-\Psi(\xi^*)$.   To estimate \eqref{bd:q2}, by the same decomposition as in \eqref{bd:exp:Q},  it suffices to bound
\begin{align*}
 | Q_2^{(\sigma)}(P_{k_1}u,P_{k_2}\overline u,P_{k_3}u,P_{k_4}\overline u)(t_i)|
 &+  \bigg{|}\int_{t_1}^{t_2}Q_2^{(\sigma)}  (P_{k_1}\mathscr{B}(u,u,u),P_{k_2}\overline u,P_{k_3}u,P_{k_4}\overline u)(s) \dif s\bigg{|}.
\end{align*}
Under the condition $\beta<\sigma/4$, 
these terms can be estimated exactly as in Case 2, 
with the bound \eqref{bd:9.m3} replacing the corresponding symbol estimate there. 
This completes the proof.

\section{Proof of the Main Result: Theorem \ref{thm1}}\label{sec:proofthm1}
This section is devoted to the proof of the main result, which consists of two parts.
The first part concerns the long-time well-posedness of equation \eqref{MMT}.
We have already established the well-posedness of the truncated equation
\eqref{TrMMT} with high probability.
The main new ingredient is a deterministic high-order energy inequality
adapted to the quasilinear structure of \eqref{MMT},
which allows us to control the high-frequency energy growth
despite the derivative loss in the nonlinearity.
Combining this estimate with the bounds for the truncated equation,
we close a bootstrap argument and extend the solution of
\eqref{MMT} up to time $T$.

\subsection{The well-posedness of \eqref{MMT}} 
In this section, the initial cutoff exponent $\eta$ will be fixed sufficiently small
after the expansion order, which is independent of $L$.

   First, we introduce the following hyper-contractivity  estimate, the proof of which can be found in \cite[Proposition 2.2]{DIP25}.
\begin{proposition}[Multilinear Gaussian estimates] \label{prop:dapiancha}
Assume that $g_k : \Omega \to \mathbb{C}, k \in \{1,...,n\}$,
 are i.i.d. normalized complex Gaussian random variables on a probability space $\Omega$ and define
 \begin{align*}
 G(\omega) :=\sum_{k_1,...,k_r\in\{1,...,n\}}
 a_{k_1,...,k_r} \prod _{i=1}^rg^{\zeta_i}_{k_i}(\omega),
 \end{align*}
 where $r \geq 1, \omega \in\Omega, \zeta_j \in \{+,-\}$, and $ a_{k_1,...,k_r}\in\mathbb{C}$. Then, for any $A \geq 1$, it holds that
 \begin{align*}
     \mathbf{P}(\omega: |G(\omega)| \geq A\|G(\omega)\|_{L^2_{\omega}})
 \leq  Ce^{-A^{\frac2r}/C},
 \end{align*}
 for some constant $C \geq 1$.
\end{proposition}

By the decay property of Schwartz functions and the definition  \eqref{wellprepared}, we have
 
\begin{align*}
   \mathbf{E}\|u_{\mathrm{in}} \|^2_{H^{12}}\lesssim L^{-1}\sum_{k\in\mathbb{Z}_L} \langle k\rangle^{24} n_{\mathrm{in}}(k) \lesssim 1. 
\end{align*}
Then, by \cite[Proposition 2.2]{DIP25}, it follows that with probability $\geq 1-L^{-40}$,
\begin{align} 
   \|u_{\mathrm{in}} \|_{H^{12}} \lesssim L^\delta.\label{uinHn0}
\end{align}

Let $c^{\leq N}_k$ be the finite expansion coefficient   constructed in \eqref{def:ck} for the initial covariance
${n_{{\rm in},\eta}}$.
Define
\begin{align*}
 V(T\tau,x)&:=\frac1L\sum_{k\in\mathbb Z_L}
c^{\leq N}_k(\tau)e^{2\pi i[|k|^\sigma T\tau+\Gamma(k,\tau)]}
 e^{2\pi ikx}. \notag
\end{align*}
A tree of order at most $N$ has at most $2N+1$ leaves, each of
frequency at most $2D_0$. Hence its output has frequency at most
$D=2(2N+1)D_0$. Since $\eta\ll\delta$, for large $L$ we have
$3D<2^K\lesssim L^{\delta/10}$, and
\begin{align*}\notag
    &\operatorname{supp}\widehat V(t)\subset\{|k|\leq D\}. 
\end{align*}
Furthermore, 
$P_{\leq K}$ acts as the identity on 
$\mathscr B(V,V,V)$, and $V$ satisfies the following equation:
\begin{align*}
    i\partial_tV+2\pi|\nabla|^\sigma V
   &=\mathscr B(V,V,V)+r,\notag
\end{align*}
with 
\begin{equation*}
 \widehat r(T\tau,k)=-\frac{i}{T}
 e^{2\pi i[|k|^\sigma T\tau+\Gamma(k,\tau)]}
 \sum_{\substack{0\leq i,j,m\leq N\\i+j+m\geq N}}
 \mathcal W(c^i,\overline{c^j},c^m)_k(\tau).\notag
\end{equation*}

\begin{lemma}\label{lem:Vestimate}
Let $\sigma\in(0,1)\cup(1,2]$, $0<\beta<\sigma/4$, and let $T$
be in \eqref{def:T}. For every fixed $Q>0$, we can choose the expansion
order $N$ sufficiently large and $0<\eta<\delta/{96}$,
independently of $L$, such that with probability  $\geq 1-L^{-50}$ and   all sufficiently large $L$,
\begin{align}
 & 
\sup_{s\in[0,{T}]}\|V(s)\|_{H^{12}}\leq L^{\delta/2},\qquad 
 \sup_{s\in[0,T]}\|V(s)\|_{W^{2,0}}\leq L^{-1/2+\delta/4},
 \label{bd:vh12+vw20}\\
 &\|u_{\rm in}-V(0)\|_{H^{{12}}}
 +\sup_{t\in[0,T]}\bigl(\|V{(t)}-u_{\rm tr}(t)\|_{H^{{12}}}+\|r(t)\|_{H^{{12}}}\bigr)
 \leq L^{-Q}.\label{bd:uin-V}
\end{align}
  
\end{lemma}

\begin{proof}
First, for the remainder term $r$, by  Proposition~\ref{prop:ckn} and the definition  \eqref{Duhamel0}, we have
\[
 \sup_{t\in [0,T]}\|r(t)\|_{H^{{12}}}
 \leq C_NL^{C-\delta N}\leq L^{-Q}.
\]
Here $C$ is a universal constant independent of $N$.  
Increasing $N$ we obtain the desired bound.

By the same calculation as in the proof of Proposition \ref{prop:rnsmall}, by choosing $N$ large enough independent of $L$, we obtain similarly to \eqref{bd:mathcalr} that  
\begin{equation*}
 \sup_{t\in [0,T]}\|u_{\rm tr}(t)-V(t)\|_{H^{{12}}}
 \lesssim L^{{7\delta}}\|\mathcal R^{N+1}\|_Z\leq L^{-Q}.\notag
\end{equation*}
Here the $Q$ is given in advance independent of $L$, and we used the fact that $\supp \widehat{\mathcal R^{N+1}}\subset \{|k|\lesssim L^{\delta}\}$.

We prove the $W^{2,0}$ bound now. By the estimates on tree couples in \eqref{bd:couple}, we obtain
\begin{align*}
    \mathbf{E}|\langle \nabla \rangle^4V(t,x)|^2\lesssim \frac{1}{L^{2}}\sum_{k \in \mathbb{Z}_L}\sum_{i,j=0}^N   \langle k \rangle^8  \mathbf{E}[c^{i}_k\overline c^{j}_k]  \lesssim  L^{-1}.\notag
\end{align*}
Applying Proposition~\ref{prop:dapiancha}, with fixed moments large
enough with a  probability 
$\geq 1-L^{-60}$, we obtain
\begin{align*}
   \sup_{t\in[0,T]}|V(t)\|_{W^{2,0}}\lesssim L^{-1/2+\delta/4}.
\end{align*} 
  Also, we have
\[
 \sup_{t\in[0,T]} \|V(t)\|_{H^{{12}}}\leq D ^{{12}} \sup_{t\in[0,T]}\|V(t)\|_2
 \leq C_N L^{{12}\eta+1/2} \sup_{t\in[0,T]}\|V(t)\|_\infty
 \leq C_NL^{{12}\eta+\delta/4}\leq L^{\delta/2},
\]
after choosing ${12}\eta<\delta/8$.

At time zero, $\widehat V(0,k)=\psi(k/D_0)\widehat u_{\rm in}(k)$. Then, the Schwartz decay shows that 
\begin{equation*}
 \mathbf E\|u_{\rm in}-V(0)\|_{H^{{12}}}^2
 =L^{-1}\sum_k\langle k\rangle^{{24}}
 |1-\psi(k/D_0)|^2n_{\rm in}(k)\lesssim L^{-2Q-1}.\notag
\end{equation*} 
By applying Proposition \ref{prop:dapiancha} we finish the proof.
\end{proof}

\begin{lemma}\label{lem:energyhigh}
Let $\sigma\in(0,1)\cup(1,2]$, $0<\beta<\sigma/4$.
Assume that   $w$ is a solution to  the equation
\begin{align}\label{eq:w}
 i\partial_tw+2\pi|\nabla|^\sigma w=\mathcal{L}\mathscr B(V)w
   &+f, 
\end{align}
where $  \mathcal{L}\mathscr B(V)w:= A_Vw+ C_V\overline w$, 
\begin{align*}
 A_Vw&:=-2\lambda^2|\nabla|^\beta
       \bigl(|\,|\nabla|^\beta V|^2|\nabla|^\beta w\bigr),\\
 C_V\overline w&:=-\lambda^2|\nabla|^\beta
       \bigl((|\nabla|^\beta V)^2|\nabla|^\beta\overline w\bigr).
\end{align*}
Let $\epsilon =\sup_{s\in[0,t]}\|V(s)\|_{W^{2,0}}$ with $\lambda^2\epsilon^2\ll1$, and $m(t):=\sup_{s\in[0,t]}\|w(s)\|_2$. Then, it holds that for any $t_1<t_2\in[0,t]$
\begin{align*}
 \|w(t_2)\|_2^2
 &\lesssim \|w(t_1)\|_2^2+\int_{t_1}^{t_2}\lambda^2\epsilon^2 D^{2\beta}\|\Pi_Dw(s)\|_2^2
 +\lambda^4\epsilon^4 m(s)^2+\lambda^2\epsilon m(s)^2 \|r(s)\|_{H^8}+  m(s)\|f(s)\|_{L^2} \dif s.\notag
\end{align*}
Here $\Pi_D$ is the Fourier projection onto
${|k|\leq 200 D}$.
All constants are independent of $L$ and $D$.

\end{lemma}

\begin{proof}
We take the energy estimate to $w$ and obtain 
\begin{align*}
 \frac{\dif}{\dif t}\|w\|_2^2
 &= 2\Im\langle A_Vw+C_V\overline w,w\rangle
   +2\Im\langle f,w\rangle\\
 &=-\frac{2\lambda^2}{L^3}\Im
   \sum_{\xi_1-\xi_2+\xi_3-\xi_4=0}
   \widehat V_{\xi_1}\overline{\widehat w}_{\xi_2}
   \widehat V_{\xi_3}\overline{\widehat w}_{\xi_4}
   B_{\xi_1,\xi_2,\xi_3,\xi_4}
   +2\Im\langle f,w\rangle,
\end{align*}
where the term $ \langle A_Vw,w\rangle$ has no contribution. The last term is bounded by 
\begin{align*}
    C\|f\|_{L^2}\|w\|_{L^2}.
\end{align*}

We now consider the first term. We recall resonant relation as in \eqref{def:psiphi}:
\[
 \Omega^{(\sigma)}(\xi)
 =|\xi_1|^\sigma-|\xi_2|^\sigma
   +|\xi_3|^\sigma-|\xi_4|^\sigma.
\]
Let $\chi_D(\xi_2,\xi_4)$ be a cutoff satisfying  $\chi_D(\xi_2,\xi_4)=1$ on $|\xi_2|+|\xi_4|\geq  200 D$, and $\chi_D(\xi_2,\xi_4)=0$ on $|\xi_2|+|\xi_4|\leq 100 D$. Then, on $\supp \chi_D$, we know that by the construction of $V$,  
\[
 |\xi_2+\xi_4|
 =|\xi_1+\xi_3|\leq4 D.
\]
So, we have  $ |\xi_2|\sim|\xi_4| \gtrsim 100D$, and then 
\begin{equation*} \notag
 -\Omega^{(\sigma)}(\xi)\gtrsim |\xi_4|^\sigma>0.
\end{equation*} 
Then, we consider a similar transform as in \eqref{nomoralform}. For simplicity, we write $\chi_D (\xi)$ be the product of $\chi_D(\xi_2,\xi_4)$ and another cutoff with support on the domain $|\xi_1|,|\xi_3|\leq3D$. Then, we define  
\begin{align*}
 &Q_D (u_1,\overline{u_2},u_3,\overline{u_4}) :=
 \frac1{L^3}
 \sum_{\xi_1-\xi_2+\xi_3-\xi_4=0}
 \widehat{u_1}_{\xi_1}\overline{\widehat{u_2}}_{\xi_2}
 \widehat{u_3}_{\xi_3}\overline{\widehat{u_4}}_{\xi_4}
 B_{\xi_1,\xi_2,\xi_3,\xi_4}
 \chi_D(\xi), \\
 &Q_D^{(\sigma)}(u_1,\overline{u_2},u_3,\overline{u_4}) :=
 \frac1{L^3}
 \sum_{\xi_1-\xi_2+\xi_3-\xi_4=0}
 \widehat{u_1}_{\xi_1}\overline{\widehat{u_2}}_{\xi_2}
 \widehat{u_3}_{\xi_3}\overline{\widehat{u_4}}_{\xi_4}
 B_{\xi_1,\xi_2,\xi_3,\xi_4}
 \frac{\chi_D(\xi) }{\Omega^{(\sigma)}(\xi)}.
\end{align*}
We define the low frequency part $Q_{D,l} $ similar to $Q_D$ with $\chi_D(\xi)$ replaced by $(1-\chi_D(\xi))$.

By a similar calculation as in \eqref{bd:exp:Q}, we need to estimate
\begin{align*}
\lambda^2\left|\int_{t_1}^{t_2}Q_D(P_{k_1}V,P_{k_2}\overline w,P_{k_3}V,P_{k_4}\overline w)(s)
 \,\dif s\right|.
\end{align*}
By symmetry, 
it suffices to estimate 
\begin{align*}
  \lambda^2 \left| 
 Q_D^{(\sigma)}(P_{k_1}V,P_{k_2}\overline w,P_{k_3}V,P_{k_4}\overline w)(t_i)\right| \notag
\end{align*}
and
\begin{align}
\lambda^2&\bigg|\int_{t_1}^{t_2}
 Q_D^{(\sigma)}(P_{k_1}(\mathscr B(V,V,V)+r),
                 P_{k_2}\overline w,P_{k_3}V,P_{k_4}\overline w)
 \,\dif s\notag\\
 &\qquad-\int_{t_1}^{t_2}
 Q_D^{(\sigma)}(P_{k_1}V,
 P_{k_2}({\overline{f{+}A_Vw{+}C_V\overline w}}),
 P_{k_3}V,P_{k_4}\overline w)\,\dif s\bigg|.\label{qdsigma2}
\end{align}

    First, we estimate the boundary term. We define \begin{align*}
        A^0_k(s):=\sum_{l\in\mathbb{Z}}2^{-\epsilon'|l-k| }\|P_lw(s)\|_{L^2}.
    \end{align*} Recall that $ 2^{k_2}\sim2^{k_4} \gtrsim 100D$, then it holds that \begin{equation*} \notag
 \left\|
 B_{\xi_1,\xi_2,\xi_3,\xi_4}
 \frac{\chi_D}{\Omega^{(\sigma)}(\xi)}
 \varphi_{\underline k}(\xi)
 \right\|_{\widetilde S^\infty}
 \lesssim 
 2^{\beta(k_1+k_3)}(2^{k_4})^{2\beta-\sigma}.
\end{equation*}
By the same calculation as in proof of  Proposition \ref{prop:highenergy}, we obtain
\begin{align*}
&2^{\epsilon'
   (\max\{\widetilde k_3,0\}
    -\min\{\widetilde k_4,0\})}
 \left| 
 Q_D^{(\sigma)}(P_{k_1}V,P_{k_2}\overline w,P_{k_3}V,P_{k_4}\overline w)(s)\right| 
\lesssim
A_{\widetilde k_1}^0(s)
A_{\widetilde k_2}^0(s)
 D^{2\beta-\sigma}\epsilon^2,\notag
\end{align*} when summing over $k_i$, its contribution is 
 \begin{align*}
    \lambda^2  D^{2\beta-\sigma}\epsilon^2  \|w\|_{L^2}^2\leq \lambda^2\epsilon^2  \|w\|_{L^2}^2.\notag
 \end{align*}

By a similar calculation as \eqref{bd:pkbu}, it holds that 
\begin{align*} 
2^{-\beta\max\{k,0\}} \|P_k\mathscr B(V,V,V)\|_{L^\infty}
 \lesssim\lambda^2 \epsilon^3,\notag\\
 2^{-2\beta\max\{k,0\}} \|P_k[A_Vw+C_V\overline w]\|_{L^2}
 \lesssim\lambda^2 A_k^0\epsilon^2.\notag
\end{align*} 
Then after summing over $k_1$, under the condition of $4\beta<\sigma$, the contribution of \eqref{qdsigma2} is 
\begin{align*}
    &\lambda^2\epsilon \|w\|_{L^2}^2( \lambda^2 \epsilon^3+\|r\|_{H^8})+\lambda^2 \epsilon^2 \|w\|_{L^2}(\|f\|_{L^2}+\lambda^2 \epsilon^2 \|w\|_{L^2})\\
    &\lesssim \lambda^4\epsilon^4 \|w\|_{L^2}^2+\lambda^2\epsilon  \|w\|_{L^2} ^2\|r\|_{H^8}+\lambda^2 \epsilon^2 \|w\|_{L^2}\|f\|_{L^2}.
\end{align*}

For the low-frequency term, both $w$ frequencies are bounded
by ${200D}$.  The multiplier $1-\chi_D$ has
uniformly bounded kernel $L^1$ norm after rescaling by $D$.
Hence, its contribution is 
\[
 |{Q_{D,l}}(V,\overline w,V,\overline w)|
 \lesssim D^{2\beta}\epsilon^2\|\Pi_Dw\|_2^2.
\] 
Adding up all the results above, we obtain the desired estimate.
\end{proof}

 \begin{lemma}\label{lem:energtlow}
 Under the condition of Lemma \ref{lem:energyhigh},  by taking the fixed expansion order
$N$ sufficiently large,  we have
\begin{align*}
 \sup_{t'\in[0,t]}\|\Pi_Dw(t')\|_2
 &\leq C_NL^{60}\left(\|w(0)\|_{H^5}
       +T\sup_{t'\in[0,t]}\|f(t')\|_{H^5}\right)
       +L^{-5}\sup_{t'\in[0,t]}\|w(t')\|_2.\notag
\end{align*} 
Then, it holds that  \begin{align*}
    m(t)^2&\lesssim   L^{125}\|w(0)\|_{H^5}^2
       +L^{125}\sup_{t'\in[0,t]}\|f(t')\|_{H^5}^2.
\end{align*}

\end{lemma}

\begin{proof}  We define
\[
 d_k(\tau)
 :=e^{-i2\pi (|k|^\sigma T\tau+\Gamma(k,\tau))}\widehat w(T\tau,k),
 \qquad 0\leq\tau\leq t/T.
\]
In the Fourier space, the equation \eqref{eq:w}  reduces to  
\[
 d=g+\mathscr Ld+H,
\]
where  the operator $\mathscr L$ is defined in \eqref{def:l}, and
\begin{align*}
 g_k(\tau)
 &:=\widehat w_k(0)
   -iT\int_0^\tau e^{-i2\pi (|k|^\sigma Ts+\Gamma(k,s))}
                    \widehat f(Ts,k)\,\dif s,\\
 H_k(\tau)
 &:=-iT\int_0^\tau e^{-i2\pi (|k|^\sigma Ts+\Gamma(k,s))}
       \widehat{(1-P_{\leq K})\mathcal{L}\mathscr B(V)w}(Ts,k)\,\dif s.
\end{align*}
For the initial part and the force term, we have  
\begin{align*}
 \|g\|_{Z}
 &\leq \|w(0)\|_{H^5}
       +T\sup_{0\leq s\leq t}\|f(s)\|_{H^5}.
\end{align*}

For the high frequency part $H$, by definition, it holds that  $H_k(\tau)=0$ for $|k|<2^{\delta/10}$.  Then, we take the expansion up to order   $N$ with $ \delta N >120$ independent of $L$. It holds that  for any $0\leq j<N$, \[
 \Pi_D\mathscr L^jH=0.
\]
Indeed, we recall the definition of $\mathscr{L}$ in \eqref{def:l}, since the frequency of $c^{\leq N}$ is no more than $D$, the change of the frequency after acting $\mathscr{L}$ is no more than $2D$.  Then the change of the frequency after acting $\Pi_D\mathscr L^j$ is no more than $O((1+2j)D)=O(N^2L^\eta)=o_{N}(L^{\delta/10})$.  Together with the fact that $H_k(\tau)=0$ for $|k|<L^{\delta/10}$, we obtain the desired result. 

By iterating $d=g+\mathscr Ld+H$ by $N$ times, we have
\[
 d=\sum_{j=0}^{N-1}\mathscr L^jg
   +\sum_{j=0}^{N-1}\mathscr L^jH
   +\mathscr L^Nd,
\]and then
\begin{equation*}
 \Pi_Dd
 =\sum_{j=0}^{N-1}\Pi_D\mathscr L^jg
  +\Pi_D\mathscr L^Nd. \notag
\end{equation*}

The right-hand side is bounded using Proposition~\ref{prop:l}.
By finite frequency propagation, the last term depends only on
$1_{\{|k|\leq C(1+N)D\}}d$. Then, we have 
\[
 \bigl\|1_{\{|k|\leq C(1+N)D\}}d\bigr\|_Z
 \lesssim  D ^5
             \sup_{t'\in[0,t]}\|w(t')\|_2.
\]
Consequently,
\begin{align*}
 \sup_{t'\in[0,t]}\|\Pi_Dw(t')\|_2
 &\lesssim C_NL^{60}\|g\|_Z
   +C_NL^{50-\delta N/2}
      D^5
             \sup_{t'\in[0,t]}\|w(t')\|_2\\
 &\leq C_NL^{60}\|g\|_Z
       +L^{-5}\sup_{t'\in[0,t]}\|w(t')\|_2.
\end{align*}

Then, by Lemma \ref{lem:Vestimate} and Lemma \ref{lem:energyhigh}, it holds that 
\begin{align*}
   m(t)^2&\lesssim \|w(0)\|_2^2+\alpha TL^{2\delta} (L^{120}\|w(0)\|_{H^5}^2
       +T^2L^{120}\sup_{t'\in[0,t]}\|f(t')\|_{H^5}^2+L^{-10} m(t)^2)\\
       &+\alpha^2TL^\delta  m(t)^2+\alpha TL^\delta  L^{-10} m(t)^2 +T  m(t) \sup_{t'\in[0,t]}\|f(t')\|_2,
\end{align*}
which implies the desire bound.
\end{proof}
\begin{proposition}[Main bootstrap argument]\label{prop:bootstrap}
Let $\sigma\in(0,1)\cup(1,2]$ and $0<\beta<\sigma/4$.
With probability  $\geq 1-L^{-40}$, the following holds.
Suppose $u$ is a smooth solution to \eqref{MMT} on $[0,t]$,
where $t\leq T$, and
\[
\sup_{s\in[0,t]}\|u(s)\|_{H^{12}} \leq L^{40},\qquad  \sup_{s\in[0,t]}\|u(s)\|_{H^{10}} \leq L^{4\delta},\qquad  \sup_{s\in[0,t]}\|u(s)\|_{W^{2,0}}\leq L^{-1/2+4\delta}.
\] 
Then we have
\begin{equation*} \notag
\sup_{s\in[0,t]}\|u(s)\|_{H^{12}} \leq L^{30},\qquad   \sup_{s\in[0,t]}\|u(s)\|_{H^{10}} \leq L^{ \delta},\qquad \sup_{s\in[0,t]}\|u(s)\|_{W^{2,0}}\leq L^{-1/2+\delta}.
\end{equation*}
\end{proposition}

\begin{proof}  

  By conservation of mass and
Sobolev interpolation we have
\[
 \|u(s)\|_{H^{12-(1-2\beta)/2}}^2
 \leq\|u(0)\|_2^{(1-2\beta)/12}
       \|u(s)\|_{H^{12}}^{2-(1-2\beta)/12}.
\]
Under the assumed bound $A:= \sup_{[0,t]}\|u\|_{H^{12}}\leq L^{40}<\infty$, $\|u\|_{W^{2,0}}\leq L^{-1/2+4\delta}$,
Theorem~\ref{thm:energy inequality} therefore yields
\begin{align*}
 \sup_{[0,t]}\|u\|_{H^{12}}^2
 \lesssim{}&\|u(0)\|_{H^{12}}^2
 +\alpha^2TL^{16\delta}\sup_{[0,t]}\|u\|_{H^{12}}^2\\
 &+1_{\{\sigma<1\}}\alpha TL^{8\delta}\|u(0)\|_{L^2}^{(1-2\beta)/12}
       \sup_{[0,t]}\|u\|_{H^{12}}^{2-(1-2\beta)/12}.
\end{align*}
Since $\alpha^2TL^{16\delta}\ll 1$,
 by Young's inequality and \eqref{uinHn0}, we  imply that  
\begin{equation*} 
 \sup_{[0,t]}\|u\|_{H^{12}}^2
 \lesssim\|u_{\rm in}\|_{H^{12}}^2
 +{1_{\{\sigma<1\}}}\|u_{\rm in}\|_2^2
      (T\alpha L^{8\delta})^{24/(1-2\beta)}\leq L^{30},\notag
\end{equation*}  
where in  the last inequality, we chose $\gamma>0$ small enough.

Define $w:=u-V$. It is easy to check that $w$ solves the following equation:  
\begin{equation*}\notag
 i\partial_tw+2\pi|\nabla|^\sigma w
 =\mathcal{L}\mathscr B(V)w+\mathscr Q_V(w)-r,
\end{equation*}
where
\begin{align*}
 \mathscr Q_V(w)
 &:=\mathscr B(V+w,V+w,V+w)-\mathscr B(V,V,V)
       -\mathcal{L}\mathscr B(V)w.
\end{align*} 
By   the bounds
on $V$ in \eqref{bd:vh12+vw20} and conservation of mass,  we have 
\begin{align*}
    \|w(s)\|_{L^2}\leq \|u(s)\|_{L^2}+\|V(s)\|_{L^2}\leq L^\delta.
\end{align*}
Together with  interpolation, it holds that
\begin{equation}\label{eq:small-difference-interpolation}
 \sup_{s\in[0,t]}\|w(s)\|_{H^{10}}
 \lesssim L^{50} m(t)^{1/6},\qquad \sup_{s\in[0,t]}\|w(s)\|_{H^{5+2\beta}}
 \lesssim L^{50} m(t)^{(7-2\beta)/12}.
\end{equation} 
Then, we have   
\begin{align*} \notag
\sup_{s\in[0,t]} \|\mathscr Q_V({w})\|_{H^5}
& \lesssim\lambda^2\bigl(\sup_{s\in[0,t]}\|V\|_{H^{5+2\beta}}
                        +\sup_{s\in[0,t]}\|u\|_{H^{5+2\beta}}\bigr)
                      \sup_{s\in[0,t]}\|w\|_{H^{5+2\beta}}^2\notag\\
                    &  \lesssim L^{150} m(t)^{1+(1-2\beta)/6}\notag.
\end{align*} 
By Lemma \ref{lem:energtlow}, we then have
 \begin{align*}
m(t)^2&\lesssim   L^{125-2Q} 
       +L^{450} m(t)^{2+(1-2\beta)/3}.
\end{align*}
Under the additional bootstrap $m(t)\leq L^{-Q/4}$,
this inequality improves this to
$m(t)\leq\tfrac12L^{-Q/4}$ by choosing $Q$ large enough.
Since by \eqref{bd:uin-V}, $m(0)\leq L^{-Q}$, continuity closes this bootstrap.
It then follows by \ref{eq:small-difference-interpolation} that
\begin{equation}\label{bd:u-VH10}
 \sup_{s\in[0,t]}\|u(s)-V(s)\|_{H^{10}}\leq L^{-150}.
\end{equation} 

 Finally, together with \eqref{bd:vh12+vw20}, we have
\begin{align*}
 \sup_{[0,t]}\|u\|_{H^{10}}
 &\leq L^{\delta/2}+L^{-150}\leq L^\delta,\\
 \sup_{[0,t]}\|u\|_{W^{2,0}}
 &\leq L^{-1/2+\delta/4}+CL^{-150}
       \leq L^{-1/2+\delta},
\end{align*} 
\end{proof}

\subsection{The derivation of WKE}  
The main bootstrap argument in Proposition~\ref{prop:bootstrap}, the energy estimate in Theorem  \ref{thm:energy inequality}, together with Galerkin approximation yields
the existence of a smooth solution to \eqref{MMT} up to time $T$.
We now verify \eqref{app}.

By \eqref{bd:u-VH10},
\eqref{bd:uin-V}, and the conservation of mass,  with a high probability, 
 we have
\[
 \sup_{t\in[0,T]}\|u(t)-u_{\rm tr}(t)\|_{H^5}\leq L^{-100},
 \qquad \sup_{t\in[0,T]}(\|u(t)\|_2+\|u_{\rm tr}(t)\|_2)\leq L^2.
\]
The Fourier normalization gives
$|\widehat f(k)|\leq L^{1/2}\|f\|_2$, and hence
\[
 \left|\mathbf E|\widehat u(t,k)|^2
       -\mathbf E|\widehat u_{\rm tr}(t,k)|^2\right|
 \lesssim L^{-90}.
\]  

Then, by Theorem \ref{thm2} we have
\[
 \mathbf E|\widehat u(t,k)|^2
 =n_{{\rm in},\eta}(k)
  +\frac{t}{T_{\rm kin}}\mathcal K(n_{{\rm in},\eta})(k)
  +o_{l^\infty_k}\left(\frac{t}{T_{\rm kin}}\right).
\]
Since $n_{\rm in}$ is Schwartz and $\eta>0$ is fixed,  as $L\to \infty$, we could replace $n_{{\rm in},\eta}$ by  $n_{\rm in}$. The proof is finished.

\medskip

\medskip

\medskip
\ 
\noindent{\bf Acknowledgment.}  
The author would like to thank Prof. Yu Deng for valuable discussions and
helpful suggestions. His insights and comments have been very helpful in
the development of this work.

\
\noindent{\bf Conflict of interest.}
The author declares that he has no conflict of interest.

\
\noindent{\bf Data availability.}
No datasets were generated or analyzed during the current study.

\bigskip

 \end{document}